\documentclass{article}

\usepackage[final]{neurips_2021}

\usepackage[utf8]{inputenc} 
\usepackage[T1]{fontenc}    
\usepackage{hyperref}       
\usepackage{url}            
\usepackage{booktabs}       
\usepackage{amsmath, amsthm, amssymb, amsfonts}      
\usepackage{nicefrac}       
\usepackage{microtype}      
\usepackage{xcolor}         
\usepackage{enumitem}
\usepackage{graphicx}
\usepackage{bbm}
\usepackage{subcaption}
\usepackage{xr}
\usepackage{algorithm}
\usepackage[noend]{algpseudocode}

\usepackage[numbers]{natbib}

\providecommand{\opmyspan}{\operatorname*{Span}}
\providecommand{\opmyreshape}{\operatorname*{Reshape}}

\providecommand{\Dim}{\operatorname{dim}}            
\providecommand{\dim}{\Dim}

\providecommand{\opvec}{\operatorname {vec}}
\providecommand*{\Span}[1]{\operatorname{span}\{{#1}\}} 
\providecommand{\sym}{\operatorname{Sym}}
\providecommand{\colspan}{\operatorname{colspan}}
\providecommand{\Sym}{\sym}
\providecommand{\symten}[1]{\sym(\CT_D^{#1})}

\renewcommand{\Im}{\operatorname{Im}}             
\providecommand{\argmax}{\operatorname*{argmax}}  
\providecommand{\Id}{\Op{Id}}

\providecommand{\diag}{\operatorname{diag}}

\providecommand{\CA}{{\cal A}}

\providecommand{\CI}{{\cal I}}

\providecommand{\CN}{{\cal N}}
\providecommand{\CO}{{\cal O}}

\providecommand{\CT}{{\cal T}}

\providecommand{\bbE}{\mathbb{E}}

\providecommand{\bbN}{\mathbb{N}}

\providecommand{\bbP}{\mathbb{P}}

\providecommand{\bbR}{\mathbb{R}}
\providecommand{\bbS}{\mathbb{S}}

\providecommand*{\textN}[1]{\|{#1}\|} 

\providecommand*{\N}[1]{\left\|{#1}\right\|} 
\newcommand*{\SN}[1]{\left|{#1}\right|}      

\newcommand*{\Op}[1]{\mathsf{#1}} 

\newcommand{\mathscalebox}[2]{%
	\text{\scalebox{#1}{$\displaystyle #2$}}}

\theoremstyle{plain}
\newtheorem{theorem}{Theorem}
\newtheorem{proposition}[theorem]{Proposition}
\newtheorem{lemma}[theorem]{Lemma}
\newtheorem{corollary}[theorem]{Corollary}

\newtheorem{conjecture}[theorem]{Conjecture}
\algrenewcommand\algorithmicrequire{\textbf{Input:}}     
\algrenewcommand\algorithmicensure{\textbf{Returns:}}     
\newtheorem{def-prop}[theorem]{Definition/Proposition}

\theoremstyle{definition}
\newtheorem{definition}[theorem]{Definition}

\newtheorem{remark}[theorem]{Remark}

\title{Landscape analysis of an improved \\ power method for tensor decomposition}

\author{
  Joe Kileel \\
   UT Austin \\
 \And
  Timo Klock \\
  Simula Research Laboratory\\
  \And
  João M. Pereira \\
  UT Austin \\
}

\begin{document}

\maketitle

\begin{abstract}
In this work, we consider the optimization formulation for  symmetric tensor decomposition recently introduced in the Subspace Power Method (SPM) of Kileel and Pereira.
Unlike popular alternative functionals for tensor decomposition, the SPM objective function has the desirable properties that its maximal value is known in advance, and its global optima are exactly the rank-1 components of the tensor when the input is sufficiently low-rank.
We analyze the non-convex optimization landscape associated with the SPM objective.
Our analysis accounts for working with noisy tensors.
We derive quantitative bounds such that any second-order critical point with SPM objective value exceeding the bound must equal a tensor component in the noiseless case, and must approximate a tensor component in the noisy case.
For decomposing tensors of size $D^{\times m}$, we obtain a near-global guarantee up to rank $\widetilde{o}(D^{\lfloor m/2 \rfloor})$ under a random tensor model, and a global guarantee up to rank $\mathcal{O}(D)$ assuming deterministic frame conditions.
This implies that SPM with suitable initialization is a provable, efficient, robust algorithm for low-rank symmetric tensor decomposition.
We conclude with numerics that show a practical preferability for using the SPM functional over a more established \nolinebreak counterpart.
\end{abstract}

\section{Introduction}

From applied and computational mathematics \cite{kolda2009tensor} to machine learning \cite{sidiropoulos2017tensor} and multivariate statistics \cite{mccullagh2018tensor} to many-body quantum systems \cite{white1992density}, high-dimensional data sets arise that are naturally organized into higher-order arrays.
Frequently these arrays, known as hypermatrices or \textit{tensors}, are decomposed into low-rank representations.
In particular, the real symmetric CANDECOMP/PARAFAC (CP) decomposition \cite{comon2008symmetric} is often appropriate:
\begin{equation} \label{eq:sym-ten-decomp}
T = \sum_{i=1}^K \lambda_i a_i^{\otimes m}.
\end{equation}
Here, we are given $T$, a real symmetric tensor  of size $D^{\times m}$.
The goal is to expand $T$ as a sum of $K$ rank-$1$ terms, coming from scalar/vector pairs $(\lambda_i, a_i) \in \mathbb{R} \times \mathbb{R}^D$.
Importantly, the number of terms $K$ must be minimal possible for the given tensor, in which case $K$ is called the \textit{rank} of the input $T$.

 When $m > 2$ and $K = \mathcal{O}(D^m)$, CP decompositions 
are generically unique (up to permutation and scaling) by fundamental results in algebraic geometry \cite{chiantini2002weakly}.
An actionable interpretation \cite{JMLR:v15:anandkumar14b} is that
\textup{CP decomposition infers well-defined \textit{latent variables} $\{(\lambda_i, a_i) : i \in [K]\}$ encoded by  $T$}.
Indeed in learning applications, where symmetric tensors are formed from statistical moments (higher-order covariances) or multivariate derivatives (higher-order Hessians), CP decomposition has enabled parameter estimation for mixtures of Gaussians \cite{ge2015learning,sherman2020estimating}, generalized linear models \cite{sedghi2016provable}, shallow neural networks \cite{fornasiershallow,janzamin2015beating,zhong2017recovery}, deeper networks \cite{fiedler2021stable,fornasier2019robust,oymak2021learning}, hidden Markov models  \cite{anandkumar2012method}, among others.

 Unfortunately, CP decomposition is NP-hard in the worst case \cite{hillar2013most}.
In fact, it is believed to possess a \textup{computational-to-statistical gap} \cite{bandeira2018notes}, with efficient algorithms expected to exist for random tensors only up to rank $K = \mathcal{O}(D^{m/2})$.
For this reason, and due to its  non-convexity and high-dimensionality, CP decomposition has served theoretically as a key testing ground for better understanding mysteries of non-convex optimization landscapes.
To date, this focus has been on the non-convex \nolinebreak program
\begin{equation}
    \max_{x \in \mathbb{R}^D : \| x \|_2 = 1} \left\langle T, x^{\otimes m} \right\rangle.
    \tag{\textbf{PM-P}}
    \label{prob:PM}
\end{equation}
We label the problem \ref{prob:PM}, standing for Power Method Program, because projected gradient ascent applied to \ref{prob:PM} corresponds to the Shifted Symmetric Higher-Order Power Method of Kolda and Mayo  \cite{kolda2014adaptive}.
Important analyses of \ref{prob:PM} include Ge and Ma's  \cite{ge2020optimization} and the earlier \cite{anandkumar2017analyzing}
on overcomplete tensors, as well as \cite{sanjabi2019does} for low-rank tensors, and
\cite{qu2019geometric} which studied \ref{prob:PM} assuming the tensor components form a unit-norm tight frame
with low incoherence.

\begin{figure}
    \centering
    \includegraphics[width=\linewidth]{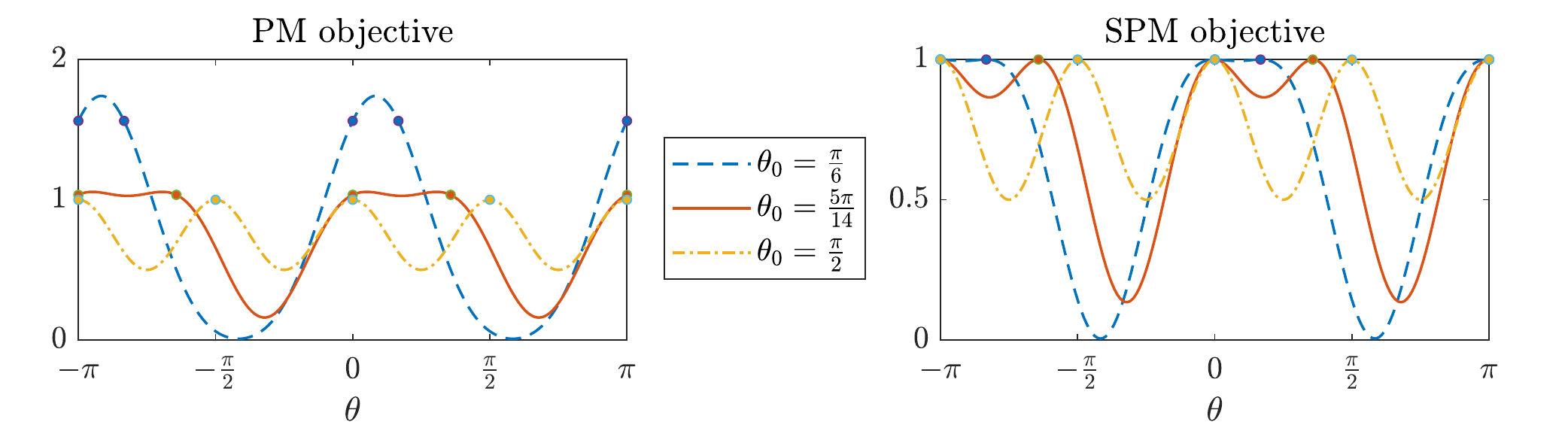}
    \caption{Illustration of \ref{prob:PM} objective (left) and \ref{prob:SPM} objective (right) for $D=K=2$ and $m=4$, where $\bbS^1$ is mapped to $(-\pi, \pi]$.
    We set $a_1 = (1,0)$, $a_2 = (\cos \theta_0,  \sin \theta_0)$ for different values of $\theta_0$.
    If $a_1$ and $a_2$ are not orthogonal ($\theta_0 \neq \frac{\pi}{2}$), then $a_1$ and $a_2$ do not coincide with global or local maximizers of \ref{prob:PM}.
    Too much correlation ($\theta_0 \le \frac{\pi}{3}$) leads to a collapsing of the two distinct maxima into one isolated maximum for \ref{prob:PM}. On the other hand, \ref{prob:SPM}
    is not affected in this way by correlation: local maxima coincide with global maxima which are the tensor components $a_1, a_2$. }
    \label{fig:spm_vs_pm_objective}
\end{figure}

In this paper, we perform an analysis of the non-convex optimization landscape associated with the Subspace Power Method (SPM) for computing symmetric tensor decompositions.
The first and third authors introduced SPM in \cite{kileel2019subspace}.
This method is based on the following non-convex program:
\begin{equation}
    \max_{x \in \mathbb{R}^D : \| x \|_2 = 1} F_{\mathcal{A}}(x),
    \tag{\textbf{SPM-P}}
    \label{prob:SPM}
\end{equation}
\vspace{-1em}
\begin{align*}
   &\textup{where } \, F_{\mathcal{A}}(x) := \| P_{\mathcal{A}}(x^{\otimes n}) \|_F^2, \,\,\,\, n := \lceil m/2 \rceil, \,\,\,\, \mathcal{A} := \textup{Span}\{a_1^{\otimes n}, \ldots, a_K^{\otimes n} \}, \\ &\textup{and } \, P_{\mathcal{A}} : (\mathbb{R}^{D})^{\otimes n} \rightarrow \mathcal{A} \textup{ is orthogonal projection with respect to Frobenius inner product}.
\end{align*}
Note that \ref{prob:SPM} is a particular polynomial optimization problem of degree $2n$ on the unit sphere.

There are at least two motivating reasons to study the optimization landscape of \ref{prob:SPM}.
Firstly, it was observed in numerical experiments in \cite{kileel2019subspace} that the SPM algorithm is competitive within its applicable rank range of $K = \mathcal{O}(D^{\lfloor m/2 \rfloor})$.
It gave a roughly one-order of magnitude speed-up over the decomposition methods in \cite{kolda2015numerical} as implemented in Tensorlab \cite{vervliet2016tensorlab}, while  matching the numerical stability of FOOBI \cite{foobi2007}.
Thus SPM is a \textit{practical} algorithm.
Secondly, from a theory standpoint, the program \ref{prob:SPM} has certain desirable properties which  \ref{prob:PM} lacks.
Specifically for an input tensor $T = \sum_{i=1}^{K} \lambda_i a_i^{\otimes m }$ with rank $K \lesssim D^{\lfloor m/2 \rfloor}$ and Zariski-generic\footnote{\textit{Zariski-generic} means that the failure set can be described by the vanishing non-zero polynomials \cite{harris2013algebraic}, so in particular, has Lebesgue measure $0$.} $\{(\lambda_i, a_i)\}_{i=1}^K$, \ref{prob:SPM} is such that:
\begin{itemize}
    \item Each component $\pm a_i$ is \textup{exactly} a global maximum, and there are \textup{no other} global maxima. 
    \item The globally maximal value is known in advance to be exactly $1$.  So the \textup{objective value} gives a certificate for global optimality, and non-global critical points can be discarded.
\end{itemize}
These properties were shown for \ref{prob:SPM} in \cite{kileel2019subspace}, but both fail for \ref{prob:PM} (see Figure~\ref{fig:spm_vs_pm_objective}).
Thus, \ref{prob:SPM} is more relevant  \textit{theoretically} than \ref{prob:PM} as a test problem for non-convex CP decomposition.

\textbf{Prior theory.} In \cite{kileel2019subspace}, it is proven that projected gradient ascent applied to \ref{prob:SPM}, initialized at almost all starting points with a constant explicitly-bounded step size, must converge to a second-order critical point of \ref{prob:SPM} at a power rate or faster.
However this left open the possible existence of \textit{spurious second-order critical points}, i.e., second-order points with reasonably high objective value that are not global maxima (unequal to, and possibly distant from, each CP component $\pm a_i$).
Such critical points could pose trouble for the successful optimization of \ref{prob:SPM}.
Furthermore all theory for \ref{prob:SPM} in \cite{kileel2019subspace} was restricted to the \textit{clean case:} that is, when the input tensor $T$ is \textup{exactly} of \nolinebreak a sufficiently low \nolinebreak rank \nolinebreak $K$.  The analysis for \ref{prob:PM} in \cite{ge2020optimization,qu2019geometric,sanjabi2019does} also assume noiseless inputs.
However, tensors arising in practice are invariably noisy due, e.g., to sampling or measurement errors.

\paragraph{Main contributions.}
We perform a landscape analysis of \ref{prob:SPM} by characterizing all second-order critical points,
using suitable assumptions on $a_1,\ldots,a_K$.
Under deterministic
frame conditions on $a_1, \ldots a_K$, which are satisfied by mutually incoherent vectors,
near-orthonormal systems, and random ensembles of size $K = \mathcal{O}(D)$,
Theorem \ref{thm:main_result_deterministic} shows that all second-order critical points of \ref{prob:SPM} coincide exactly with $\pm a_i$.
Theorem \ref{thm:main_result_for_now} shows the same result for overcomplete random ensembles of size $K = \tilde o(D^{\lfloor m/2\rfloor})$,
however requiring an additional superlevel set condition to exclude maximizers with vanishing objective values.
Both results
extend to noisy tensor decomposition, where SPM is applied to a perturbation $\hat T \approx T$. In this setting,
second-order critical points with objective values exceeding $\CO(\|\hat T-T\|_F)$ are $\CO(\|\hat T-T\|_F)$-near to
one of the components $\pm a_i$.
We also show in Lemma \ref{lem:perturbations_spurious_local_max} that spurious local maximizers (with lower objective values) do exist in the noisy \nolinebreak case.

The results imply a clear separation of the functional landscape between near-global maximizers,
with objective values close to $1$, and spurious local maximizers with small objective value. Hence,
the SPM objective can be used to validate a final iterate of projected gradient ascent in the noisy case.
In Theorem~\ref{thm:SPM_deflation_analysis}, we combine our landscape analysis with bounds on error propagation incurred during SPM's deflation steps. This gives guarantees for \textit{end-to-end} tensor decomposition using SPM.

Lastly,
we expose the relation between \ref{prob:PM}
and \ref{prob:SPM}.
Specifically, \ref{prob:SPM}
can be expressed as \ref{prob:PM} with the appropriate insertion of the inverse of the Grammian $(G_n)_{ij} := \langle a_i, a_j\rangle^n$.
The resulting
de-biasing effect on the local maximizers with respect to the components $\pm a_i$ (cf. Figure~\ref{fig:spm_vs_pm_objective})
is responsible for many advantages of \ref{prob:SPM}.
 Along the way,
we state a conjecture about the
minimal eigenvalue of $G_n$ when $a_i$ are i.i.d. uniform on the sphere, which may be of \nolinebreak independent \nolinebreak interest.

\section{Notation} \label{sec:notation}
\textbf{Vectors and matrices.}
When $x$ is a vector, $\| x \|_{p}$ is the $\ell_p$-norm ($p \in \mathbb{R}_{\geq 1} \cup \{ \infty \}$).
For $x, y \in \mathbb{R}^D$, the entrywise (or Hadamard) product is $x \odot y \in \mathbb{R}^D$, and the entrywise power is $x^{\odot s} := x \odot \ldots \odot x \in \mathbb{R}^D$ ($s \in \mathbb{N}$).
When $M$ is a matrix, $\| M \|_2$ is the spectral norm.
If $M$ is real symmetric, $\mu_j(M)$ is the eigenvalue of $M$ that is the $j$-th largest in absolute value. We denote the identity by $\Id_D \in \mathbb{R}^{D \times D}$.

\smallskip

\textbf{Tensors.}
A real tensor of length $D$ and order $m$ is an array of size $D \times D \times \ldots \times D$ ($m$ times) of real numbers.
Write $\CT_D^m := \left(\mathbb{R}^{D}\right)^{\otimes m} \cong \bbR^{D^m}$ for the  space of tensors of size $D^{\times m}$.
Meanwhile, $\symten{m} \subseteq \CT_D^m$ is the subspace of symmetric tensors (i.e., tensors unchanged by any permutation of indices).
The Frobenius inner product and norm are denoted by  $\langle \cdot, \cdot \rangle$ and $\N{\cdot}_F$, respectively.
Given any linear subspace of tensors $\mathcal{A} \subseteq \CT_D^m$, let
$P_{\mathcal{A}} : \CT_D^m \rightarrow \mathcal{A}$ denote the orthogonal projector onto $\mathcal{A}$ with respect to $\langle \cdot, \cdot \rangle$.
In the case $\mathcal{A} = \symten{m}$, the projector $P_{\symten{m}}$ is the symmetrization operator, $\operatorname{Sym}: \CT_D^m \rightarrow \symten{m}$.
Given $T \in \CT_D^{m_1}$ and $S \in \CT_D^{m_2}$, the tensor (or outer) product is $T \otimes S \in \CT_D^{m_1 + m_2}$, defined by
$(T \otimes S)_{i_1, \ldots, i_{m_1 + m_2}} := T_{i_1, \ldots, i_{m_1}} S_{i_{m_1 + 1}, \ldots, i_{m_2}}$.
For $T \in \CT_D^{m}$ and $s \in \mathbb{N}$,  the tensor power is
$T^{\otimes s} := T \otimes \ldots \otimes T \in \CT_D^{s m}$.
For $T \in \CT_D^{m_1}$, $S \in \CT_D^{m_2}$  with $m_1 \geq m_2$, the contraction $T \cdot S \in \CT_D^{m_1 - m_2}$ is defined by
 $\left(T \cdot S\right)_{i_1, \ldots, i_{m_1 - m_2}} := \sum_{j_1, \ldots, j_{m_2}} T_{i_1, \ldots, i_{m_1 - m_2}, j_1, \ldots, j_{m_2}} S_{j_1, \ldots, j_{m_2}}$.
Let $\opmyreshape(T, [d_1, \dots, d_{\ell}])$ be the function that reshapes the tensor $T$ to have dimensions $d_1, \dots, d_{\ell}$, as in corresponding Matlab/NumPy commands.

\smallskip

\textbf{Other.} The unit sphere in $\mathbb{R}^D$ is $\mathbb{S}^{D-1}$, and $\textrm{Unif}(\mathbb{S}^{D-1})$ is the associated  uniform probability distribution.
Given a function $f : \mathbb{R}^D \rightarrow \mathbb{R}$, the Euclidean gradient and Hessian matrix at $x \in \mathbb{R}^D$ are  $\nabla f(x) \in \mathbb{R}^D$ and $\nabla^2 f(x) \in \symten{2}$.
The Riemannian gradient and Hessian with respect to $\bbS^{D-1}$ at $x \in \bbS^{D-1}$ are  $\nabla_{\bbS^{D-1}} f(x)$ and $\nabla_{\bbS^{D-1}}^2 f(x)$ (see \cite{absil2009optimization}).
Write $\opmyspan$ for linear span,
$[K] := \{1, \ldots, K\}$, and $|A|$ for the cardinality of a finite set $A$.
Lastly, we use asymptotic notation \nolinebreak freely.

\section{Symmetric tensor decomposition via Subspace Power Method}
\label{sec:properties_SPM_objective}

In this section, we outline the tensor decomposition method SPM of \cite{kileel2019subspace}, and provide basic insights on the program \ref{prob:SPM}.
Throughout we assume that $m \geq 3$ is an integer and define $n := \lceil m/2\rceil$.

\textbf{SPM algorithm.} The input is a tensor $\hat T \in \symten{m}$, with the promise that $\hat T \approx T = \sum_{i=1}^{K} \lambda_i a_i^{\otimes m}$ for $\{(\lambda_i, a_i)\}_{i=1}^K$ Zariski-generic and $K \leq \binom{D+n-1}{n} - D$ (if $m$ is even) and $K \leq D^n$ (if $m$ is odd).
As a first step, SPM obtains the orthogonal projector $P_{\hat \CA} : \CT_D^m \rightarrow \hat \CA$ that projects onto
the column span of $\opmyreshape(\hat T, [D^{n}, D^{m-n}])$, by using matrix singular value decomposition.
Provided that $\hat T \approx T$, the associated subspace approximation error $\hat \CA \approx \CA = \Span{a_1^{\otimes n}, \ldots, a_K^{\otimes n}}$
defined by 
\begin{equation}
	\label{eq:DeltaA}
	\Delta_{\CA}:= \N{P_{\CA}-P_{\hat \CA}}_{F\rightarrow F} = \sup\limits_{T \in \CT_{D}^{n},\ \N{T}_F = 1}\N{P_{\CA}(T)-P_{\hat \CA}(T)}_F,
\end{equation}
can be bounded as follows. 
(Note that by \cite[Lem.~2.3]{chen2016perturbation}, we know $\Delta_{\CA} \leq 1$ a priori.)
\begin{lemma}[Error in subspace]
	\label{lem:error_in_subspace}
  Let $m \geq 3$, $n = \lceil \frac{m}{2} \rceil$, $T \in \symten{m}$ and assume that $M:=\opmyreshape(T, [D^{n}, D^{m-n}])$ has exactly $K$ nonzero singular values $\sigma_1(M) \geq \ldots \geq \sigma_K(M) > 0$.
  Let $\hat T \in \symten{m}$, $\hat M := \opmyreshape(T, [D^{n}, D^{m-n}])$.
  Assume $\Delta_M := \| M - \hat M \|_2 < \sigma_K(M)$. Then
  \begin{align}
    \label{eq:subspace_bound}
    \N{P_{\Im(M)}-P_{\Im_K(\hat M)}}_2 \leq \frac{\Delta_M}{\sigma_K(M) - \Delta_M},
  \end{align}
  where $\Im(M) \subseteq \mathbb{R}^{D^n}$ denotes the image of $M$ and $\Im_K(\hat M) \subseteq \mathbb{R}^{D^n}$ denotes the subspace spanned by the $K$ leading left singular vectors of $\hat M$.
  In particular, if $T = \sum_{i=1}^{K}\lambda_i a_i^{\otimes 2n}$, $\CA = \Span{a_i^{\otimes n} : i \in [K]}$, $\dim(\CA) = K$, and $\hat \CA$
  is the subspace spanned by $K$ leading  tensorized left singular vectors of $\hat M$,  the right-hand side
  of \eqref{eq:subspace_bound} upper-bounds $\Delta_{\CA}$.
\end{lemma}

\begin{remark}\label{rem:small-lambda}
If $T = \sum_{i=1}^K \lambda_i a_i^{\otimes m}$, one coefficient $\lambda_i$ is small and the vectors $\{a_i : i \in [K]\}$ are not too correlated, then the flattened tensor $\opmyreshape(T, [D^{n}, D^{m-n}])$ has a small eigenvalue.  This makes estimating the corresponding eigenvector sensitive to noise.  See Remark~\ref{rem:small-lambda} in the appendix.  
\end{remark}

Given $\hat \CA$, SPM seeks one tensor component $a_i$ by solving the noisy variant of \ref{prob:SPM} defined by
\begin{align}
  \tag{\textbf{nSPM-P}}
	\label{prob:nspm}
	\max_{x \in \bbS^{D-1}} F_{\hat \CA}(x),\quad \textrm{where}\quad F_{\hat \CA}(x) := \N{P_{\hat \CA}(x^{\otimes n})}_F^2.
\end{align}
Starting from a random initial point $x_0 \sim \textrm{Unif}(\bbS^{D-1})$, the projected gradient ascent iteration
\begin{equation} \label{eq:proj_grad}
    x \leftarrow \frac{x + \gamma  P_{\hat{\mathcal{A}}}(x^{\otimes n}) \cdot x^{\otimes (n-1)}}{\| x + \gamma  P_{\hat{\mathcal{A}}}(x^{\otimes n}) \cdot x^{\otimes (n-1)}\|_2},
\end{equation}
with a constant step-size $\gamma$,  is guaranteed to converge
to a second-order critical point of \ref{prob:nspm} almost surely by \cite{kileel2019subspace}.
Here we require that the step-size $\gamma$ is less than an explicit upper bound given in \cite{kileel2019subspace}.
Denoting by $\hat{a}_i$ the final iterate obtained by SPM, we accept the candidate approximate tensor component $\hat{a}_i$ if $F_{\hat \CA}(\hat a_i)$ is large enough; otherwise we draw a new starting point $x_0$   and re-run \eqref{eq:proj_grad}.

Next given $\hat{a}_i$,  SPM
evaluates a deflation formula based on Wedderburn rank reduction \cite{chu1995rank} from matrix algebra to compute
the corresponding weight $\hat{\lambda}_i$.
Then, we update the tensor $\hat{T} \leftarrow \hat{T} - \hat{\lambda}_i \hat{a}_i^{\otimes m}$.

To finish the tensor decomposition, SPM performs the projected gradient ascent and deflation steps $K$ times to compute all of the tensor components and weights
$\{(\hat{\lambda}_i, \hat{a}_i)\}_{i=1}^K$.

\smallskip

\textbf{Preparatory material about \ref{prob:nspm}.} The goal of this paper is to show that
second-order critical points of \ref{prob:nspm} with reasonable function value must be near the global maximizers $\pm a_1,\ldots,\pm a_K$
of \ref{prob:SPM}, under suitable incoherence assumptions on the rank-one components $a_1,\ldots,a_K$.
Naturally, the optimality conditions for \ref{prob:nspm} play an important part in this analysis.

\begin{proposition}[Optimality conditions]
\label{prop:optimality_conditions}
Let $x \in \bbS^{D-1}$ be first and second-order critical for \textup{\ref{prob:nspm}}.  Then for each $z \in \bbS^{D-1}$ with $z \perp x$, we have
\begin{align}
	\label{eq:stationary_point}
	& P_{\hat \CA}(x^{n}) \cdot x^{n-1} = F_{\hat \CA}(x)x, \\
	\label{eq:before_derived_second_order_optimality}
	&F_{\hat \CA}(x) \geq n\|P_{\hat \CA}(x^{n-1}z)\|_F^2 + (n-1)\langle P_{\hat \CA}(x^{n}), x^{n-2}z^{2}\rangle.
\end{align}
Furthermore, for any $y\in \bbS^{D-1}$ we have
\begin{equation}
\begin{aligned}
	\label{eq:derived_second_order_optimality}
	F_{\hat \CA}(x) \geq\ & n\N{P_{\hat \CA}(x^{n-1} y)}_F^2 + (n-1)\langle P_{\hat \CA}(x^{ n}), x^{n-2} y^{ 2}\rangle - 2(n-1)F_{\hat \CA}(x)\langle x, y\rangle^2.
\end{aligned}
\end{equation}
\end{proposition}

In the analysis later, we make frequent use of expressing the objective $F_{\CA}(x)$
using the Gram matrix
\begin{equation} \label{eq:grammian-def}
G_n \in \symten{2} \quad \textup{ defined by } \,\,  (G_{n})_{ij} := \langle a_i^{\otimes n}, a_j^{\otimes n}\rangle= \langle a_i, a_j\rangle^n.
\end{equation}
Under linear independence of the tensors $a_1^{\otimes n},\ldots, a_K^{\otimes n}$, which is implied by our assumptions made later, 
 the inverse $G_n^{-1}$ exists and the noiseless program \ref{prob:SPM} can be expressed as follows.
\begin{lemma}
\label{lem:reformulation_frob_norms}
Let $A := [a_1|\ldots|a_K] \in \bbR^{D\times K}$ and $\{a_i^{\otimes n}: i \in [K]\}$ be linearly independent. We have
\begin{align}
\label{eq:norm_identity}
F_{\CA}(x) = \N{P_{\CA}(x^{\otimes n})}_F^2 = \left((A^\top x)^{\odot n}\right)^\top G_n^{-1} \left((A^\top x)^{\odot n}\right).
\end{align}
\end{lemma}

Lemma \ref{lem:reformulation_frob_norms} exposes the relation between \ref{prob:SPM} and \ref{prob:PM}. While
\ref{prob:PM} can be rewritten  $\langle T, x^{\otimes m} \rangle = \langle (A^\top x)^{\odot n}, (A^\top x)^{\odot n}\rangle$ if $m$ is even,
\ref{prob:SPM} takes into account correlations among the tensors
$a_1^{\otimes n},\ldots,a_K^{\otimes n}$ and inserts the Grammian $G_n^{-1}$ into \eqref{eq:norm_identity}.
Consequently,
correlations among the tensor components are considered in \ref{prob:SPM}, 
without any a priori knowledge of the tensors \nolinebreak $a_1^{\otimes n},\ldots,a_K^{\otimes n}$.

In the special case of orthonormal systems, or more generally systems that resemble equiangular tight frames \cite{fickus2012steiner,sustik2007existence}, the \ref{prob:SPM} and \ref{prob:PM}
objectives coincide up to shift and scaling.

\begin{lemma}
\label{pm_vs_spm}
Assume there exist $\rho \in (-1,1) \setminus \{ \tfrac{-1}{K-1} \}$ and $M \in \bbR$ such that $\langle a_i,a_j\rangle^n = \rho$ for
all $i \neq j$ and $\sum_{i \in [K]}\langle x,a_i\rangle^n = M$ for all $x \in \bbS^{D-1}$. Denote $A = [a_1|\ldots|a_K] \in \bbR^{D \times K}$. Then
\begin{equation}
F_{\CA}(x) = (1-\rho)^{-1}\|A^\top x\|_{2n}^{2n}-\left((1-\rho)^2+K\rho(1-\rho)\right)^{-1} \! \rho M^2.
\end{equation}
\end{lemma}

\section{Main results}
\label{sec:landscape_analysis}
In this section, we present the main results about local maximizers of the \ref{prob:nspm} program. Section
\ref{subsec:deterministic} is tailored to low-rank tensor models with $K = \CO(D)$ components that satisfy certain
deterministic frame conditions. Section \ref{subsec:average_case_analysis}  then considers the overcomplete case $K = \widetilde o(D^{\lfloor m/2\rfloor})$
in an average case scenario, where $a_1,\ldots,a_K$ are modeled as independent copies
of an isotropic random vector.

\subsection{Low-rank tensors under deterministic frame conditions}
\label{subsec:deterministic}
Motivated by frame constants in frame theory \cite{christensen2003introduction}, we measure the incoherence of the ensemble $a_1,\ldots,a_K$ by scalars
$\rho_s \in \bbR_{\geq 0}$, which are defined via
\begin{align}
\label{eq:def_rho}
\rho_{s} := \sup_{x \in \bbS^{D-1}} \sum_{i=1}^{K}\SN{\langle x, a_i\rangle}^s - 1.
\end{align}
They satisfy the order relation $\rho_s \leq \rho_{s'}$ for $s' \leq s$, due to $\N{a_i}_2 = 1$, and can be related
to extremal eigenvalues of Grammians $G_{s}$ and $G_{\lfloor s/2 \rfloor}$ as shown in the following result.
\begin{lemma}
\label{lem:incoherence_scalars}
Let $\{a_i : i \in [K]\} \subseteq \bbS^{D-1}$ and $(G_{s})_{ij} := \langle a_i, a_j\rangle^s$ for $s \in \bbN$. Then
\begin{align}
\label{eq:rho_s_grammian_characterization}
1-\rho_s \leq \mu_K(G_s) \leq \mu_1(G_s) \leq 1+\rho_s \leq \mu_1(G_{\lfloor s/2\rfloor}).
\end{align}
\end{lemma}
The characterization in Lemma \ref{lem:incoherence_scalars} allows to compute bounds for
$\rho_s$ for low-rank tensors with mutually incoherent components or rank-$\CO(D)$ tensor with random components.
We provide details on this in Remark \ref{rem:mutual_coherence} below, but first state the main
guarantee about local maximizers using $\rho_2$ and $\rho_n$.

\begin{theorem}[Main deterministic result]
	\label{thm:main_result_deterministic}
	Let $\{a_i : i \in [K] \} \subseteq \mathbb{S}^{D-1}$ and $\CA = \operatorname{Span}\{a_i^n : i \in [K]\}$.
	Let $\hat \CA \subseteq \symten{n}$ be a perturbation of $\CA$ with $\Delta_{\CA} = \| P_{\CA} - P_{\hat \CA}\|_{F \rightarrow F}$. Let
	\begin{equation} \label{eq:def-tau}
		\tau := \frac{1}{6} - n^2 \rho_2 - (n^2 + n) \rho_n\quad \text{and} \quad \Delta_0 := \frac{2 \tau}{2 + 4 \tau + 3 n^2}.
	\end{equation}
	Then, if $\Delta_{\CA} < \Delta_0$, the program \textup{\ref{prob:nspm}} has exactly $2K$ second-order critical points in the superlevel set where
	\begin{equation}\label{eq:deterministic-superlevel}
		F_{\hat \CA}(x) \ge  \frac{2 + 2\tau + 3n^2}{2 \tau} \Delta_{\CA}.
	\end{equation}
	Each of these critical points is a strict local maximizer for $F_{\hat \CA}$.
	Further for each such point $x^*$, there exists unique $i \in [K]$ and $s \in \{-1,1\}$ such that
	\begin{equation}
		\|x^* - s a_i\|_2^2  \le \frac{2 \Delta_{\CA}}{n}.
	\end{equation}
	In the noiseless case ($\Delta_{\CA} =0$), if $\tau \ge 0$ then there are precisely $2K$ second-order critical points of \textup{\ref{prob:SPM}} with positive functional value, and they are the global maximizers  $s a_i$ $(i \in [K], s \in \{-1,1\})$.
\end{theorem}

\begin{remark}
\label{rem:mutual_coherence} Using Lemma \ref{lem:incoherence_scalars}, we identify
two situations where $\rho_2$ and $\rho_n$ can be bounded from above.
\vspace{-1.2em}
\begin{enumerate}
\item \textit{Mutually incoherent ensembles.}
Let $a_1,\ldots,a_K$ have mutual incoherence $\rho := \max\limits_{i\neq j}\SN{\langle a_i, a_j\rangle}$.
Using Gershgorin's circle theorem, we obtain
\begin{align*}
\rho_s &\leq \mu_1(G_{\lfloor s/2\rfloor}) - 1 \leq \max_{i \in [K]}\sum_{j\neq i}\SN{\langle a_i, a_j\rangle}^{\lfloor s/2\rfloor}\leq (K-1)\rho^{\lfloor s/2\rfloor} \quad \textrm{ for any } s \in \bbN_{\geq 2},
\end{align*}
which implies that the conditions of Theorem \ref{thm:main_result_deterministic} are saisfied if $K\rho$ is sufficiently small. This setting is comparable
to the analysis for \ref{prob:PM} in \cite{sanjabi2019does}.
Moreover, if $K=D$ and $a_1, \ldots, a_K$ are mutually orthogonal, then $\rho_s = 0$ for each $s \geq 2$.  Therefore
Theorem~\ref{thm:main_result_deterministic} holds with $\tau = 1/6$ for \textup{orthogonally decomposable tensors} \cite{auddy2020perturbation}.

\item \textit{Low-rank random ensembles.} Let $a_1,\ldots,a_K$ be independent copies of an isotropic unit-norm random
vector $a$ with sub-Gaussian norm $\CO(1/\sqrt{D})$ (e.g., $a_i\sim \textrm{Unif}(\bbS^{D-1})$). With high probability, the ensemble can achieve arbitrarily
small $\rho_2$, provided that $K \leq CD$ for a sufficiently small constant $C > 0$. The proof of this fact relies on
$A=[a_1|\ldots|a_K]$ satisfying, with high probability, the so-called $(K, \delta)$ restricted isometry property \cite[Thm.~5.65]{vershynin2010introduction}
as defined in Definition \ref{def:RIP} in the next section. We note that despite requiring conditions  milder than those in the unit-norm tight frame analysis
for \ref{prob:PM} in \cite{qu2019geometric}, we achieve a comparable scaling of $K = \CO(D)$.
\end{enumerate}
\end{remark}

In the noiseless case where $\Delta_{\CA} = 0$, Theorem \ref{thm:main_result_deterministic} shows
that all local maximizers coincide with global maximizers $\pm a_1,\ldots,\pm a_K$,
provided the tensor components are sufficiently incoherent to ensure $\tau \geq 0$.
In the noisy case, all local maximizers with objective values $F_{\hat \CA}(x) \geq C(n,\tau)\Delta_{\CA}$ are $\CO(\Delta_{\CA})$-close to the
global optimizers of the noiseless objective $F_{\CA}$, where the constant $C(n,\tau)$ increases
as the incoherence of the vectors $a_1,\ldots,a_K$ shrinks.
Unfortunately, the presence of
spurious local maximizers of $F_{\hat \CA}$ with small objective values cannot be avoided under a deterministic noise model, as the next result shows.

\begin{lemma}
\label{lem:perturbations_spurious_local_max}
Let $\delta \in (0,1)$, $a \in \bbS^{D-1}$ and $\CA = \Span{a^{\otimes n}}$.
Then there exists a subspace $\hat \CA\subseteq \textup{Sym}(\CT_D^n)$  with $\dim(\hat \CA) = 1$ and
$\|P_{\CA} - P_{\hat \CA}\|_{F\rightarrow F} = \delta$ such that \textup{\ref{prob:nspm}}
possesses a strict local maximizer of objective \nolinebreak value exactly $\delta^2$.
\end{lemma}

\subsection{Average case analysis of overcomplete tensors}
\label{subsec:average_case_analysis}
The overcomplete case with $K = \widetilde o(D^{\lfloor m/2\rfloor})$ falls outside the range of
Theorem \ref{thm:main_result_deterministic}, because $\tau$ in \eqref{eq:def-tau}
becomes negative when $K\gg D$.
Instead, our analysis for the overcomplete case relies on $A = [a_1|\ldots|a_K] \in \bbR^{D \times K}$ obeying the $(p,\delta)$-restricted isometry property (RIP) for $p = \CO(D/\log(K))$.

\begin{definition}
	\label{def:RIP}
	Let $A \in \bbR^{D\times K}$, $1\leq p \leq K$ be an integer, and $\delta \in (0,1)$.
	We say that $A$ is \textit{$(p, \delta)$-RIP} if every $D\times p$ submatrix $A_p$ of $A$ satisfies
	$\|A_p^\top A_p - \Id_p\|_{2}\leq \delta$.
\end{definition} 

A consequence of the RIP, which is particularly useful in the analysis of overcomplete tensor models, is
that the correlation coefficients $\{\langle a_i, x\rangle : i \in [K]\}$ can naturally be split into two groups.
\begin{lemma}[RIP-induce partitioning of correlation coefficients]
\label{prop:RIP_consequence}
Suppose that  $A = [a_1|\ldots|a_K] \in \bbR^{D \times K}$ satisfies the $(p,\delta)$-RIP for $p = \lceil c_\delta D/\log(K)\rceil$. Let  $\tilde c_\delta := (1+\delta)/c_\delta$. Then for all $x\in \bbS^{D-1}$ there is a subset of indices $\CI(x) \subseteq [K]$ with cardinality $p$ such that
\begin{align}
	\label{eq:rip_requirement_1}
	1 - \delta \leq \sum_{i \in \CI(x)}\langle a_i, x\rangle^2 \leq 1 + \delta\quad \text{and}\quad \langle a_i, x\rangle^2 \leq 	\tilde c_\delta \frac{\log(K)}{D} \text{ for } i \not \in \CI(x),
\end{align}
\end{lemma}
Let us now collect all assumptions needed to analyze the overcomplete case.
\begin{definition}
\label{def:assumptions_overcomplete}
Let $A := [a_1|\ldots|a_K]$. We require the following
assumptions:
\begin{enumerate}[leftmargin=1.5cm,label=\textbf{A\arabic*}]
  \item\label{enum:RIP} There exists $c_{\delta}>0$, depending only on $\delta$, such that $A$ is $(\lceil c_{\delta}D/\log(K) \rceil,\delta)$-RIP.
  \item\label{enum:correlation} There exists  $c_1 > 0$, independent of $K,D$, such that $\max_{i,j:i\neq j}\langle a_i, a_j\rangle^2 \leq c_1\log(K)/D$.
  \item \label{enum:GInverse} There exists $c_2 > 0$, independent of $K,D$, such that $\N{G_{n}^{-1}}_2 \leq c_2$.
\end{enumerate}
\end{definition}

To the best of our knowledge, there are no known deterministically constructed systems of size $K = \widetilde o(D^{\lfloor m/2\rfloor})$ which
satisfy \ref{enum:RIP} - \ref{enum:GInverse} for small $\delta \ll 1$ and
constants $c_{\delta},c_1,c_2$ that do not depend on $K,D$.
However, by modeling components via a sufficiently-spread random vector,
i.e., by considering an average case scenario with $a_i \sim \textup{Unif}({\bbS^{D-1}})$, they hold with high probability.

\begin{proposition}[\ref{enum:RIP} - \ref{enum:GInverse} for random ensembles]
\label{lemma:random_results}
Let $a\sim \textup{Unif}(\bbS^{D-1})$ and assume $a_1,\ldots,a_K$ are $K$ independent copies of $a$, where
$\log(K) = o(D)$.
Fix an arbitrary constant $\delta \in (0,1)$.
There exist constants $C > 0$ and $D_0 \in \bbN$ depending only on $\delta$ such that
for all $D \geq D_0$, and with probability at least $1-K^{-1}-2\exp(-C\delta^2 D)$,
conditions \textup{\ref{enum:RIP}} and \textup{\ref{enum:correlation}} (with $c_1\leq C$) hold.
Furthermore,
if $n = 2$ and $K = o(D^2)$, \textup{\ref{enum:GInverse}} holds with probability
$1-C(e D/\sqrt{K})^{-C\sqrt{K}}$ for $c_2\leq C$.
\end{proposition}

\begin{remark}[\ref{enum:GInverse} when $n > 2$]
Following \cite{adamczak2011restricted,fengler2019restricted}, we give a self-contained proof for \ref{enum:GInverse} when $n = 2$ in the appendix.
We are currently not able to extend the proof to $n > 2$, because some technical tools such as
the Hanson-Wright inequality \cite{rudelson2013hanson} and an extension of \cite[Thm.~3.3]{adamczak2011restricted},
which ensures the RIP for matrices whose columns have sub-exponential tails,
have not yet been fully developed for random vectors with dependent entries and heavier tails (so-called $\alpha$-sub-exponential tails
or sub-Weibull tails).
However we strongly believe that \ref{enum:GInverse} holds for $n > 2$ and $K = o(D^n)$.
For now, we formulate this as a conjecture, supplemented with numerical evidence presented in Figure \ref{fig:grammian_wellposedness_suppl}. We also add that the conjecture, once proven, would complement recent advances on the well-posedness of random tensors with fewer statistical dependencies among the components \cite{bryson2019marchenko, vershynin2020concentration}.
\end{remark}

\begin{figure}
\centering
\includegraphics[width = \textwidth]{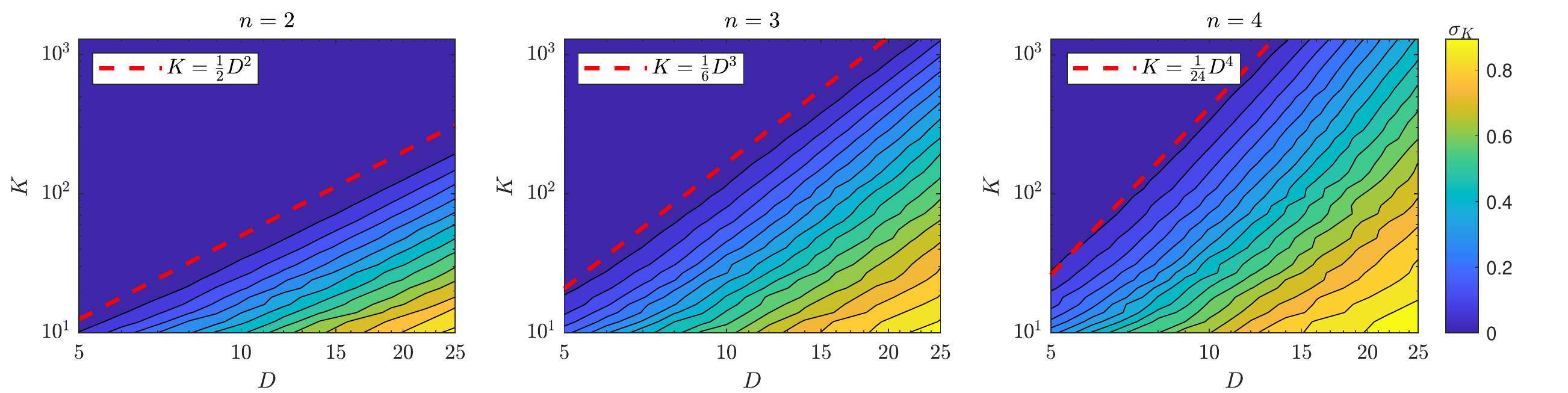}
\caption{The average smallest eigenvalue of the Grammian $(G_{n})_{ij}=\langle a_i, a_j\rangle^n$ for a random ensemble $a_1,\ldots,a_K$
consisting of independent copies of the random vector $a \sim \textrm{Unif}(\bbS^{D-1})$, over 100 runs per $K,D,n$.
From left to right, we consider different tensor orders $n=2,3,4$. As long as $K\leq C(n) D^n$ for some constant $C(n)$ depending only on $n$ (red line),
inverting the Grammian is well-posed and the smallest eigenvalue of $G_n$ is bounded from below, in accord with Conjecture \ref{conj:grammians}.}
\label{fig:grammian_wellposedness_suppl}
\end{figure}

\begin{conjecture}[Grammians of independent symmetric rank-one tensors]
\label{conj:grammians}
Let $a_1,\ldots,a_K$ be independent copies of the random vector $a \sim \textup{Unif}(\bbS^{D-1})$
and fix $\epsilon > 0$ arbitrarily.
Then there exists some constant $\kappa_n > 0$ and an increasing function $\gamma_n:(0, \kappa_n) \to (0,1)$, both depending only on $n$,
such that $\gamma_n(\kappa) \rightarrow 1$ as $\kappa \rightarrow 0$, and if $K\leq \kappa D^n$ for some $\kappa < \kappa_n$ we have
\begin{align}
\label{eq:conjecture_2}
\bbP(\N{G_{n}^{-1}}_2 \geq \gamma_n(\kappa) -\epsilon) \rightarrow 1\quad \textrm{as}\quad D\rightarrow \infty.
\end{align}
In particular, if $K = o(D^n)$ then $\|G_{n}^{-1}\|_2\rightarrow 1$ as $D \rightarrow \infty$.
\end{conjecture}

We now present our main theorem about local maximizers of \ref{prob:nspm} in the overcomplete \nolinebreak case.

\begin{theorem}[Main random overcomplete result]
\label{thm:main_result_for_now}
Let $K, D \in \bbN$, define $\varepsilon_K := K\log^n(K)/D^n$, 
and suppose that $\lim_{D\to\infty} \varepsilon_K = 0$.
Assume $a_1,\ldots,a_K\in \bbR^D$ satisfy \textup{\ref{enum:RIP}} - \textup{\ref{enum:GInverse}}
for some $\delta,c_1,c_2 > 0$. Then
there exist $\delta_0$, depending only on $n$, $c_1$ and $c_2$, and $D_0, \Delta_0, C$, which depend additionally on $c_\delta$,
such that if $\delta < \delta_0$, $D > D_0$,
and $\Delta_{\CA}\leq \Delta_0$, the program \textup{\ref{prob:nspm}} has exactly
$2K$ 
second-order critical points in the superlevel set
\begin{equation}\label{eq:level_set_def}
\left\{x\in\bbS^{D-1}: F_{\hat \CA}(x) \ge C \varepsilon_K + 5 \Delta_{\CA} \right\}.
\end{equation}
Each of these critical points is a strict local maximizer for $F_{\hat \CA}$.
	Further for each such point $x^*$, there exists unique $i \in [K]$ and $s \in \{-1,1\}$ such that
	\begin{equation}
		\|x^* - s a_i\|_2^2  \le \frac{2 \Delta_{\CA}}{n}.
	\end{equation}
		In particular, in the noiseless case ($\Delta_{\CA} =0$), the second-order critical points in the superlevel set \eqref{eq:level_set_def} are exactly the global maximizers  $s a_i$ $(i \in [K], s \in \{-1,1\})$.
\end{theorem}

By Theorem \ref{thm:main_result_for_now}, all non-degenerate local maximizers in the superlevel set \eqref{eq:level_set_def}
are close to global maximizers of the noiseless objective $F_{\CA}$. In the noiseless case
$\Delta_{\CA} = 0$, they coincide with global maximizers and the right hand side in \eqref{eq:level_set_def} tends to $0$
as $D\rightarrow \infty$. Hence, we have a clear separation between the global maximizers $\pm a_1,\ldots, \pm a_K$ and
degenerate local maximizers with vanishing objective value, so that the objective
value acts as a certificate for the validity \nolinebreak of \nolinebreak an \nolinebreak identified \nolinebreak maximizer.

Theorem \ref{thm:main_result_for_now} is not a global guarantee because
a random starting point $x_0 \sim \textrm{Unif}(\bbS^{D-1})$, used for starting the projected gradient ascent iteration \eqref{eq:proj_grad},
is, with high probability, not contained in \eqref{eq:level_set_def}.

\begin{remark}[Objective value at random initialization]
For a random sample $x_0 \sim \textrm{Unif}(\bbS^{D-1})$ in the clean case $\Delta_{\CA}=0$, we empirically observe the objective value
$ F_{\CA}(x_0) \approx CK/D^n$ as illustrated in Figure \ref{fig:expected_value}.
As such, we conjecture $\bbE_{x_0, a_1, \ldots, a_k \sim \operatorname{Unif}(\mathbb{S}^{D-1})}[F_{\CA}(x_0)] = \mathcal{O}(\frac{K}{D^n})$.
Comparing the objective to the level set condition in Theorem~\ref{thm:main_result_for_now},
\begin{equation}
\label{eq:level_set_suppl_aux}
\left\{x \in \bbS^{D-1} : F_{\CA}(x) \geq C\frac{K\log^n(K)}{D^n}+ 5\Delta_{\CA}\right\},
\end{equation}
 a random starting point $x_0$ 
 therefore falls short of satisfying \eqref{eq:level_set_suppl_aux} by a $\log^n(K)$-factor only.
 In this sense, Theorem \ref{thm:main_result_for_now} furnishes a ``near-global" guarantee for \ref{prob:nspm} in the random overcomplete \nolinebreak case.
\end{remark}
 
\begin{figure}
	\centering
	\includegraphics[width = .5\textwidth]{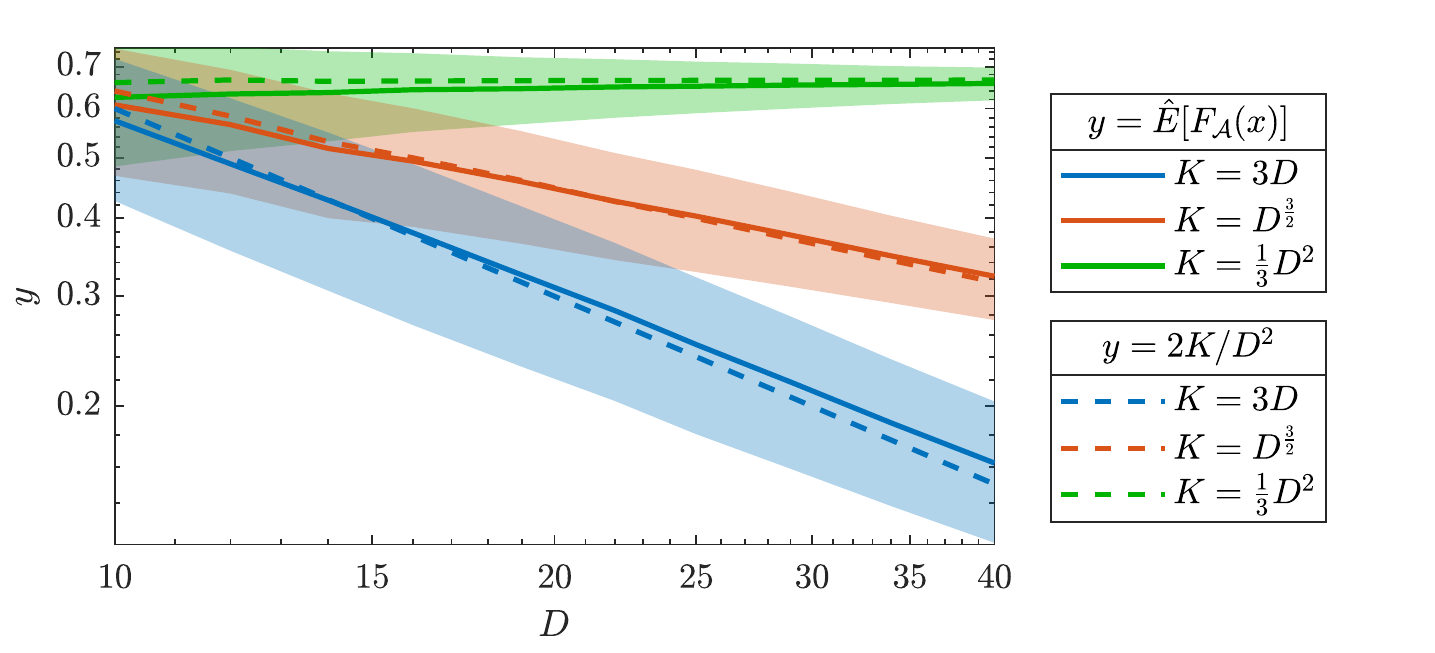}
	\caption{The empirical average of $F_{\CA}(x)$ over 10000 trials of a random ensemble $a_1,\ldots,a_K, x$ consisting of $K+1$ independent copies of the random vector $a \sim \textrm{Unif}(\bbS^{D-1})$ for different scalings of $K$ and $D$.  The shaded areas indicate plus/minus one empirical standard deviation.}
	\label{fig:expected_value}
\end{figure}

\subsection{End-to-end tensor decomposition}

By combining our landscape analyses with bounds for error propagation during deflation, we obtain a theorem about end-to-end tensor decomposition using the SPM algorithm.  
That is, under the conditions of Theorem~\ref{thm:main_result_deterministic} or \ref{thm:main_result_for_now}, a tweaking of SPM (Algorithm~\ref{alg:modified_SPM} in the appendix) recovers the entire CP decomposition   exactly in the noiseless case ($\hat T = T$).
In the noisy regime, it obtains an approximate CP decomposition, and we bound the error in terms of $\Delta_{\CA}$. Due to space constraints, we leave precise descriptions of Algorithm~\ref{alg:modified_SPM} and our deflation bounds
to the supplementary material.

\begin{theorem}[Main result on end-to-end tensor decomposition]\label{thm:SPM_deflation_analysis}
	Let $T = \sum_{i=1}^{K}\lambda_i a_i^{\otimes m} \in \symten{m}$ and $M:=\opmyreshape(T, [D^{n}, D^{m-n}])$.  
	Let  $\sigma_1(M) \geq \ldots \geq \sigma_K(M)$ be the singular values of $M$, and assume $\sigma_K(M) > 0$.
	For other tensor $\hat T \in \symten{m}$, let $\hat M =\opmyreshape(\hat T, [D^{n}, D^{m-n}])$, assume $\Delta_M := \| M - \hat M \|_2 < \frac12\sigma_K(M)$ and let $\hat \Delta_{\CA} = \frac{\Delta_M}{\sigma_K(M) - \Delta_M}$. 
	Suppose that $T$ satisfies the assumptions of either Theorem \ref{thm:main_result_deterministic} or Theorem \ref{thm:main_result_for_now}, define $\Delta_0$ as in the corresponding theorem statement and let $\ell(\Delta_\CA)$ be the corresponding level set threshold.\footnote{In both theorem statements, the level set threshold depends on $\Delta_\CA$.} Then there exist constants $C_1,C_2$, not depending on $\hat T$ or $\Delta_M$, such that if we define $\tilde\Delta_\CA := C_1 \hat \Delta_{\CA} + C_2\sqrt{\hat \Delta_{\CA}}$ the following holds.  Assume $\tilde\Delta_\CA < \Delta_0$  and
	when we apply Algorithm~\ref{alg:modified_SPM} to $\hat{T}$ each run of projected gradient ascent converges to a point with functional value at least $\ell(\tilde\Delta_{\CA})$. Then, letting $(\hat a_i, \hat \lambda_i)_{i \in [K]}$ be the output of Algorithm~\ref{alg:modified_SPM} applied to $\hat T$, there exist a permutation $\pi \in \operatorname{Perm}(K)$ and signs $s_1, \ldots, s_K \in \{-1, +1\}$ with
	\begin{equation*}
	\|s_i a_{\pi(i)} - \hat a_i \|_2 \le \sqrt{\frac{2 \hat \Delta_{\CA}}{n }} \quad \text{and} \quad \SN{\frac{s_i^m}{\lambda_{\pi(i)}} - \frac1{\hat \lambda_i}} \le 
	\frac{2\sqrt{m/n} }{\sigma_K(M)}\sqrt{\hat \Delta_{\CA}}  +  \frac{4}{\sigma_K(M)}\hat \Delta_{\CA} \quad \forall \, i \in [K].
\end{equation*}
	In particular, in the noiseless case ($\Delta_M=0$), Algorithm 1 returns the exact CP decomposition of $T$. 
\end{theorem}

We conclude that the Subspace Power Method, with an initialization scheme for \ref{prob:nspm} that gives $x\in \bbR^D$ where $F_{\hat{ \mathcal{A}}}(x) > \ell(\Delta_{\CA})$, is a guaranteed algorithm for low-rank tensor decomposition.

\section{Numerical experiments}
\label{sec:numerical}

\begin{figure}
\begin{subfigure}{\textwidth}
  \centering
  \includegraphics[width=\textwidth]{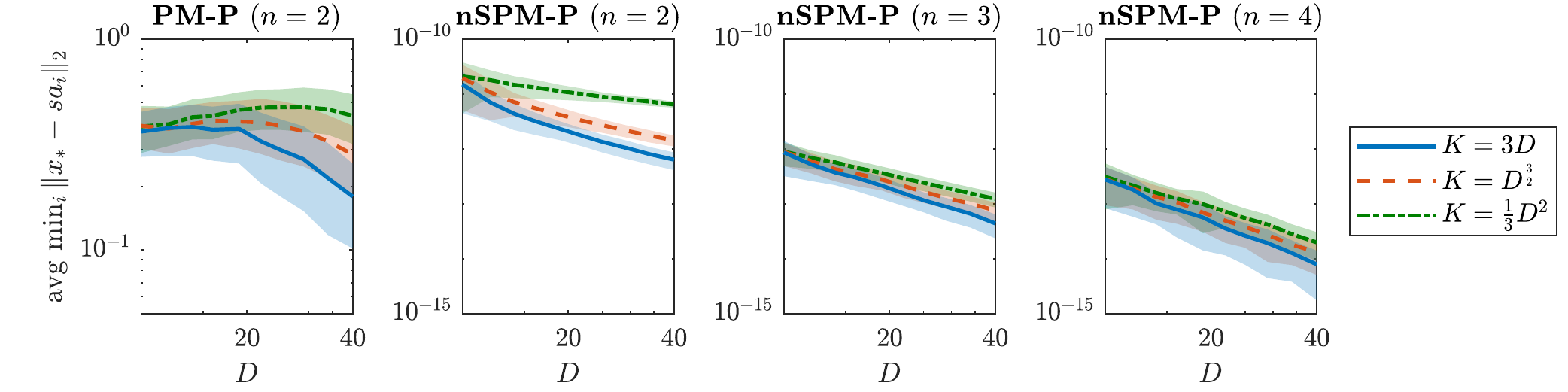}
  \caption{Recovering individual tensor components $a_i$ using \ref{prob:PM} and \ref{prob:nspm} in the noiseless case;}
  \label{fig:PM_vs_SPM_L}
\end{subfigure}
\begin{subfigure}{.31\linewidth}
	\centering
	\includegraphics[width=\textwidth]{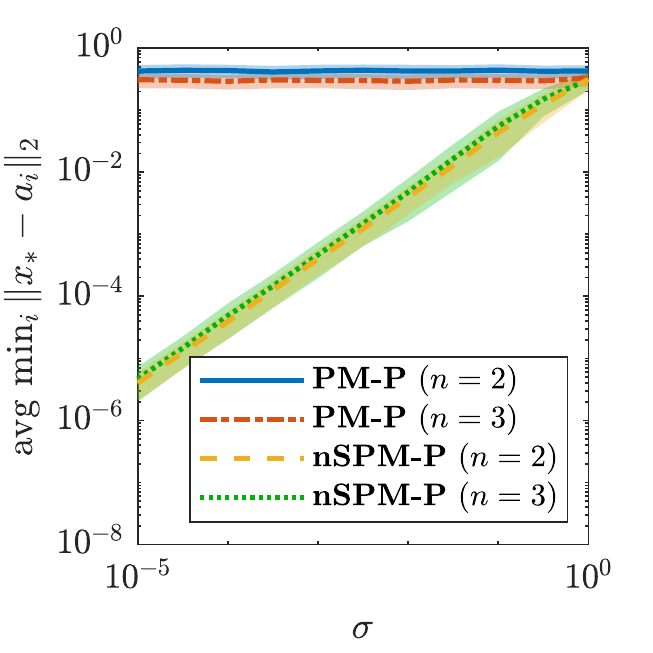}
	\caption{\ref{prob:PM} and \ref{prob:nspm} applied to noisy tensors;}
	\label{fig:PM_vs_SPM_noise}
\end{subfigure}
\begin{subfigure}{.68\linewidth}
	\centering
	\includegraphics[width=\textwidth]{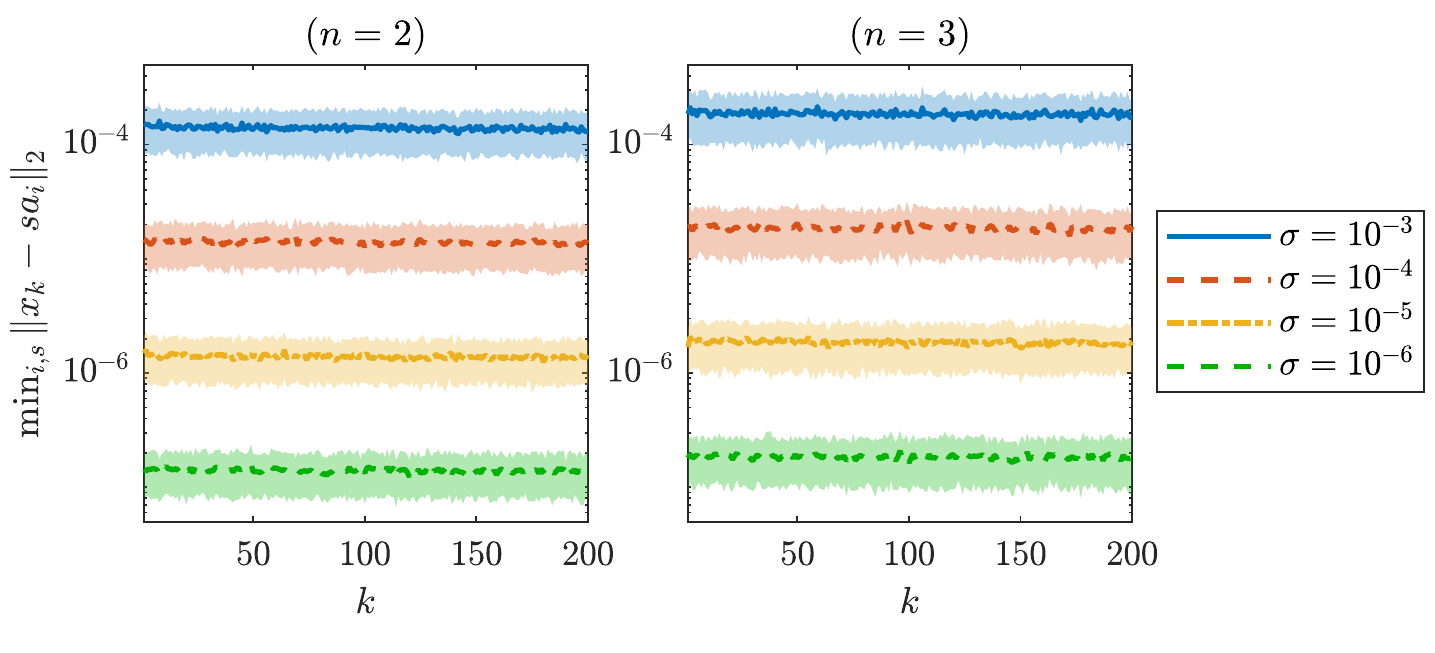}
	\caption{Error of individual recovery with \ref{prob:nspm} after $k$ deflations.\newline}
	\label{fig:SPM_deflation}
\end{subfigure}

\caption{Numerical results for symmetric tensor decomposition using \ref{prob:PM} versus \ref{prob:nspm}. The shaded areas indicate plus/minus one standard deviation. The experiments are described in Section~\ref{sec:numerical}.}
\label{fig:PM_vs_SPM_numsect}
\end{figure}

Here we present numerical experiments that corroborate the theoretical findings of Section \ref{sec:landscape_analysis}.
We illustrate that  SPM identifies exact tensor components in the noiseless case (in contrast to PM),
and that SPM behaves robustly in noisy tensor decomposition.
We use the implementation of SPM of the first and third authors, available at \verb|https://github.com/joaompereira/SPM|, which is licensed under the MIT license.
All of our experiments presented below may be conducted on a standard laptop computer within a few hours.
For further
numerical experiments, we refer the reader to
\cite{kileel2019subspace}, where SPM was tested in a variety of other scenarios, justifying it as a possible replacement for
state-of-the-art symmetric tensor decomposition methods such as FOOBI \cite{foobi2007} or Tensorlab \cite{vervliet2016tensorlab}.

\paragraph{Global optimizers and noise robustness.}
In the first set of experiments, we are interested in the recovery of individual tensor components $a_i$
for different $D$, $K$ and $m$ from noisy approximately low rank tensors.
With $m=2n$ as the tensor order, we create noiseless tensors with $a_i \sim \textrm{Unif}(\bbS^{D-1})$) as the tensor components and $\lambda_i =  \sqrt{D^m / K}\,\tilde \lambda_i $ as the tensor weights, where $\tilde \lambda_i \sim \text{Unif}([1/2, 2])$. This way, the variance of each entry of the tensor is about $1$.
We construct noisy tensors $\hat T \approx T$ by adding independent copies of $\epsilon \sim \CN(0,m! \sigma^2)$
to each entry of the tensor and then project onto $\symten{m}$. The entries with all distint indices of the projected noise tensor have variance $\sigma^2$;
the variance of the remaining entries is also a multiple of $\sigma^2$ and the number of such entries is $o(D^m)$.

After constructing the tensor, we sample $10$ points from $\textrm{Unif}(\bbS^{D-1})$ as initializations for
 projected gradient ascent applied to \ref{prob:PM} and \ref{prob:nspm}. We iterate until convergence
and record the distance between the final iterate and the closest $\pm a_i$ among the tensor components.
Averages and standard deviations (represented by shaded areas) over 1000 distances (10 distances per tensor, 100 tensors) are depicted in Figures \ref{fig:PM_vs_SPM_L} and \ref{fig:PM_vs_SPM_noise}.
In Figure~\ref{fig:PM_vs_SPM_L}, we use noiseless tensors with $\sigma = 0$, while in Figure~\ref{fig:PM_vs_SPM_noise} $\sigma$ ranges from $10^{-5}$ to $1$, and the plots for $4^\text{th}$ and $6^\text{th}$ order tensors are superposed. We set $D=20, K=100$ for the $4^\text{th}$ order tensor $(n=2)$ and $D=10, K=100$ for the $6^\text{th}$ order $(n=3)$.

Figure \ref{fig:PM_vs_SPM_L} illustrates that projected gradient ascent applied to \ref{prob:nspm}
always converges to the global maxima in the noiseless case (up to numerical error),
provided the scaling is $K = \CO(D^{\lfloor m/2\rfloor}) = \CO(D^{2})$ and the constant adheres to the rank constraints
pointed out in \cite[Prop.~4.2]{kileel2019subspace}. This is in agreement with Theorem \ref{thm:main_result_for_now},
which shows local maxima with sufficiently large objective have to coincide with global maximizers. In the noisy case,
Figure \ref{fig:PM_vs_SPM_noise} illustrates the robustness of \ref{prob:nspm}, giving an error of $\CO(\sigma)$.
In contrast to \ref{prob:nspm}, the \ref{prob:PM} objective suffers from a large bias and does not
recover exact tensor components. The effect of the bias even dominates
errors induced by moderately-sized entrywise noise.

\paragraph{Deflation with SPM.}
In the second experiment, we recover complete tensors by using the deflation procedure described in \cite{kileel2019subspace}. We are
mostly interested in the noisy case, since deflation with exact $a_i$'s, as
identified by \ref{prob:nspm} in the noiseless case, does not induce any additional error.

We test $4^{\text{th}}$ order tensors $(n=2)$ with $D = 40, K = 200$ and  $6^{\text{th}}$ order tensors $(n=3)$ with $D = 15, K = 200$, vary the noise level $\sigma$, and construct random tensors as in the previous experiment.
Figure \ref{fig:SPM_deflation} plots the average error over 100 repetitions of successively recovered tensor components,
where the $x$-axis ranges from the first recovered component at $1$ to the last component at index $300$. The figure illustrates
that \ref{prob:nspm} combined with modified deflation allows for recovering all tensor components
up to an error of $\CO(\sigma)$.
Surprisingly, we do not observe error propagation, despite the fact that noisy recovered tensor components
are being used within each deflation step.

\section{Conclusion}
We presented a quantitative picture
for the optimization landscape of a recent formulation \cite{kileel2019subspace} of the non-convex, high-dimensional problem of symmetric tensor decomposition.
We identified different assumptions on the tensor components $a_1, \ldots, a_K$ and bounds on the rank $K$ so that all
second-order critical points of the optimization problem \ref{prob:SPM} with sufficiently high functional value must equal one of the input tensor's CP components.
In Theorem~\ref{thm:main_result_deterministic} the assumptions were deterministic frame and low-rank conditions, while in Theorem~\ref{thm:main_result_for_now} the hypotheses were random components and an overcomplete rank.
Our proofs accommodated noise in the input tensor's entries, and we 
obtained robust results for only by analyzing the program \ref{prob:nspm}.
 Our analysis has algorithmic implications.  As the Subspace Power Method of \cite{kileel2019subspace} is guaranteed to converge to second-order critical points, by combining with analysis of deflation, it follows that SPM (with sufficient initialization) is provable under our assumptions.  
 In Theorem~\ref{thm:SPM_deflation_analysis} we gave guarantees for end-to-end  decomposition using \nolinebreak SPM.

Compared to the usual power method functional, the  novelty of the SPM functional is the de-biasing role played by the inverse of a Grammian matrix recording correlations between rank-1 tensors (recall Eq.~\eqref{eq:grammian-def}).
This Grammian matrix is responsible for many of the SPM functional's desirable properties, but it also complicated our analysis.
We showed that there are theoretical and numerical advantages in using the SPM functional for tensor decomposition over its usual power method counterpart.

This paper suggests several directions for future research:
\begin{itemize}
\item What are the average-case properties of the Grammian matrix $G_n$?  We formulated Conjecture~\ref{conj:grammians} about the minimal eigenvalue of $G_n$.  How about the other eigenvalues?
\item In the random overcomplete setting, if we assume no noise ($\Delta_{\CA}=0$) then can we dispense with the superlevel condition in  Theorem~\ref{thm:main_result_for_now}?  This would give a fully global guarantee.
\item Why do we see no error propagation when using deflation and sequential solves of \ref{prob:nspm} for CP  decomposition?
Can errors accumulate if we choose the noise deterministically?
\end{itemize}

\section*{Acknowledgements and funding disclosure}
We thank Massimo Fornasier for suggesting this collaboration in the beginning, after a fortuitous exchange in  Oaxaca, Mexico at the \textit{Computational Harmonic Analysis and Data Science} workshop (October 2019).
J.K. acknowledges partial support from start-up grants provided by the College of Natural Sciences and Oden Institute
for Computational Engineering and Sciences at UT Austin. 

\bibliographystyle{abbrv}
{\small
\bibliography{references_all.bib}

\begin{thebibliography}{10}

\bibitem{absil2009optimization}
P.-A. Absil, R.~Mahony, and R.~Sepulchre.
\newblock {\em Optimization Algorithms on Matrix Manifolds}.
\newblock Princeton University Press, 2009.

\bibitem{absil2013extrinsic}
P.-A. Absil, R.~Mahony, and J.~Trumpf.
\newblock An extrinsic look at the {R}iemannian {H}essian.
\newblock In {\em International Conference on Geometric Science of
  Information}, pages 361--368. Springer, 2013.

\bibitem{adamczak2011restricted}
R.~Adamczak, A.~E. Litvak, A.~Pajor, and N.~Tomczak-Jaegermann.
\newblock Restricted isometry property of matrices with independent columns and
  neighborly polytopes by random sampling.
\newblock {\em Constructive Approximation}, 34(1):61--88, 2011.

\bibitem{JMLR:v15:anandkumar14b}
A.~Anandkumar, R.~Ge, D.~Hsu, S.~M. Kakade, and M.~Telgarsky.
\newblock Tensor decompositions for learning latent variable models.
\newblock {\em Journal of Machine Learning Research}, 15(80):2773--2832, 2014.

\bibitem{anandkumar2017analyzing}
A.~Anandkumar, R.~Ge, and M.~Janzamin.
\newblock Analyzing tensor power method dynamics in overcomplete regime.
\newblock {\em The Journal of Machine Learning Research}, 18(1):752--791, 2017.

\bibitem{anandkumar2012method}
A.~Anandkumar, D.~Hsu, and S.~M. Kakade.
\newblock A method of moments for mixture models and hidden {M}arkov models.
\newblock In {\em Conference on Learning Theory}, pages 33--1. JMLR Workshop
  and Conference Proceedings, 2012.

\bibitem{auddy2020perturbation}
A.~Auddy and M.~Yuan.
\newblock Perturbation bounds for orthogonally decomposable tensors and their
  applications in high dimensional data analysis.
\newblock {\em arXiv preprint arXiv:2007.09024}, 2020.

\bibitem{bandeira2018notes}
A.~S. Bandeira, A.~Perry, and A.~S. Wein.
\newblock Notes on computational-to-statistical gaps: {P}redictions using
  statistical physics.
\newblock {\em Portugaliae Mathematica}, 75(2):159--186, 2018.

\bibitem{boumal2020intromanifolds}
N.~Boumal.
\newblock An introduction to optimization on smooth manifolds.
\newblock Available online, Nov 2020.

\bibitem{bryson2019marchenko}
J.~Bryson, R.~Vershynin, and H.~Zhao.
\newblock Marchenko-pastur law with relaxed independence conditions.
\newblock {\em arXiv preprint arXiv:1912.12724}, 2019.

\bibitem{chen2016perturbation}
Y.~M. Chen, X.~S. Chen, and W.~Li.
\newblock On perturbation bounds for orthogonal projections.
\newblock {\em Numerical Algorithms}, 73(2):433--444, 2016.

\bibitem{chiantini2002weakly}
L.~Chiantini and C.~Ciliberto.
\newblock Weakly defective varieties.
\newblock {\em Transactions of the American Mathematical Society},
  354(1):151--178, 2002.

\bibitem{christensen2003introduction}
O.~Christensen et~al.
\newblock {\em An Introduction to Frames and Riesz Bases}, volume~7.
\newblock Springer, 2003.

\bibitem{chu1995rank}
M.~T. Chu, R.~E. Funderlic, and G.~H. Golub.
\newblock A rank--one reduction formula and its applications to matrix
  factorizations.
\newblock {\em \textit{SIAM Review}}, 37(4):512--530, 1995.

\bibitem{comon2008symmetric}
P.~Comon, G.~Golub, L.-H. Lim, and B.~Mourrain.
\newblock Symmetric tensors and symmetric tensor rank.
\newblock {\em SIAM Journal on Matrix Analysis and Applications},
  30(3):1254--1279, 2008.

\bibitem{foobi2007}
L.~De~Lathauwer, J.~Castaing, and J.-F. Cardoso.
\newblock Fourth-order cumulant-based blind identification of underdetermined
  mixtures.
\newblock {\em \textit{IEEE Transactions on Signal Processing}},
  55(6):2965--2973, 2007.

\bibitem{fengler2019restricted}
A.~Fengler and P.~Jung.
\newblock On the restricted isometry property of centered self {K}hatri-{R}ao
  products.
\newblock {\em arXiv preprint arXiv:1905.09245}, 2019.

\bibitem{fickus2012steiner}
M.~Fickus, D.~G. Mixon, and J.~C. Tremain.
\newblock Steiner equiangular tight frames.
\newblock {\em Linear Algebra and its Applications}, 436(5):1014--1027, 2012.

\bibitem{fiedler2021stable}
C.~Fiedler, M.~Fornasier, T.~Klock, and M.~Rauchensteiner.
\newblock Stable recovery of entangled weights: Towards robust identification
  of deep neural networks from minimal samples.
\newblock {\em arXiv preprint arXiv:2101.07150}, 2021.

\bibitem{fornasier2019robust}
M.~Fornasier, T.~Klock, and M.~Rauchensteiner.
\newblock Robust and resource efficient identification of two hidden layer
  neural networks.
\newblock {\em Construction Approximation}, to appear, 2021.

\bibitem{fornasiershallow}
M.~Fornasier, J.~Vybíral, and I.~Daubechies.
\newblock {Robust and resource efficient identification of shallow neural
  networks by fewest samples}.
\newblock {\em Information and Inference: A Journal of the IMA}, 2021.

\bibitem{ge2015learning}
R.~Ge, Q.~Huang, and S.~M. Kakade.
\newblock Learning mixtures of gaussians in high dimensions.
\newblock In {\em Proceedings of the Forty-Seventh Annual ACM Symposium on
  Theory of Computing}, pages 761--770, 2015.

\bibitem{ge2020optimization}
R.~Ge and T.~Ma.
\newblock On the optimization landscape of tensor decompositions.
\newblock {\em Mathematical Programming}, pages 1--47, 2020.

\bibitem{harris2013algebraic}
J.~Harris.
\newblock {\em Algebraic Geometry: A First Course}, volume 133.
\newblock Springer Science \& Business Media, 2013.

\bibitem{hillar2013most}
C.~J. Hillar and L.-H. Lim.
\newblock Most tensor problems are {NP}-hard.
\newblock {\em \textit{Journal of the ACM (JACM)}}, 60(6):45:1--45:39, 2013.

\bibitem{janzamin2015beating}
M.~Janzamin, H.~Sedghi, and A.~Anandkumar.
\newblock Beating the perils of non-convexity: Guaranteed training of neural
  networks using tensor methods.
\newblock {\em arXiv preprint arXiv:1506.08473}, 2015.

\bibitem{kileel2019subspace}
J.~Kileel and J.~M. Pereira.
\newblock Subspace power method for symmetric tensor decomposition and
  generalized {PCA}.
\newblock {\em arXiv preprint arXiv:1912.04007}, 2019.

\bibitem{kolda2015numerical}
T.~G. Kolda.
\newblock Numerical optimization for symmetric tensor decomposition.
\newblock {\em Mathematical Programming}, 151(1):225--248, 2015.

\bibitem{kolda2009tensor}
T.~G. Kolda and B.~W. Bader.
\newblock Tensor decompositions and applications.
\newblock {\em SIAM Review}, 51(3):455--500, 2009.

\bibitem{kolda2011shifted}
T.~G. Kolda and J.~R. Mayo.
\newblock Shifted power method for computing tensor eigenpairs.
\newblock {\em SIAM Journal on Matrix Analysis and Applications},
  32(4):1095--1124, 2011.

\bibitem{kolda2014adaptive}
T.~G. Kolda and J.~R. Mayo.
\newblock An adaptive shifted power method for computing generalized tensor
  eigenpairs.
\newblock {\em \textit{SIAM Journal on Matrix Analysis and Applications}},
  35(4):1563--1581, 2014.

\bibitem{mccullagh2018tensor}
P.~McCullagh.
\newblock {\em Tensor Methods in Statistics}.
\newblock Courier Dover Publications, 2018.

\bibitem{oymak2021learning}
S.~Oymak and M.~Soltanolkotabi.
\newblock Learning a deep convolutional neural network via tensor
  decomposition.
\newblock {\em Information and Inference: A Journal of the IMA}, 2021.

\bibitem{qu2019geometric}
Q.~Qu, Y.~Zhai, X.~Li, Y.~Zhang, and Z.~Zhu.
\newblock Geometric analysis of nonconvex optimization landscapes for
  overcomplete learning.
\newblock In {\em International Conference on Learning Representations}, 2019.

\bibitem{rudelson2013hanson}
M.~Rudelson and R.~Vershynin.
\newblock Hanson-{W}right inequality and sub-{G}aussian concentration.
\newblock {\em Electronic Communications in Probability}, 18, 2013.

\bibitem{sanjabi2019does}
M.~Sanjabi, S.~Baharlouei, M.~Razaviyayn, and J.~D. Lee.
\newblock When does non-orthogonal tensor decomposition have no spurious local
  minima?
\newblock {\em arXiv preprint arXiv:1911.09815}, 2019.

\bibitem{sedghi2016provable}
H.~Sedghi, M.~Janzamin, and A.~Anandkumar.
\newblock Provable tensor methods for learning mixtures of generalized linear
  models.
\newblock In {\em Artificial Intelligence and Statistics}, pages 1223--1231.
  PMLR, 2016.

\bibitem{sherman2020estimating}
S.~Sherman and T.~G. Kolda.
\newblock Estimating higher-order moments using symmetric tensor decomposition.
\newblock {\em SIAM Journal on Matrix Analysis and Applications},
  41(3):1369--1387, 2020.

\bibitem{sidiropoulos2017tensor}
N.~D. Sidiropoulos, L.~De~Lathauwer, X.~Fu, K.~Huang, E.~E. Papalexakis, and
  C.~Faloutsos.
\newblock Tensor decomposition for signal processing and machine learning.
\newblock {\em IEEE Transactions on Signal Processing}, 65(13):3551--3582,
  2017.

\bibitem{sustik2007existence}
M.~A. Sustik, J.~A. Tropp, I.~S. Dhillon, and R.~W. Heath~Jr.
\newblock On the existence of equiangular tight frames.
\newblock {\em Linear Algebra and its Applications}, 426(2-3):619--635, 2007.

\bibitem{thompson1972principal}
R.~C. Thompson.
\newblock Principal submatrices ix: Interlacing inequalities for singular
  values of submatrices.
\newblock {\em Linear Algebra and its Applications}, 5(1):1--12, 1972.

\bibitem{vershynin2010introduction}
R.~Vershynin.
\newblock Introduction to the non-asymptotic analysis of random matrices.
\newblock {\em arXiv preprint arXiv:1011.3027}, 2010.

\bibitem{vershynin2018high}
R.~Vershynin.
\newblock {\em High-dimensional Probability: an Introduction with Applications
  in Data Science}, volume~47.
\newblock Cambridge university press, 2018.

\bibitem{vershynin2020concentration}
R.~Vershynin.
\newblock Concentration inequalities for random tensors.
\newblock {\em Bernoulli}, 26(4):3139--3162, 2020.

\bibitem{vervliet2016tensorlab}
N.~Vervliet, O.~Debals, L.~Sorber, M.~Van~Barel, and L.~De~Lathauwer.
\newblock Tensorlab 3.0. {A}vailable online.
\newblock {\em URL: http://www.tensorlab.net}, 2016.

\bibitem{weyl1912asymptotische}
H.~Weyl.
\newblock Das asymptotische verteilungsgesetz der eigenwerte linearer
  partieller differentialgleichungen (mit einer anwendung auf die theorie der
  hohlraumstrahlung).
\newblock {\em Mathematische Annalen}, 71(4):441--479, 1912.

\bibitem{white1992density}
S.~R. White.
\newblock Density matrix formulation for quantum renormalization groups.
\newblock {\em Physical Review Letters}, 69(19):2863, 1992.

\bibitem{zhong2017recovery}
K.~Zhong, Z.~Song, P.~Jain, P.~L. Bartlett, and I.~S. Dhillon.
\newblock Recovery guarantees for one-hidden-layer neural networks.
\newblock In {\em International Conference on Machine Learning}, pages
  4140--4149. PMLR, 2017.

\end{thebibliography}
}

\appendix

\renewcommand{\theequation}{S.\arabic{equation}}
\setcounter{equation}{0}
\renewcommand{\thetheorem}{S.\arabic{theorem}}
\setcounter{theorem}{0}
\renewcommand{\thealgorithm}{S.\arabic{algorithm}}

\section*{\LARGE{Supplementary material}}

\smallskip

In these appendices, we supply proofs for the  statements made in the main body of the paper. Results and equations solely within this supplementary material are labeled with a leading \nolinebreak S. 
%

\medskip

\textbf{Organization.}
\vspace{-0.5em}
\begin{itemize}[label=\textup{Section}]
    \item \ref{sec:proof_ingredients}: Overview of main proofs
    \item \ref{app:setup}: Setup and notation
    \item \ref{sec:suppl_isolated_statements}: Proofs of lemmas in Sections~\ref{sec:properties_SPM_objective} and \ref{sec:landscape_analysis}
    \item \ref{sec:suppl_preparatory_material}: Derivation of Riemannian derivatives and optimality conditions for \ref{prob:nspm}
    \item \ref{sec:uniform_ensembles}: Preparatory results on random ensembles over the sphere and Conjecture~\ref{conj:grammians}
    \item \ref{subsec:prep_material_proofs_main}: Technical tools for the proofs of Theorems \ref{thm:main_result_deterministic} and
     \ref{thm:main_result_for_now}
    \item \ref{subsec:proof_theorem_main_deterministic}: Proof of Theorem \ref{thm:main_result_deterministic}
    \item \ref{subsec:proof_theorem_main_overcomplete}: Proof of Theorem \ref{thm:main_result_for_now}
    \item \ref{sec:deflation}: Deflation bounds and proof of Theorem~\ref{thm:SPM_deflation_analysis} 
\end{itemize}

\smallskip

\section{Overview of main proofs}
\label{sec:proof_ingredients}
Before commencing with the exact details, let us first give a high-level overview of the main steps in the proofs of Theorems \ref{thm:main_result_deterministic} and \ref{thm:main_result_for_now}.
We recall the definition $A := [a_1|\ldots|a_K] \in \bbR^{D\times K}$.
The proofs consist of three steps. We begin by plugging in $y = a_i$, where $i$ is the index such that $|\langle a_i, x\rangle| = \|A^\top x\|_{\infty}$, in  \eqref{eq:derived_second_order_optimality}. Then using either incoherence constants $\rho_s$ and $\rho_n$ for Theorem \ref{thm:main_result_deterministic}, or using Assumptions \ref{enum:RIP}-\ref{enum:GInverse} for Theorem \ref{thm:main_result_for_now},
a suitable rearrangement of \eqref{eq:derived_second_order_optimality} immediately gives
\begin{align}
\label{eq:coefficient_inequality}
\|A^\top x\|_{\infty}^2 \geq  \begin{cases}
1- 2\rho_2 - 2\frac{\rho_n}{1-2\rho_n} - 3\frac{\Delta_{\CA}}{F_{\CA}(x)},& \textrm{for Theorem \ref{thm:main_result_deterministic}},\\
1- 2\delta - \bar C\left(\frac{\varepsilon_K}{F_{\CA}(x)}\right)^{\!\frac{1}{n}} - 3\frac{\Delta_{\CA}}{F_{\CA}(x)},& \textrm{for Theorem \ref{thm:main_result_for_now}},
\end{cases}
\end{align}
where the constant $\bar C$ in the second case depends on $n$, $c_1$ and $c_2$.
For Theorem \ref{thm:main_result_deterministic}, this immediately
suggests the largest correlation coefficient satisfies $\|A^\top x\|_{\infty} \geq 1- \CO(\rho_2+\rho_n + \Delta_{\CA})$,
provided $F_{\CA}(x) \gg \Delta_{\CA}$ as assumed in the statement. In the case of Theorem \ref{thm:main_result_for_now}, and
within the level set $F_{\hat \CA}(x) \geq C\varepsilon_K + 5 \Delta_{\CA}$, \eqref{eq:coefficient_inequality} first implies
that $\|A^\top x\|_{\infty}$ is bounded from below by a constant that is independent of $D,K$.
However, by $F_{\CA}(x) \geq  \|A^\top x\|_{\infty}^{2n}$, we can update the
lower bound on $F_{\CA}(x)$ and re-use it again in \eqref{eq:derived_second_order_optimality}. The resulting bootstrap argument implies
$\|A^\top x\|_{\infty} \geq 1 - \CO(\varepsilon_K^{1/n} + \Delta_{\CA})$, i.e., there
must be at least one large correlation coefficient.

The second step is to show that the Riemannian Hessian $\nabla^2_{\bbS^{D-1}}F_{\hat \CA}$
is negative definite on spherical caps around the tensor components $\pm a_i$.
We prove, uniformly for all unit norm $z \perp x$, the bound
\begin{align}
\label{eq:strict_concavity}
\frac{1}{2n}z^\top \nabla_{\bbS^{D-1}}^2 \hat f(x) z \leq \begin{cases}
-1+ n\frac{2-\rho_n}{1-\rho_n}\eta + \frac{3n-2}{1-\rho_n}\rho_n + 4\Delta_{\CA},&  \textrm{for Theorem \ref{thm:main_result_deterministic}},\smallskip\\
-1 +  \bar C\sqrt{\eta} + \tilde C \varepsilon_K^{\frac{n-1}{2n}} + 4\Delta_{\CA},&  \textrm{for Theorem \ref{thm:main_result_for_now}},
\end{cases}
\end{align}
where $\eta := 1-\max_{i \in [K]}\langle a_i, x\rangle^{2n}$, and $\bar C,\tilde C$ depend only on $n,c_1,c_2$ and $c_{\delta}$. 
Hence, the Riemannian Hessian is negative definite and $F_{\hat \CA}$ is strictly concave in spherical caps around $\pm a_i$.

In a third step, we consider a spherical cap $B_{sa_i}$ ($s\in \{-1,1\}$) around $s a_i$.
By compactness, a local maximizer $x^*$ of $F_{\hat \CA}$ must exist within the cap $B_{sa_i}$
and by strict concavity it must be unique. Moreover, its objective is bounded from below via $F_{\hat \CA}(x^*)\geq F_{\hat \CA}(a_i) \geq 1-\Delta_{\CA}$.
Using a Taylor expansion of the function $g(t) = F_{\hat \CA}(\gamma(t))$ where $\gamma$ is the geodesic path from $x^*$ to $a_i$,
and combining this with the upper bound in \eqref{eq:strict_concavity}, we obtain $\|x^*-sa_i\|_2^2 \leq 2\Delta_{\CA}/n$.
Since we can start the argument with a spherical cap which is strictly larger than $\{y \in \bbS^{D-1} : \langle y, sa_i\rangle \geq 1-\Delta_{\CA}/n\}$, it follows that $x^*$ is not on the boundary
of $B_{sa_i}$, and thus $x^*$ is a local maximizer of \ref{prob:nspm}. 

We complete the proof by noting that all local maximizers of \ref{prob:nspm} contained
in the superlevel set, which is specified in the respective theorem, satisfy
$\|A^\top x\|_{\infty} \geq 1 - \epsilon$, where $\epsilon$ depends on $\Delta_{\CA}$ and $\rho_2, \rho_n$, or $\varepsilon_K$.
In particular, they are contained in spherical caps $B_{s a_i}$ constructed in the third step.

\medskip

\section{Setup and notation} \label{app:setup}

\textbf{Additional notation.} Besides the notation in Section~\ref{sec:notation}, we will use the following in the rest of the supplementary materials.
\begin{itemize}\setlength\itemsep{0.5em}
\item For tensors $S, T$, we abbreviate their tensor product via concatenation, $ST := S \otimes T$, and the tensor power by \nolinebreak $S^{n} := S^{\otimes n}$. (It will be clear from context when concatenation denotes matrix multiplication instead.)
\item We use $\operatorname{Vec}(T)$ to denote the vectorization of a tensor $T \in \CT_D^n$, i.e., $\operatorname{Vec}(T) := \operatorname{Reshape}(T, D^n)$.
Also, $\operatorname{Tensor}(-) := \operatorname{Reshape}(-, [D, \ldots, D])$ indicates the inverse mapping, which tensorizes a vector.

\item We denote the (columnwise) Khatri-Rao power by a bullet: if $a \in \mathbb{R}^D$ then $a^{\bullet n} := \operatorname{Vec}(a^{\otimes n}) \in \mathbb{R}^{D^n}$.
If $A = [a_1 | \ldots | a_K ] \in \mathbb{R}^{D \times K}$ then  $A^{\bullet n} := [ a_1^{\bullet n}|\ldots| a_K^{\bullet n}] \in \bbR^{D^n\times K}$.
\item We write the Tucker tensor product as follows:
 for $T \in \CT_D^n$ and $M^{(1)}, \ldots, M^{(n)} \in \mathbb{R}^{E \times D}$, the Tucker product is denoted $T \times_1 M^{(1)} \times_2 \ldots \times_n M^{(n)} \in \CT_{E}^n$ and defined by
\begin{equation*}
\left(T \times_1 M^{(1)} \times_2 \ldots \times_n M^{(n)}\right)_{i_1, \ldots, i_n} := \sum_{j_1=1}^{D} \ldots \sum_{j_n=1}^{D} T_{j_1, \ldots, j_n} M^{(1)}_{i_1, j_1} \ldots M^{(n)}_{i_n,j_n}.
\end{equation*}
\end{itemize}

\smallskip

The next definition serves to collect together the key quantities appearing later in our proofs.

\smallskip

\begin{definition} \label{def:basic-setup} 
   Let $\{ a_i : i \in [K]\} \subseteq \mathbb{S}^{D-1}$ be a given set of unit-norm vectors.  Write $A = [a_1 | \ldots | a_K] \in \mathbb{R}^{D \times K}$.  Consider the corresponding set of tensors $\{a_i^n : i \in [K]\}$, and the subspace $\mathcal{A} = \operatorname{Span}\{a_i^n : i \in [K]\} \subseteq \symten{n}$.
   Let $P_{\CA} : \CT^n_D \rightarrow \CA$ denote orthogonal projection onto $\CA$.
   Write $A^{\bullet n} := [\operatorname{Vec}(a_1^{n}) | \ldots | \operatorname{Vec}(a_K^{n})] \in \mathbb{R}^{D^n \times K}$ for the $n$-th columnwise Khatri-Rao power of $A$.
   \begin{enumerate}\setlength\itemsep{1pt}
       \item Define the $n$-th \textit{Grammian matrix} to be  $G_n = (A^{\bullet n})^{\top} A^{\bullet n} \in \Sym(\CT_K^2)$.  Thus
        \begin{align}
       (G_n)_{ij} := \langle a_i^n, a_j^n \rangle = \langle a_i, a_j \rangle^n \quad \quad \forall \, i, j \in [K].
          \end{align}
       \item Let $x \in \mathbb{S}^{D-1}$.  Define the \textit{correlation coefficients} of $x$ with respect to $\{a_i : i \in [K]\}$ to be $\zeta = \zeta(x) := A^{\top}x  \in \mathbb{R}^K$.  Thus
   \begin{align}
       \zeta_i := \langle x, a_i \rangle  \quad \quad \forall \, i \in [K].
   \end{align}
       \item Let $x, y \in \mathbb{S}^{D-1}$ and $s \in \{0, 1, \ldots, n\}$.  Define the \textit{correlation coefficients} of $x^s y^{n-s}$ with respect to $\{a_i^n : i \in [K]\}$ to be $\eta = \eta(x^s y^{n-s}) = (A^{\bullet n})^{\top} \operatorname{Vec}(x^{s}y^{n-s}) \in \mathbb{R}^K$.  Thus
       \begin{align}
           \eta_i := \langle x^s y^{n-s}, a_i^n \rangle = \langle x, a_i \rangle^s \langle y, a_i \rangle^{n-s} \quad \quad \forall \, i \in [K].
       \end{align}
       \item Let $x \in \mathbb{S}^{D-1}$.  Define the \textit{expansion coefficients} of $P_{\CA}(x^n)$ with respect to $\{a_i^n : i \in [K]\}$ to be $\sigma = \sigma(x) \in \mathbb{R}^K$ so that the following expansion holds
       \begin{equation} \label{eq:def-sigma}
       P_{\CA}(x^n) =: \sum_{i=1}^K \sigma_i a_i^n.
       \end{equation}
       (If $a_1^n, \ldots, a_K^n$ are linearly independent, then Eq.~\eqref{eq:def-sigma} uniquely defines $\sigma$; otherwise, we fix once and for all a choice of $\sigma$ so that \eqref{eq:def-sigma} holds.)
       \item Let $x \in \mathbb{S}^{D-1}$.  For $i \in [K]$ and $s \in \{0, \ldots, n\}$,
       define the \textit{remainder term} for $P_{\CA}(x^n)$ with respect to $x^{n-s}a_i^s$ to be $R_{i,s} = R_{i,s}(x) \in \mathbb{R}$ so that the following equation holds
       \begin{equation} \label{eq:remainder}
       \langle P_{\CA}(x^n), x^{n-s}a_i^s \rangle =: \sigma_i \zeta_i^{n-s} + R_{i,s}.
       \end{equation}
       Last, define the maximal $s$-th \textit{remainder size} to be $R_s = R_s(x) :=  \max_{i \in [K]} |R_{i,s}| \in \mathbb{R}$.
   \end{enumerate}
\end{definition}

\smallskip

\begin{remark}
Arguably all the quantities in Definition~\ref{def:basic-setup} are natural to consider, with the possible exception of $R_{i,s}$ and $R_s$.  In fact these terms play a key role in our proofs; see inequality \eqref{eq:starting_point_max_correlation_coeff} in Proposition~\ref{prop:auxliary_one_big}.  We call $R_{i,s}$ a ``remainder" because, under the assumptions of Theorems~\ref{thm:main_result_deterministic} and \ref{thm:main_result_for_now}, as we will see it holds that $\langle P_{\CA}(x^n), x^{n-s} a_i^s \rangle \approx \sigma_i \zeta_i^{n-s}$, whence $R_{i,s}$ and $R_s$ are small.
\end{remark}

\medskip

\section{Proofs of lemmas in Sections~\ref{sec:properties_SPM_objective} and \ref{sec:landscape_analysis}} 
\label{sec:suppl_isolated_statements}
In this section, we prove the various isolated results stated in Sections~\ref{sec:properties_SPM_objective} and \ref{sec:landscape_analysis} of the main body.

\subsection{Proof of Lemma~\ref{lem:error_in_subspace}}

\begin{proof}
	We note that \eqref{eq:subspace_bound} is valid for any matrices $W, \hat W \in \bbR^{p \times q}$, such that the rank of $W$ is $r$. That is, defining $\Delta_W:= \|W-\hat W\|_2<\sigma_r(W)$ and $\Im(W), \Im_r(\hat W)$ as in the statement, we have
	\begin{align}
		\label{eq:subspace_bound_supplementary_W}
		\N{P_{\Im(W)}-P_{\Im_r(\hat W)}}_2 \leq \frac{\Delta_W}{\sigma_r(W) - \Delta_W}.
	\end{align}
	We first show the result if $W, \hat W$ are symmetric matrices. In that case, we denote the eigenvalue decomposition of $W,\hat W$ as
	\begin{align*}
		W = \begin{pmatrix}
			V_1 & V_2
		\end{pmatrix}
		\begin{pmatrix}
			\Lambda_1 & 0 \\
			0 & \Lambda_2
		\end{pmatrix}
		\begin{pmatrix}
			V_1^\top \\
			V_2^\top
		\end{pmatrix}
		\quad \textrm{ and }\quad
		\hat W = \begin{pmatrix}
			\hat V_1 & \hat V_2
		\end{pmatrix}
		\begin{pmatrix}
			\hat \Lambda_1 & 0 \\
			0 & \hat \Lambda_2
		\end{pmatrix}
		\begin{pmatrix}
			\hat V_1^\top \\
			\hat V_2^\top
		\end{pmatrix},
	\end{align*}
	where $\Lambda_1$ and $\hat\Lambda_1$ are diagonal matrices with the largest $K$ eigenvalues, in magnitude, of $W$ and $\hat W$ on the diagonal, respectively. Letting $\Sigma_1$ be a diagonal matrix with the singular values of $W$, we have $|\Lambda_1| = \Sigma_1$, and the same holds for $\hat W$.
	From the eigenvector decomposition, we can write $P_{\Im(W)} = V_1 V_1^\top$ and $P_{\Im_r(\hat W)} = \hat V_1\hat V_1^\top$. Lemma 2.3 in \cite{chen2016perturbation} implies that
	$$\N{P_{\Im(W)} - P_{\Im_r(\hat W)}}_2 = \|V_1^\top \hat V_2\|_2.$$
	Then \eqref{eq:subspace_bound_supplementary_W} holding for symmetric matrices follows from
	\begin{align*}
	\Delta_W &= \|W - \hat W\|_2 \ge \|V_1^\top (W - \hat W) \hat V_2\|_2 = \|\Lambda_1 V_1^\top \hat V_2 - V_1^\top \hat V_2 \hat \Lambda_2\|\\
	&\ge \|\Lambda_1 V_1^\top \hat V_2\|_2 - \|V_1^\top \hat V_2 \hat \Lambda_2\|_2
	\ge	\sigma_r(W)\|V_1^\top \hat V_2\|_2 - \|\hat \Lambda_2\|_2 \|V_1^\top \hat V_2\|_2\\
	&\ge (\sigma_r(W) -\Delta_W)\N{P_{\Im(W)} - P_{\Im_r(\hat W)}}_2.
	\end{align*}
	Here we used the following facts:
	\begin{itemize}
		\item $\|W - \hat W\|_2 \ge \|V_1^\top (W - \hat W) \hat V_2\|_2$ since $V_1$ and $V_2$ have orthonormal columns.
		\item Since $\hat V_2^\top V_1 (\Lambda_1^2 - \sigma_r(W)^2 \Id_r) V_1^\top \hat V_2$ is positive semi-definite, we have
		$$\|\Lambda_1 V_1^\top \hat V_2\|_2 = \sqrt{\|\hat V_2^\top V_1 \Lambda_1^2 V_1^\top \hat V_2\|_2}\ge \sigma_r(W) \sqrt{\|\hat V_2^\top V_1 V_1^\top \hat V_2\|_2} = \sigma_r(W) \|V_1^\top \hat V_2\|_2.$$
		\item By Weyl's inequality \cite{weyl1912asymptotische}, it holds $\|\hat \Lambda_2\|\le \Delta_W$.
	\end{itemize}

	Now, if $W, \hat W$ are not symmetric, we apply the result for the symmetric matrices
	\begin{align*}
	H =	\begin{pmatrix}
		0 & W \\ W^\top & 0
	\end{pmatrix}
	\quad \textrm{ and }\quad
	\hat H = \begin{pmatrix}
			0 & \hat W \\ \hat W^\top & 0
		\end{pmatrix}.
	\end{align*}
	Denote the singular vector decomposition of $W$ and $\hat W$ by $W = U \Sigma V^\top$ and $\hat W = \hat U \hat \Sigma \hat V^\top$. We can further split $U= (U_1, U_2)$, $\Sigma = \operatorname{Blockdiag}(\Sigma_1, \Sigma_2)$ and $V = (U_1, U_2)$, where $U_1, \Sigma_1, V_1$ correspond to the largest $K$ singular values of $W$, and split $\hat U, \hat \Sigma, \hat V$ analogously. The eigen-decomposition of $H$ is related to the singular vector decomposition of $W$ as follows
	\begin{equation}\label{eq:Heigdec}
		H = \begin{pmatrix}
			\frac{1}{\sqrt2} U & \frac{1}{\sqrt2} U\\
			\frac{1}{\sqrt2} V & -\frac{1}{\sqrt2} V
		\end{pmatrix}
		\begin{pmatrix}
			\Sigma & 0 \\
			0 & -\Sigma
		\end{pmatrix}
		\begin{pmatrix}
			\frac{1}{\sqrt2} U^\top & \frac{1}{\sqrt2} V^\top\\
			\frac{1}{\sqrt2} U^\top & -\frac{1}{\sqrt2} V^\top
		\end{pmatrix},
	\end{equation}
	and an analogous relation holds between $\hat H$ and $\hat W$. It can be checked that this is an eigen-decomposition by using that $W = U \Sigma V^\top$ to show that the RHS in \eqref{eq:Heigdec} is equal to $H$, and use that $U$ and $V$ have orthonormal columns to show$$\begin{pmatrix}
		\frac{1}{\sqrt2} U & \frac{1}{\sqrt2} U\\
		\frac{1}{\sqrt2} V & -\frac{1}{\sqrt2} V
	\end{pmatrix}^\top\begin{pmatrix}
	\frac{1}{\sqrt2} U & \frac{1}{\sqrt2} U\\
	\frac{1}{\sqrt2} V & -\frac{1}{\sqrt2} V
\end{pmatrix} = \Id_{2 r}.$$
Thus, if $W$ has rank $r$, $H$ has rank $2r$, and since $H$ is symmetric,
	\begin{equation}
		\label{eq:subspace_bound_supplementary_H}
		\N{P_{\Im(H)}-P_{\Im_{2r}(\hat H)}}_2 \leq \frac{\Delta_H}{\sigma_{2r}(H) - \Delta_H}
	\end{equation}
	Regarding the right-hand side, denote $\delta W = \hat W - W$. Equation  \eqref{eq:Heigdec} implies $\sigma_{2r}(H) = |\lambda_{2r}(H)| = \sigma_{r}(W)$, while
	\begin{align*}
	 \Delta_H &= \|H - \hat H\|_2 = \N{\begin{pmatrix}
	 		0 & \delta W \\ \delta W^\top & 0
	 \end{pmatrix}}_2 = \N{\begin{pmatrix}
	 0 & \delta W \\ \delta W^\top & 0
 \end{pmatrix}^2}_2^{\frac12}\\
	&= \N{\begin{pmatrix}
			\delta W \delta W^\top & 0 \\ 0 & \delta W^\top \delta W
		\end{pmatrix}}_2^{\frac12} = \max\{\|\delta W \delta W^\top\|_2, \|\delta W^\top \delta W\|_2\}^\frac12\\
	&= \|\delta W \|_2 = \Delta_W
	\end{align*}
	On the left-hand side, we have
	\begin{align*}
	P_{\Im(H)} &=\begin{pmatrix}
				\frac{1}{\sqrt2} U_1 & \frac{1}{\sqrt2} U_1\\
				\frac{1}{\sqrt2} V_1 & -\frac{1}{\sqrt2} V_1
			\end{pmatrix}
			\begin{pmatrix}
				\frac{1}{\sqrt2} U_1^\top & \frac{1}{\sqrt2} V_1^\top\\
				\frac{1}{\sqrt2} U_1^\top & -\frac{1}{\sqrt2} V_1^\top
			\end{pmatrix}
		= \begin{pmatrix}
			U_1 U_1^\top & 0\\
			0 & V_1 V_1^\top
		\end{pmatrix},
	\end{align*}
	thus
	\begin{align*}
	\N{P_{\Im(H)} - P_{\Im_{2r}(\hat H)}}_2 &= \N{\begin{pmatrix}
			U_1 U_1^\top - \hat U_1 \hat U_1^\top& 0\\
			0 & V_1 V_1^\top - \hat V_1 \hat V_1^\top
		\end{pmatrix}}_2 \\
	&= \max\left\{\N{P_{\Im(W)} - P_{\Im_{r}(\hat W)}}_2, \N{P_{\Im(W^\top)} - P_{\Im_{r}(\hat W^\top)}}_2\right\}
	\end{align*}
	Replacing these in \eqref{eq:subspace_bound_supplementary_H}, we show \eqref{eq:subspace_bound_supplementary_W} when $W, \hat W$ are not symmetric.
	$$	\N{P_{\Im(W)} - P_{\Im_{r}(\hat W)}}_2\le \N{P_{\Im(H)} - P_{\Im_{2r}(\hat H)}}_2 \le \frac{\Delta_H}{\sigma_{2r}(H) - \Delta_H}= \frac{\Delta_W}{\sigma_{r}(W) - \Delta_W}.$$
	The bound for $\Delta_{\CA}$ follows from
	\begin{align*}
		\Delta_{\CA} &= \sup\limits_{T \in \CT_{D}^{n},\ \N{T}_F = 1}\N{P_{\CA}(T)-P_{\hat \CA}(T)}_F =
		\sup\limits_{u \in \bbR^{D^n},\ \N{u}_2 = 1}\N{P_{\Im(M)}(u)-P_{\Im_K(\hat M)}(u)}_2 \\
		&=\N{P_{\Im(M)}-P_{\Im_K(\hat M)}}_2 . \qedhere
	\end{align*}
\end{proof}

\smallskip

\subsection{Proof of Lemma~\ref{lem:reformulation_frob_norms}}

\begin{proof}
This is the case $n = s$ in \eqref{eq:norm_identity_supplementary} in Lemma \ref{lem:frob_norms_extended} below, as $\eta(x^n) = \zeta(x)^{\odot n} = (A^{\top}x)^{\odot n}$.
\end{proof}

\begin{lemma}\label{lem:frob_norms_extended}
Assume $\{a_i^n : i \in [K]\} \subseteq \symten{n}$ is linearly independent.
Then in the setup of Definition~\ref{def:basic-setup},  we have the following formulas in terms of the inverse Grammian: 
\begin{align}
& \N{P_{\CA}(x^s y^{n-s})}_F^2 = \eta(x^sy^{n-s})^\top G_n^{-1}\eta(x^sy^{n-s}) \quad \text{ for all } x, y \in \mathbb{S}^{D-1}, \label{eq:norm_identity_supplementary}
 \\[0.5em]
&\text{and } \quad \sigma(x) = G_n^{-1} \eta (x^n) = G_n^{-1} (\zeta(x)^{\odot n}) \quad  \text{for all } x \in \mathbb{S}^{D-1}.  \label{eq:coefficient_identity}
\end{align}
\end{lemma}
\begin{proof}
Let $A^{\bullet n} := [\operatorname{Vec} (a_1^{\otimes n})|\ldots|\operatorname{Vec}(a_K^{\otimes n})] \in \bbR^{D^n\times K}$.
Up to reshapings, $P_{\CA}$ is projection onto the column space of $A^{\bullet n}$.
Since $A^{\bullet n}$ has full column rank by assumption, projection onto the column space of  $A^{\bullet n}$ is represented by the matrix
$$
A^{\bullet n} \left( (A^{\bullet n})^{\top} A^{\bullet n} \right)^{-1} (A^{\bullet n})^{\top} = A^{\bullet n} G_n^{-1} (A^{\bullet n})^{\top}.
$$
It follows that for $T \in \CT_D^n$, we have
\begin{align}
P_{\CA}(T) &= \operatorname{Tensor}\left(A^{\bullet n} G_n^{-1} (A^{\bullet n})^{\top} \operatorname{Vec}(T) \right), \label{eq:Gn-1_first} \\
\N{P_{\CA}(T)}_F^2 &= \left\langle P_{\CA}(T), T \right\rangle = \operatorname{Vec}(T)^{\top}A^{\bullet n} G_n^{-1} (A^{\bullet n})^{\top} \operatorname{Vec}(T).  \label{eq:Gn-1_second}
\end{align}
Substituting $T = x^s y^{n-s}$ in \eqref{eq:Gn-1_first}
and using $(A^{\bullet n})^{\top} \operatorname{Vec}(T) = \eta(x^s y^{n-s})$ yields \eqref{eq:norm_identity_supplementary},
while substituting $T = x^n$ into \eqref{eq:Gn-1_second} gives \eqref{eq:coefficient_identity}.  This completes the calculation.
\end{proof}

\medskip

\subsection{Proof of~Lemma \ref{pm_vs_spm}}

\begin{proof}
Write $\zeta := \zeta(x) = A^\top x$.
From $\langle a_i^n, a_j^n \rangle = \langle a_i,a_j\rangle^n = \rho$ whenever $i \neq j$ and $\langle a_i^n, a_i^n \rangle = \langle a_i, a_i \rangle^n = 1$ for each $i$, we can write the Grammian of $a_1^n, \ldots, a_K^n$ as
\begin{equation}
G_n = (1-\rho)\Id_K + \rho \mathbbm{1}_K \otimes \mathbbm{1}_K,
\end{equation}
where $\mathbbm{1}_K \in \mathbb{R}^D$ denotes the all-ones vector.
By the Sherman-Morrison inversion formula,
\begin{align*}
G_n^{-1} &= \left((1-\rho)\Id_K + \rho \mathbbm{1}_{K}\otimes \mathbbm{1}_K\right)^{-1}
= \frac{1}{1-\rho}\left(\Id_K + \frac{\rho}{1-\rho}\mathbbm{1}_{K}\otimes \mathbbm{1}_K\right)^{-1}\\
&= \frac{1}{1-\rho}\left(\Id_K - \frac{\frac{\rho}{1-\rho}}{1+K\frac{\rho}{1-\rho}}\mathbbm{1}_{K}\otimes \mathbbm{1}_K\right)
=\frac{1}{1-\rho} \Id_K - \frac{\rho}{(1-\rho)^2+K\rho(1-\rho)} \mathbbm{1}_K \otimes \mathbbm{1}_K.
\end{align*}
Using that $\sum_{i \in [K]}\langle x,a_i\rangle^n = M$ uniformly on $\bbS^{D-1}$, we have $\langle \mathbbm{1}_K, \zeta(x)^{\odot n} \rangle = M$.
By writing $F_{\CA}$ as in Lemma \ref{lem:reformulation_frob_norms}, we can complete the proof:
\begin{align*}
F_{\CA}(x) = \zeta^{\odot n} G_n^{-1}\zeta^{\odot n} &= \frac{1}{1-\rho}\N{\zeta}_{2n}^{2n}-\frac{\rho}{(1-\rho)^2+K\rho(1-\rho)}\left(\sum_{i=1}^{K}\langle a_i, x\rangle^n\right)^2 \\
& =\frac{1}{1-\rho}\N{\zeta}_{2n}^{2n}-\frac{\rho M^2}{(1-\rho)^2+K\rho(1-\rho)}. \hfill \qedhere
\end{align*}
\end{proof}

\smallskip

\subsection{Proof of Lemma~\ref{lem:incoherence_scalars}}

\begin{proof}
By Gershgorin's circle theorem and the fact $\N{a_i}_2 = 1$, the eigenvalue $\mu_{\ell}(G_{s})$ for each  $\ell=1, \ldots, K$ adheres to
\begin{align*}
|\mu_\ell(G_{s}) - 1| \leq \sum_{\substack{i: i\neq \ell}} \SN{\langle a_\ell, a_i\rangle}^{s} \leq \rho_s,\quad \textrm{hence}\quad \mu_1(G_s) \leq 1+\rho_s\quad \textrm{and}\quad
\mu_K(G_s) \geq 1-\rho_s.
\end{align*}
First, we suppose that $s$ is even.
By a classic result in frame theory \cite[Prop.~3.6.7]{christensen2003introduction}, the extremal eigenvalues of the Grammian $G_{s/2}$ satisfy
\begin{align} \label{eq:classic-frame}
\mu_K(G_{s/2}) \leq \sum_{i=1}^{K}\langle T, a_i^{s/2}\rangle^2\leq \mu_1(G_{s/2}), \quad \textrm{for all unit-norm } T \in \Span{a_i^{s/2} : i \in [K]}\!.
\end{align}
The right-most inequality of \eqref{eq:classic-frame} implies
\begin{equation} \label{eq:from-classic-frame}
    \sum_{i=1}^{K}\langle T, a_i^{s/2}\rangle^2 = \sum_{i=1}^K \left\langle P_{\Span{a_i^{s/2} : i \in [K]}} (T), a_i^{s/2} \right\rangle^2 \leq \mu_1(G_{s/2}) \quad \textrm{for all unit-norm } T \in \CT_D^{s/2}\!.
\end{equation}
Substituting $T = x^{s/2}$ into \eqref{eq:from-classic-frame} and using that $s$ is even, we obtain
\begin{equation*}
\sum_{i=1}^K |\langle x, a_i \rangle|^s = \sum_{i=1}^K \langle x^{s/2}, a_i^{s/2} \rangle^2 \leq \mu_1(G_{s/2}) \quad \textrm{for all }x \in \mathbb{S}^{D-1},
\end{equation*}
whence $\rho_s \leq \mu_1(G_{s/2}) - 1$. 
If instead $s$ is odd, we may similarly derive the upper bound $\rho_s \leq \mu_1(G_{\lfloor s/2\rfloor})$,
so that we have the following relation in general
\begin{equation*}
1-\rho_s \leq \mu_K(G_s) \leq \mu_1(G_s) \leq 1+\rho_s \leq \mu_1(G_{\lfloor s/2\rfloor}).  \hfill \qedhere
\end{equation*}
\end{proof}

\medskip

\subsection{Proof of Lemma~\ref{lem:perturbations_spurious_local_max}}

\begin{proof}
This argument relies on Lemma~\ref{lem:riemmanian_derivatives}, which is proven independently in the next section.
Choose $b \in \bbS^{D-1}$ with $b \perp a$.  Define $S := \sqrt{1-\delta^2} a^n - \delta b^n \in \symten{n}$, and
consider the corresponding subspace $\hat \CA := \Span{S}$.

For each $T \in \CT_D^n$ with
$\N{T}_F = 1$, using $\|a^n\|_F = \|b^n\|_F =1 $, $\langle a^n, b^n \rangle = 0$ and $\|S\|_F=1$,  
\begin{align*}
    \N{P_{\CA}(T)-P_{\hat \CA}(T)}_F^2
    &= \N{\langle a^n, T \rangle a^n - \langle \sqrt{1-\delta^2} a^n - \delta b^n, T \rangle (\sqrt{1-\delta^2} a^n - \delta b^n)}_F^2 \\
    &= \N{\langle \delta^2 a^n + \delta \sqrt{1 - \delta^2} b^n, T\rangle a^n + \langle \delta \sqrt{1 -\delta^2} a^n - \delta^2 b^n, T\rangle b^n}_F^2 \\
    & = \langle \delta^2 a^n + \delta \sqrt{1 - \delta^2} b^n, T\rangle^2 + \langle \delta \sqrt{1 -\delta^2} a^n - \delta^2 b^n, T\rangle^2 \\
    & = \delta^2 \langle a^n, T \rangle^2 + \delta^2 \langle b^n, T \rangle^2.
\end{align*}
This quadratic is maximized over the unit sphere in $\CT_D^n$ at any unit-norm $T \in \Span{a^n, b^n}$, where its value is $\delta^2$.
It follows $\N{P_{\mathcal{A}} - P_{\hat{\mathcal{A}}
}}_{F \rightarrow F} = \delta$.

Also, we have
\begin{equation*}
F_{\hat{\mathcal{A}}}(b) = \N{\langle S, b^n \rangle S}_F^2 = \langle S, b^n \rangle^2 = \delta^2.
\end{equation*}

We now verify that $b$ is a local maximizer of $F_{\hat \CA}$ by checking the optimality conditions.
By \eqref{eq:riemannian_gradient},
\begin{align*}
    \nabla_{\mathbb{S}^{D-1}}F_{\hat{\mathcal{A}}}(b)
    &= 2n P_{\hat{\mathcal{A}}}(b^n) \cdot b^{n-1} - 2nF_{\hat{\mathcal{A}}}(b)b \\
    &= 2n \langle S, b^n \rangle S \cdot b^{n-1} - 2n \delta^2n \\
    &= 2n(-\delta)(-\delta b) - 2n\delta^2 b \\
    & = 0,
\end{align*}
which shows that $b$ is a stationary point of $F_{\hat{\mathcal{A}}}$.
Now taking any unit norm $z \perp b$ and using \eqref{eq:riemannian_hessian},
\begin{align*}
    z^{\top} \nabla^2_{\mathbb{S}^{D-1}} F_{\hat{\mathcal{A}}}(b) z
    &= 2n^2 \N{P_{\hat{\mathcal{A}}}(b^{n-1}z)}_F^2 + 2n(n-1)\langle P_{\hat{\mathcal{A}}}(b^n), b^{n-2}z^2  \rangle - 2n F_{\hat{\mathcal{A}}}(b) \\
    &= 2n^2 \N{\langle S, b^{n-1}z \rangle S}_F^2 + 2n(n-1) \langle \langle S, b^n \rangle S, b^{n-2}z^2 \rangle - 2n \delta^2 \\
   &= -2n(n-1) \delta \sqrt{1 - \delta^2} \langle a, b \rangle^{n-2} \langle a, z \rangle^2 - 2n\delta^2 \\
   &\leq -2n \delta^2 \\
   &< 0,
\end{align*}
where we used the fact $n \geq 2$ in third equality.
Thus, the Riemannian Hessian of $F_{\hat{\mathcal{A}}}$ is strictly negative-definite at $b$.
Therefore, $b$ is a strict local maximizer of $F_{\hat{\mathcal{A}}}$ as we wanted.
\end{proof}

\medskip

\section{Derivation of Riemannian derivatives and optimality conditions for \ref{prob:nspm}}
\label{sec:suppl_preparatory_material}
In this section, we derive the Riemannian derivatives and optimality conditions for \ref{prob:nspm}.

\subsection[Riemmanian derivatives of nSPM-P on the sphere]{Riemannian derivatives of $F_{\hat \CA}$ on the sphere}

\begin{lemma}
	\label{lem:riemmanian_derivatives}
	Let $x,z \in \bbS^{D-1}$ with $z \perp x$.
	The Riemannian gradient and Hessian of $F_{\hat \CA}$ satisfy
	\begin{align}
		\label{eq:riemannian_gradient}
		\nabla_{\bbS^{D-1}}F_{\hat \CA}(x) &= 2n P_{\hat \CA}(x^{ n}) \cdot x^{ n-1} - 2nF_{\hat \CA}(x)x,\\
		\label{eq:riemannian_hessian}
		z^\top \nabla_{\bbS^{D-1}}^2 F_{\hat \CA}(x) z &=2n^2 \|P_{\hat \CA}(x^{ n-1} z)\|_F^2 + 2n(n-1)\langle P_{\hat \CA}(x^n), x^{ n-2} z^{ 2}\rangle - 2nF_{\hat \CA}(x).
	\end{align}
\end{lemma}
\begin{remark}
In this lemma $z$ is a \textup{tangent vector} to the unit sphere at $x$, since $z \perp x$.
For convenience we work with \textup{normalized} tangent vectors, which is why there is the further assumption that $\|z \|_2 =1$.
This is without loss of generality:
second-order criticality for \ref{prob:nspm} is the condition that the quadratic form on the tangent space given by $\nabla^2_{\mathbb{S}^{D-1}} F_{\hat \CA}(x)$ is negative semi-definite. 
By homogeneity of the quadratic form, this is verified by just checking normalized tangent
vectors.
\end{remark}
\begin{proof}
We first calculate the Euclidean gradient and Hessian of $F_{\hat \CA}$.
Let $U_1, \dots, U_K \in \sym(\CT_{D}^n)$ form an orthonormal basis of $\hat\CA$.
Then for all $S, T \in \CT_{D}^n$,
\begin{equation*}P_{\hat \CA}(T) = \sum_{i=1}^K \langle U_i, T\rangle U_i ,\quad \langle P_{\hat \CA}(T), S \rangle = \sum_{i=1}^K \langle U_i, T \rangle \langle U_i, S\rangle,\quad \text{and} \quad \|P_{\hat \CA}(T)\|_F^2 = \sum_{i=1}^K \langle U_i, T\rangle^2.\end{equation*}
Also, direct calculations verify that if $T\in \sym(\CT_{D}^n)$, then $\nabla \langle T, x^n\rangle=n T \cdot x^{n-1}$ and $\nabla^2 \langle T, x^n\rangle= n (n-1) T \cdot x^{n-2}$ (see \cite[Lem.~3.1, Lem.~3.3]{kolda2011shifted}).
Therefore,
\begin{align*}
\nabla F_{\hat \CA}(x) &= \nabla\|P_{\hat \CA}(x^n)\|_F^2= \sum_{i=1}^K \nabla \langle U_i, x^n \rangle^2 \\
&= 2 \sum_{i=1}^K \langle U_i, x^n \rangle \nabla \langle U_i, x^n \rangle = 2n \sum_{i=1}^K \langle U_i, x^n \rangle U_i \cdot x^{n-1} = 2n P_{\hat \CA}(x^{ n}) \cdot x^{ n-1},
\end{align*}
and
\begin{align*}
\nabla^2 F_{\hat \CA}(x) &= \sum_{i=1}^K \nabla^2 \langle U_i, x^n \rangle^2 = 2 \sum_{i=1}^K (\nabla\langle U_i, x^n \rangle) (\nabla \langle U_i, x^n \rangle)^\top +2 \sum_{i=1}^K \langle U_i, x^n \rangle \nabla^2 \langle U_i, x^n \rangle \\
&= 2n^2 \sum_{i=1}^K (U_i \cdot x^{n-1}) (U_i \cdot x^{n-1})^\top + 2n(n-1) \sum_{i=1}^K \langle U_i, x^n \rangle U_i \cdot x^{n-2}.\\
\end{align*}
Thus,
\begin{align*}
z^\top \nabla^2 F_{\hat \CA}(x) z
&= 2n^2 \sum_{i=1}^K \langle U_i, x^{n-1} z \rangle^2 + 2n(n-1) \sum_{i=1}^K \langle U_i, x^{n} \rangle \langle U_i, x^{n-2} z^2 \rangle\\
&= 2n^2 \|P_{\hat \CA}(x^{ n-1} z)\|_F^2 + 2n(n-1)\langle P_{\hat \CA}(x^n), x^{ n-2} z^{ 2}\rangle.
\end{align*}
We now calculate the Riemannian gradient and Hessian of $F_{\hat \CA}$.
In general, for a twice-differentiable function $g:\bbR^n\to \bbR$, the Riemannian gradient and Hessian of the restriction of $g$ to the unit sphere $\mathbb{S}^{D-1}$ are related to the Euclidean counterparts of $g$ as follows (see \cite[Ex.~3.6.1]{absil2009optimization} and \cite[Eq.~(10), Sec.~4.2]{absil2013extrinsic}):
\begin{align*}
 \nabla_{\bbS^{D-1}}g(x) &= (I - x x^\top) \nabla g(x),\\
 \text{and} \quad \nabla^2_{\bbS^{D-1}}g(x) &= (I - x x^\top) (\nabla^2 g(x) - (x^\top \nabla g(x)) \Id_D ) (I - x x^\top).
\end{align*}
Applying this to $F_{\hat \CA}$, and using $x^\top \nabla F_{\hat \CA}(x) = 2n  P_{\hat \CA}(x^{ n}) \cdot x^{n}  = 2n F_{\hat \CA}(x)$, $(I - x x^\top) z = z$ (since $z\perp x$) and $z^{\top}z =1$ (since $z \in \mathbb{S}^{D-1}$), yields the result.
\end{proof}

\medskip

\subsection{Proof of Proposition~\ref{prop:optimality_conditions}}

\begin{proof}
	The first-order (stationary point) condition is given by setting the Riemannian gradient \eqref{eq:riemannian_gradient} to zero.
	The second-order condition is the requirement that the Riemannian Hessian be negative semi-definite, i.e.,
	that \eqref{eq:riemannian_hessian} is non-positive for all unit norm $z \perp x$.
	Dividing by $2n$ and rearranging yields \eqref{eq:stationary_point} and \eqref{eq:before_derived_second_order_optimality} respectively.

	It remains to show \eqref{eq:derived_second_order_optimality}.
	Assume $x \in \mathbb{S}^{D-1}$ is first and second-order critical for \textup{\ref{prob:nspm}}.
	For each $y\in \bbS^{D-1}$, let $\alpha, \beta\in [0,1]$, $z\in \bbS^{D-1}$ be such that $\beta = \sqrt{1-\alpha^2}$, $z\perp x$ and $y = \alpha x + \beta z$. 
  Then,
	\begin{align}
	\nonumber
\N{P_{\hat \CA}(x^{n-1}y)}_F^2
&= \N{\alpha P_{\hat \CA}(x^n) + \beta P_{\hat \CA}(x^{n-1}z)}_F^2 \nonumber \\
&= \alpha^2 \N{P_{\hat \CA}(x^n)}_F^2 + \beta^2 \N{P_{\hat \CA}(x^{n-1}z)}_F^2 + 2 \alpha \beta \langle P_{\hat \CA}(x^n), P_{\hat \CA}(x^{n-1}z) \rangle \nonumber \\
&= \alpha^2 F_{\hat \CA}(x) + \beta^2 \N{P_{\hat \CA}(x^{n-1}z)}_F^2 + 2 \alpha \beta \langle P_{\hat \CA}(x^n), x^{n-1}z \rangle \nonumber \\
&= \alpha^2 F_{\hat \CA}(x) + \beta^2 \N{P_{\hat \CA}(x^{n-1}z)}_F^2 + 2 \alpha \beta \langle P_{\hat \CA}(x^n) \cdot x^{n-1}, z \rangle \nonumber \\
&= \alpha^2 F_{\hat \CA}(x) + \beta^2 \N{P_{\hat \CA}(x^{n-1}z)}_F^2 + 2 \alpha \beta \langle F_{\hat \CA}(x) x, z \rangle \label{eq:opt_cond_1_1} \\
&= \alpha^2 F_{\hat \CA}(x) + \beta^2 \N{P_{\hat \CA}(x^{n-1}z)}_F^2, \label{eq:opt_cond_1_2}
\end{align}
	where we used the stationary point condition \eqref{eq:stationary_point} in \eqref{eq:opt_cond_1_1} and $x\perp z$ in \eqref{eq:opt_cond_1_2}.
	By similar logic,
	\begin{align}
		\nonumber
		\langle P_{\hat \CA}(x^{ n}), x^{n-2} y^{ 2}\rangle &= \langle P_{\hat \CA}(x^{ n}), x^{n-2} (\alpha x + \beta z)^2\rangle \\
		&= \nonumber\alpha^2 \N{P_{\hat \CA}(x^{n})}_F^2 + 2\alpha \beta \langle P_{\hat \CA}(x^{n}),x^{n-1}z\rangle + \beta^2 \langle P_{\hat \CA}(x^{ n}), x^{n-2} z^{ 2}\rangle\\
		&=\label{eq:opt_cond_2_2}
		 \alpha^2 F_{\hat \CA}(x) + \beta^2 \langle P_{\hat \CA}(x^{ n}), x^{n-2} z^{ 2}\rangle.
	\end{align}
	Finally, we substitute \eqref{eq:opt_cond_1_2} and \eqref{eq:opt_cond_2_2} into \eqref{eq:derived_second_order_optimality}, and use the second-order condition \eqref{eq:before_derived_second_order_optimality}:
	\begin{align}
	\hspace{40pt}&\hspace{-40pt} n\N{P_{\hat \CA}(x^{n-1} y)}_F^2 + (n-1)\langle P_{\hat \CA}(x^{ n}), x^{n-2} y^{ 2}\rangle \nonumber \\
	&= (2n -1) \alpha^2 F_{\hat \CA}(x) + \beta^2\left(n\N{P_{\hat \CA}(x^{n-1} z)}_F^2 + (n-1)\langle P_{\hat \CA}(x^{ n}), x^{n-2} z^{ 2}\rangle\right) \nonumber \\
	&\le  (2n -1)\alpha^2 F_{\hat \CA}(x) + \beta^2 F_{\hat \CA}(x) \nonumber \\
	&=  (2n -1)\alpha^2 F_{\hat \CA}(x) + (1-\alpha^2)F_{\hat \CA}(x) \nonumber\\
	&= (1 + (2n-2)\alpha^2) F_{\hat \CA}(x). \label{eq:almost_derived}
	\end{align}
	Substituting $\alpha = \langle x, y\rangle$ into \eqref{eq:almost_derived} gives  \eqref{eq:derived_second_order_optimality} as desired.
\end{proof}

\section{Preparatory results on random ensembles over the sphere and Conjecture~\ref{conj:grammians}} 
\label{sec:uniform_ensembles}
In this section, we prove the results on random ensembles over the sphere that are used in the analysis of the overcomplete random tensor model. Specifically, we prove Lemma \ref{prop:RIP_consequence} and Proposition \ref{lemma:random_results}. 

\subsection{Proof of Lemma~\ref{prop:RIP_consequence}}

\begin{proof}
Fix $x \in \bbS^{D-1}$, and assume without loss of generality that the indices are ordered
according to $\langle a_1,x\rangle^2 \geq \ldots\geq \langle a_K, x\rangle^2$.
We claim that
the index set $\CI(x)=[p]$ satisfies the specified conditions.
Denote by
$A_p \in \bbR^{D\times p}$ the submatrix of $A$ consisting of the first $p$ columns.
The $(p,\delta)$-RIP implies
\begin{align}
1+\delta = (1 + \delta) \| x \|_2^2 \geq x^\top A_{p}A_{p}^\top x = \sum_{i =1}^{p}\langle a_i, x\rangle^2 \geq (1 - \delta) \| x \|_2^2 = 1 - \delta. \label{eq:RIP-first}
\end{align}
Hence, $\mathcal{I}(x)$ satisfies the first condition.
As for the second condition, by the ordering assumption,
\begin{equation} \label{eq:RIP-second}
    \sum_{i =1}^{p}\langle a_i, x\rangle^2 \geq p \langle a_p, x \rangle^2.
\end{equation}
Inequalities \eqref{eq:RIP-first} and \eqref{eq:RIP-second}
imply $\langle a_p, x \rangle^2\leq (1+\delta)p^{-1}$.  Then again by the ordering relation,
\begin{equation*}
\max_{i \not\in \CI(x)}\langle a_i, x\rangle^2 \leq \langle a_p, x\rangle^2 \leq \frac{1+\delta}{p} \leq \frac{1+\delta}{c_\delta}\frac{\log(K)}{D} = \tilde c_{\delta} \frac{\log(K)}{D}. \hfill \qedhere
\end{equation*}
\end{proof}

\medskip

\subsection{Proof of Proposition~\ref{lemma:random_results}}

Now consider a random ensemble of independent copies $a_1,\ldots,a_K$ of the random
vector $a \sim \textrm{Unif}(\bbS^{D-1})$.
Before showing Proposition~\ref{lemma:random_results}, which states that for this ensemble the assumptions \ref{enum:RIP}-\ref{enum:correlation} hold with high probability for all $n$, and \ref{enum:GInverse} holds with high probability if $n=2$, we will recall these assumptions.
Write $A := [a_1|\ldots|a_K] \in \mathbb{R}^{D \times K}$. 
\begin{enumerate}[leftmargin=1.5cm,label=\textbf{A\arabic*}]
	\item\label{enum:RIP_supp} There exists $c_{\delta}>0$, depending only on $\delta$, such that $A$ is $(\lceil c_{\delta}D/\log(K) \rceil,\delta)$-RIP. 
	\item\label{enum:correlation_supp} There exists  $c_1 > 0$, independent of $K,D$, such that $\max_{i\neq j}\langle a_i, a_j\rangle^2 \leq c_1 \log (K)/D$.
	\item \label{enum:GInverse_supp} There exists $c_2 > 0$, independent of $K,D$, such that $\N{G_{n}^{-1}}_2 \leq c_2$.
\end{enumerate}

The proof for \ref{enum:GInverse} is lengthy and more technically  involved (even though we only prove it for the case of $n = 2$).
So, we first give the arguments for \ref{enum:RIP}-\ref{enum:correlation}. 

\begin{lemma}[\ref{enum:RIP}-\ref{enum:correlation} in Proposition \ref{lemma:random_results}]
\label{lem:random_results_supplementary}
Let $a_1,\ldots,a_K$ be independent copies of the random vector $a \sim \textup{Unif}(\bbS^{D-1})$,
and define $A := [a_1|\ldots|a_K] \in \bbR^{D\times K}$. Assume $\log K = o(D)$.

\noindent 1. Fix  $\delta \in (0,1)$. There exists a universal constant $C >0$ and constants $D_0 \in \bbN, c_{\delta} > 0$ depending only on $\delta$
such that, if $D \geq D_0$ and we put $p := \lceil c_\delta D/\log(K) \rceil$, then $A$ is $(p,\delta)$-RIP with probability at least $1-2\exp(-C\delta^2 D)$.

\noindent 2. There exists a universal constant $c_1 > 0$ such that 
\begin{align*}
\bbP\left(\max_{i,j: i\neq j} \langle a_i, a_j\rangle^2 \leq  c_1 \frac{\log(K)}{D}\right) \geq 1 - \frac{1}{K}.
\end{align*}
\end{lemma}

\begin{proof}
1. By \cite[Thm.~5.65]{vershynin2010introduction},  $A$
satisfies the $(p,\delta)$-RIP with probability at least
$1-2\exp(-C\delta^2 D)$, provided that $D \geq C'\delta^{-2} p\log(e K/p)$,
where $C, C' > 0$ are universal constants.
Thus it suffices to check the latter inequality holds if $c_{\delta}, D_0$ are appropriately chosen.
Because $\log K = o(D)$,
 for each fixed $c_{\delta} >0$, there exists $D_0$ such that $p > e$ whenever $D \geq D_0$.
In this case, we have
\begin{align}
C'\delta^{-2} p\log(e K/p)  & \, <  \, C'\delta^{-2} p\log(K) \, = \, C' \delta^{-2} \left\lceil c_{\delta}D / \log(K) \right\rceil \log(K) \nonumber \\   & \, \leq \,  C'\delta^{-2}c_\delta D + C' \delta^{-2} \log(K)  \, < \, C' \delta^{-2} \tfrac{3}{2} c_{\delta}D, \label{eq:A1-RIP}
\end{align}
where the last inequality in \eqref{eq:A1-RIP} is because $p > e$ implies $\log(K) < \tfrac{1}{2} c_{\delta} D$.
We note that the right-most quantity in  \eqref{eq:A1-RIP} is bounded above by $D$, as desired, if we choose $c_{\delta} \leq 2 \delta^{2}/(3C')$.

\smallskip

2.  Let $d_{ij} := \SN{\langle a_i, a_j\rangle}$. By rotational symmetry, $d_{ij} \stackrel{d}{=} \SN{a_{11}}$ for $i\neq j$ in
distribution, where $a_{11} := \langle a_{1}, e_{1} \rangle$ denotes the first coordinate of $a_1$.  By \cite[Thm.~3.4.6]{vershynin2018high},  $a_{11}$ is a sub-Gaussian random variable with sub-Gaussian norm $\N{a_{11}}_{\psi_2}\leq c'_1 D^{-\frac{1}{2}}$ for a universal constant $c'_1 > 0$.
Using this, a union bound and a tail bound for sub-Gaussian random variables \cite[Eq.~2.14]{vershynin2018high},
\begin{align}
\bbP\left(\max_{i\neq j} d_{ij} \leq t\right) &\, = \, 1 - \bbP\left(\textrm{there exists } i\neq j \textrm{ s.t. } d_{ij} > t\right)
\, \geq \, 1 - \sum_{i=1}^{K}\sum_{j=i+1}^{K}\bbP\left(d_{ij} > t\right) \nonumber \\
& \,=\, 1 - \frac{K(K-1)}{2}\bbP\left(\SN{a_{11}} > t\right) \, > \, 1 - K^2 \exp\left(-\frac{t^2D}{c''_1}\right) \nonumber
\end{align}
for a universal constant $c''_1 > 0$.
We finish by setting $t = \sqrt{3 c''_1 \log(K)/D}$ and \nolinebreak $c_1 = 3c''_1$.
\end{proof}

\medskip

Next we come to \ref{enum:GInverse} for a random ensemble on the sphere when $n=2$.

\begin{proposition}[\ref{enum:GInverse} in Proposition \ref{lemma:random_results}]
\label{prop:grammian_result}
Let $K,D \in \bbN$ satisfy $1\leq K \leq D^2$.
Let $a_1,\ldots,a_K$ be independent copies of the random vector $a \sim \textup{Unif}(\bbS^{D-1})$.
Define the Grammian $G_2 \in \operatorname{Sym}(\mathcal{T}^2_K)$ by
$(G_2)_{ij} := \langle a_i^{\bullet 2}, a_j^{\bullet 2} \rangle = \langle a_i, a_j\rangle^2$. Then there exists a universal constant $C > 0$ such that
the minimal eigenvalue of $G_2$ satisfies
\begin{equation}  \label{eq:bound-Grammian-n2}
\mu_K(G_2) \geq 1 - (C+1)\frac{\sqrt{K}}{D}\log\left(\frac{eD}{\sqrt{K}}\right)\!, \,\,\,\, \text{with probability at least} \,\,\,\, 1 - C \left( \frac{eD}{\sqrt{K}} \right)^{\!\!-C\sqrt{K}}
\end{equation}
In particular, if $K = o(D^2)$, then for each constant $c_2 > 1$, we have $\| G_2^{-1} \|_2 \leq c_2$  with probability tending to $1$ as $D \rightarrow \infty$.
\end{proposition}
The statement is implied by $(p,\delta)$-RIP with $p = K$ for matrices like $[a_1^{\bullet 2}|\ldots|a_K^{\bullet 2}]$
(columns are Khatri-Rao squares of vectors).
Such RIP statements have been
analyzed in \cite{fengler2019restricted}.
The main
technical ingredients are \cite[Thm.~3.3]{adamczak2011restricted}, which proves RIP
for matrices containing columnwise sub-exponential random vectors, and the fact
that certain vectorized  Khatri-Rao products of sub-Gaussian random vectors are sub-exponential with favorable sub-exponential norm.
For completeness,
we include a self-contained proof of Proposition \ref{prop:grammian_result}.
First, here is a version of \cite[Thm.~3.3]{adamczak2011restricted} tailored to our needs.

\smallskip

\begin{proposition}[Tailored version of {\cite[Thm.~3.3]{adamczak2011restricted}}]
\label{thm:adamczak_restricted}
Let $K,D \in \bbN$ satisfy $1\leq K \leq D^2$.
Let $X_1,\ldots,X_K \in \bbR^{D^2}$ be
independent copies of a sub-exponential random vector $X \in \bbR^{D^2}$ with $\CO(1)$ sub-exponential norm (i.e., upper bounded independently
of $D$ and $K$).
Furthermore, assume that $\N{X}_2 = D$  almost surely.
Define $M := \frac{1}{D} [X_1|\ldots|X_K] \in \bbR^{D^2\times K}$.
Then there exists a universal constant $C > 0$ such that \begin{align} \label{eq:MMt-bound-prob}
    \left\| M^{\top} M - \Id_K \right\|_2 \leq C\frac{\sqrt{K}}{D}\log\left(\frac{eD}{\sqrt{K}}\right)\!, \,\,\,\, \text{with probability at least} \,\,\, 1 - C \left(\frac{eD}{\sqrt{K}}\right)^{\!\!-C\sqrt{K}}
\end{align}
In particular, in this event the minimal eigenvalue of the associated Grammian satisfies
\begin{align} \label{eq:MMt-bound}
\mu_K(M^\top M) \geq 1 - C\frac{\sqrt{K}}{D}\log\left(\frac{eD}{\sqrt{K}}\right).
\end{align}
\end{proposition}
\begin{proof}
The first part of the statement is \cite[Thm.~3.3]{adamczak2011restricted},
in the special case (relating their notation to our notation)
$n = D^2$, $m = K$, $N=  K$, $r = 1$, $K  =1$, $K' = 1+\epsilon$ for any $\epsilon \in (0,1)$
and $\theta'\rightarrow 0$, where we note that $M$ satisfies $(K, \delta)$-RIP if and only if $\| M^{\top}M - \Id_K\|_2 \leq \delta$, since $M$ has $K$ columns.
The second part follows immediately from the definition of the spectral norm.
\end{proof}

\medskip

To  deduce Proposition \ref{prop:grammian_result}, we apply Proposition \ref{thm:adamczak_restricted}
to a centered and scaled version of the Khatri-Rao squares $a_i^{\bullet 2} = \operatorname{Vec}(a_i^{2})$, following \cite{fengler2019restricted}.
\begin{proof}[Proof of Proposition \ref{prop:grammian_result}]
Consider the Khatri-Rao square $A^{\bullet 2} := [a_1^{\bullet 2} | \ldots | a_K^{\bullet 2}] \in \mathbb{R}^{D^2 \times K}$ and note that $G_2 = (A^{\bullet 2})^{\top} A^{\bullet 2}$.
Instead of working with $a_1^{\bullet 2},\ldots,a_K^{\bullet 2}$, we introduce auxiliary variables
\begin{equation*}
X_i :=
\sqrt{\frac{D}{D-1}}\operatorname{Vec}(D a_i^2 - \Id_D) \in \mathbb{R}^{D^2}\!\!,
\end{equation*}
in order to be able to apply
Proposition \ref{thm:adamczak_restricted}.
Then $X_1, \ldots, X_K$ are independent copies of the random vector $X \in \mathbb{R}^{D^2}$ given by
\begin{equation}
X := \sqrt{\frac{D}{D-1}} \operatorname{Vec}(Da^2 - \Id_D), \quad \textup{where } a \sim \operatorname{Unif}(\mathbb{S}^{D-1}). \label{eq:def-X}
\end{equation}

Note that
$\mathbb{E}[X_i] = \sqrt{\frac{D}{D-1}} \left( \operatorname{Vec} (D \mathbb{E}[a_i^2] - \Id_D) \right) = 0$.
Using $a_i \in \mathbb{S}^{D-1}$, we also have
\begin{align*}
    \| X_i \|_2^2 &= \frac{D}{D-1} \|D a_i a_i^{\top} - \Id_D \|_F^2 = \frac{D}{D-1} \left( \| D a_i a_i^{\top} \|_F^2 - 2 \langle D a_i a_i^{\top}, \Id_D \rangle + \| \Id_D \|_F^2 \right) \\
    &= \frac{D}{D-1}\left( D^2 \|a_i\|_2^4 - 2D \|a_i\|_2^2 + D \right) = \frac{D}{D-1} \left( D^2 - 2D + D \right) = D^2,
\end{align*}
so that $\|X_i\|_2 = D$ as required in Proposition \ref{thm:adamczak_restricted}.

Define $M := \frac{1}{D}[X_1 | \ldots | X_K] \in \mathbb{R}^{D^2 \times K}$.  For each $i,j \in [K]$,  again using $a_i, a_j \in \mathbb{S}^{D-1}$, we have
\begin{align*}
   \left( M^{\top} M \right)_{i,j} &= \frac{1}{D(D-1)} \langle D a_i a_i^{\top} - \Id_D, D a_j a_j^{\top} - \Id_D \rangle_F \\
   &= \frac{1}{D(D-1)}\left( D^2 \langle a_i a_i^{\top}, a_j a_j^{\top} \rangle_F - D \langle a_i a_i^{\top} + a_j a_j^{\top}, \Id_D \rangle_F  + \langle \Id_D, \Id_D \rangle_F \right) \\
   &= \frac{1}{D(D-1)} \left( D^2 \langle a_i, a_j \rangle^2 - D \| a_i \|_2^2 - D \|a_j\|_2^2 + D \right) \\
   &= \frac{D}{D-1} \langle a_i, a_j \rangle^2 - \frac{1}{D-1}.
\end{align*}
Hence, $M^{\top}M = \frac{D}{D-1} G_2 - \frac{1}{D-1} \mathbbm{1}_K \mathbbm{1}_K^{\top}$, which may be rewritten $G_2 = \frac{D-1}{D} M^{\top} M + \frac{1}{D} \mathbbm{1}_K \mathbbm{1}_K^{\top}$.
This implies
\begin{equation} \label{eq:shifted-grammian}
\mu_K(G_2) \geq \frac{D-1}{D} \mu_K(M^{\top}M).
\end{equation}

Thus, we focus on the Grammian associated with the shifted and centered random variables $X_1, \ldots, X_K$.
To apply Proposition~\ref{thm:adamczak_restricted} to $M^{\top}M$,
it remains to show that the sub-exponential norm of $X$ in \eqref{eq:def-X} is
bounded by some universal constant, independent of $K,D$.

We briefly indicate how this is done.
We first note that
all random vectors $\sqrt{D}a_1,\ldots,\sqrt{D}a_K$ have the so-called convex concentration property for some universal constant $C'>0$, see \cite[Thm.~6, Def.~7, Thm.~8]{fengler2019restricted}.
Following \cite{fengler2019restricted} and the references therein, this can be used to prove the Hanson-Wright type inequality
\begin{align}
\label{eq:aux_hanson_wright}
\bbP\left(\SN{D(a_i^\top Y a_i - \bbE[a_i^\top Y a_i])} > t\right) \leq  2\exp\left(-C'\min\left(\frac{t^2}{\N{Y}_F^2},\frac{t}{\N{Y}_2}\right)\right),
\end{align}
for all (deterministic) $Y \in \mathbb{R}^{D \times D}$ and $t > 0$.
Taking now an arbitrary unit-norm vector $y \in \bbR^{D^2}$, denoting $Y \in \bbR^{D\times D}$
as the matrix satisfying $y = \opvec(Y)$, and using $\mathbb{E}[aa^{\top}] = \frac{1}{D} \Id_D$, we have
\begin{align}
    \langle X, y \rangle &= \sqrt{\frac{D}{D-1}} \langle D a a^{\top} - \Id_D, Y \rangle \nonumber \\
    &= \sqrt{\frac{D}{D-1}} \langle D a a^{\top} - D \mathbb{E}[aa^{\top}], Y \rangle \nonumber \\
    &= \sqrt{\frac{D}{D-1}} D(a^{\top}Ya - \mathbb{E}[a^{\top}Ya]). \label{eq:easy-swap-Y-y}
\end{align}
Combining \eqref{eq:easy-swap-Y-y}, \eqref{eq:aux_hanson_wright} and $\N{Y}_2 \leq \N{Y}_F = 1$, we obtain
\begin{align*}
\bbP\left(\langle X, y\rangle > \sqrt{\frac{D}{D-1}} \, t \right) \leq 2\exp(-C'\min(t^2, t)).
\end{align*}
This implies $\langle X, y\rangle$ is sub-exponential with sub-exponential norm $C'\sqrt{D/(D-1)} = \CO(1)$.
Taking a supremum over all unit-norm vectors $y \in \bbR^{D^2}$ does not change this bound, since it is independent of $y$,
and so the random vector $X$ has sub-exponential norm $\CO(1)$.
Thus, Proposition~\ref{thm:adamczak_restricted} applies and implies there exists a universal constant $C > 0$ such that \eqref{eq:MMt-bound} holds with the probability in~\eqref{eq:MMt-bound-prob}.

We now notice that \eqref{eq:bound-Grammian-n2} follows by substituting \eqref{eq:MMt-bound} into \eqref{eq:shifted-grammian} and using that $1 \leq K \leq D^2$ implies $\frac{1}{D} \leq \frac{\sqrt{K}}{D}\log(eD/\sqrt{K})$, because then
\begin{align} \label{eq:Gram-n2-almostdone}
    \mu_{K}(G_2) \geq (1 - \frac{1}{D}) \left( 1 - C\frac{\sqrt{K}}{D}\log\left(\frac{eD}{\sqrt{K}}\right) \right) \geq 1 - (C+1) \frac{\sqrt{K}}{D} \log\left( \frac{eD}{\sqrt{K}}\right).
\end{align}

To conclude, we justify the last sentence in Proposition~\ref{prop:grammian_result} where $K = o(D^2)$.
It only remains to note that, in this case, the right-most quantity in \eqref{eq:Gram-n2-almostdone} tends to $1$ as $D \rightarrow \infty$.
However, this holds because
 $\frac{\sqrt{K}}{D} \rightarrow 0^{+}$ by $K = o(D^2)$, and $\lim_{x \rightarrow 0^{+}} x\log(\frac{1}{x}) = 0$ by L'H{\^o}pital's rule.
\end{proof}

\medskip

As mentioned in the main body of the paper, this proof technique cannot be immediately applied to the case $n > 2$
because some key technical results are missing. 
We are not aware of
extensions of Proposition \ref{thm:adamczak_restricted} to random variables with tails heavier than
subexponential random variables. 
Further we do not know how to show that
an auxiliary variable $\tilde z$ formed from higher-order vectorized tensors $a_i^{n}$
satisfies a suitable tail bound, which would be needed in an extended Proposition \ref{thm:adamczak_restricted}.

\section{Technical tools for proving Theorems \ref{thm:main_result_deterministic} and \ref{thm:main_result_for_now}}
\label{subsec:prep_material_proofs_main}

In this section we prove the key technical tools to be used in the proofs of our landscape theorems.

\begin{lemma}[Subspace perturbation effect]\label{lem:DeltaCA_techlemma}
Let $\CA$, $\hat{\CA}$ be any two subspaces of $\symten{n}$, and \nolinebreak define
\begin{equation*}
	\Delta_{\CA} := \|P_\CA - P_{\hat{\CA}}\|_{F\rightarrow F} = \sup_{T \in \symten{n},\ \N{T}_F = 1}\|P_\CA(T) - P_{\hat{\CA}}(T)\|_F.
\end{equation*}
Then for all $S, T$ in $\CT_D^{n}$, we have
\begin{equation}\label{eq:deltaCA_dot_bound}
|\langle P_\CA(T) - P_{\hat{\CA}}(T), S\rangle| \le \Delta_{\CA}\|\sym(T)\|_F\|\sym(S)\|_F  \le \Delta_{\CA}\|T\|_F\|S\|_F,
\end{equation}
and
\begin{equation}\label{eq:deltaCA_norm_bound}
\left|\|P_\CA(T)\|_F^2 - \|P_{\hat{\CA}}(T)\|_F^2\right| \le \Delta_{\CA}\|\sym(T)\|_F^2  \le \Delta_{\CA}\|T\|_F^2.
\end{equation}
\end{lemma}

\begin{proof}
    Using the fact $\CA, \hat{\CA} \subseteq \symten{n}$, the Cauchy-Schwarz inequality, and the definition of $\Delta_{\CA}$,
    \begin{align*}
        \left| \langle P_{\CA}(T) - P_{\hat \CA}(T), S \rangle \right| &= \left| \langle P_{\CA}(\operatorname{Sym}(T)) - P_{\hat \CA}(\operatorname{Sym}(T)), \operatorname{Sym}(S) \rangle \right| \\
        & \leq \left\| P_{\CA}(\operatorname{Sym}(T)) - P_{\hat \CA}(\operatorname{Sym}(T)) \right\|_F \left\| \operatorname{Sym}(S) \right\|_F \\
        & \leq \Delta_{\CA} \left\| \operatorname{Sym}(T) \right\|_F \| \operatorname{Sym}(S) \|_F,
    \end{align*}
    which gives \eqref{eq:deltaCA_dot_bound}.  Equation \eqref{eq:deltaCA_norm_bound} then follows by setting $S = T$ in \eqref{eq:deltaCA_dot_bound}, and using $\langle P_{\CA}(T) - P_{\hat \CA}(T), T \rangle = \| P_{\CA}(T) \|_F^2 - \| P_{\hat \CA}(T)\|_F^2$ which holds since $P_{\CA}$ and $P_{\hat \CA}$ are orthogonal projectors.
\end{proof}

Next we recall the remainder terms for $P_{\CA}(x^n)$ with respect to $x^{n-s}a_i^s$ \eqref{eq:remainder}. These are defined by
\begin{equation*}
    \langle P_{\CA}(x^n), x^{n-s}a_i^s \rangle =: \sigma_i \zeta_i^{n-s} + R_{i,s}.
\end{equation*}

\smallskip

\begin{lemma} \label{lem:trivial-Ris-stuff}
    The following alternative expressions for the remainder terms hold:
\begin{enumerate}
    \item For any $x \in \mathbb{S}^{D-1}$, $i \in [K]$ and $0 \leq s \leq n$, we have
    \begin{equation}  \label{eq:Ris-another}
        R_{i,s} = \langle P_{\CA}(x^n), x^{n-s}a_i^s \rangle - \sigma_i \zeta_i^{n-s}.
    \end{equation}
    \item For any $x \in \mathbb{S}^{D-1}$ and $i \in [K]$, we have
    \begin{equation} \label{eq:sigma_and_zeta}
        R_{i,n} = \zeta_i^n - \sigma_i.
    \end{equation}
    \end{enumerate}
\end{lemma}

In the first step of the proofs of our two main results, we derive lower bounds for $\N{A^\top x}_{\infty}$, where $x$ is a  second-order critical point.
For both theorems, this step is based on the following inequality.

\smallskip

\begin{proposition}[Tool \#1]
\label{prop:auxliary_one_big}
For any second-order critical point $x$ of \textup{\ref{prob:nspm}} and $i \in [K]$, it holds
    \begin{align}
	\label{eq:starting_point_max_correlation_coeff}
	(1+2(n-1)\zeta_i^2) F_{\CA}(x) &\geq (2n-1)\sigma_i\zeta_i^{n-2} + n\zeta_i^{n-2}R_{i,n} + (n-1)R_{i,2} - (4n-2)\Delta_{\CA}.
\end{align}
\end{proposition}
\begin{proof}
   We substitute $y = a_i$ into the inequality \eqref{eq:derived_second_order_optimality}, which is a consequence of $x$ being first and second-order critical for \textup{\ref{prob:nspm}}:
    \begin{align}
\label{eq:det_second_order_optimality_alt_1}
(1+2(n-1)\zeta_i^2)F_{\hat \CA}(x) &\geq n \N{P_{\hat \CA}(x^{n-1}a_i)}_F^2 + (n-1)\langle P_{\hat \CA}(x^n), x^{n-2}a_i^2\rangle.
\end{align}
Then by the subspace perturbation results \eqref{eq:deltaCA_dot_bound} and \eqref{eq:deltaCA_norm_bound}, we have
\begin{align}
  &  F_{\hat \CA}(x) = \| P_{\hat \CA}(x^n) \|_F^2 \leq \| P_{\CA}(x^n) \|_F^2 + \Delta_{\CA} = F_{\CA}(x) + \Delta_{\CA}, \label{eq:subperturb-1} \\[0.7em]
  & \|P_{\hat \CA}(x^{n-1} a_i)\|_F^2 \geq \|P_{\CA}(x^{n-1}a_i)\|_F^2 - \Delta_{\CA} \geq \|P_{\operatorname{Span}\{a_i^n\}}(x^{n-1}a_i)\|_F^2 - \Delta_{\CA} \nonumber \\
  & \hspace{6.75em} = \zeta_i^{2n-2} - \Delta_{\CA} = \zeta_i^{n-2}(R_{i,n} + \sigma_i) - \Delta_{\CA}, \label{eq:subperturb-2} \\[0.7em]
  &\langle P_{\hat \CA}(x^n), x^{n-2} a_i^2 \rangle \geq \langle P_{\CA}(x^n), x^{n-2} a_i^2 \rangle - \Delta_{\CA} = R_{i,2} + \sigma_i \zeta_i^{n-2} - \Delta_{\CA}. \label{eq:subperturb-3}
\end{align}
Here we used $\operatorname{Span}\{a_i^n\} \subseteq \CA$ in the second inequality of \eqref{eq:subperturb-2}, the identity \eqref{eq:sigma_and_zeta} in the second equality of \eqref{eq:subperturb-2}, and the identity \eqref{eq:Ris-another} in the equality of \eqref{eq:subperturb-3}.
Substituting \eqref{eq:subperturb-1}, \eqref{eq:subperturb-2} and \eqref{eq:subperturb-3} into \eqref{eq:det_second_order_optimality_alt_1} completes the proof.
\end{proof}

\medskip

In the second part of both proofs, we show concavity holds in a spherical cap near each $s a_i$, $i\in[K],\,s\in\{-1,1\}$. As a starting point in both arguments, we use the next statement to upper bound the eigenvalues of the Riemannian Hessian. 
\begin{proposition}[Tool \#2]
	\label{prop:concavity_base}
	For any $x, z \in \bbS^{D-1}$ with $z \perp x$, the Riemannian Hessian of $F_{\hat\CA}$ satisfies,
  for any $i \in [K]$,
	\begin{align}
		\label{eq:lb_hessian_value}
		\frac{1}{2n}z^\top \nabla_{\bbS^{D-1}}^2 F_{\hat \CA}(x) z \leq n\N{P_{\CA}(x^{n-1}z)}_F^2 + (n-1)\langle P_{\CA}(x^n), x^{n-2}z^2\rangle - \zeta_i^{2n} + 4\Delta_{\CA}.
	\end{align}
\end{proposition}
\begin{proof}
	We first show that for all integers $0 \leq s \leq n$
	\begin{equation} \label{eq:sym_xzlemma}
		\left\| \sym(x^{n-s}z^s) \right\|_F = \binom{n}{s}^{\!-1/2}.
	\end{equation}
	Let $R \in \textup{SO}(D)$ (special orthogonal group).  Then we have (denoting Tucker product on the RHS)
	\begin{equation} \label{eq:tucker-1}
		\left\| \sym(x^{n-s}z^s) \right\|_F = \left\| \sym(x^{n-s}z^s) \times_1 R \times_2 \ldots \times_n R \right\|_F
	\end{equation}
	by rotational invariance of Frobenius norm.
	Furthermore,
	\begin{equation} \label{eq:tucker-2}
		\sym(x^{n-s}z^s) \times_1 R \times_2 \ldots \times_n R = \sym(x^{n-s} z^s \times_1 R \times_2 \ldots \times_n R) = \sym((Rx)^{n-s} (Rz)^s)
	\end{equation}
	by a direct calculation.
	Inserting \eqref{eq:tucker-2} into \eqref{eq:tucker-1}, and choosing $R$ appropriately,
	\begin{equation*}
		\left\| \sym(x^{n-s}z^s) \right\|_F  = \left\| \sym(e_1^{n-s}e_2^s) \right\|_F
	\end{equation*}
	where $e_1, e_2$ are the first two standard basis vectors of $\mathbb{R}^D$.
	However, $\sym(e_1^{n-s}e_2^s)$ is the tensor
	whose $(i_1, \ldots, i_n)$-entry equals $\binom{n}{s}^{-1}$ if $(i_1, \ldots, i_n)$ consists of $n-s$ ones and $s$ twos (in some order), and equals $0$ otherwise.  Hence,
	\begin{equation*}
		\left\| \sym(e_1^{n-s}e_2^s) \right\|_F = \sqrt{\binom{n}{s} \cdot \binom{n}{s}^{\!-2}} = \binom{n}{s}^{\!-1/2}.
	\end{equation*}

	Then, by the formula for the Riemannian Hessian \eqref{eq:riemannian_hessian},
	\begin{align}
		\label{eq:starting_point_concativity_}
		\frac{1}{2n}z^\top \nabla_{\bbS^{D-1}}^2 F_{\hat \CA}(x) z = n\N{P_{\hat \CA}(x^{n-1}z)}_F^2 + (n-1)\langle P_{\hat \CA}(x^n), x^{n-2}z^2\rangle - F_{\hat \CA}(x).
	\end{align}
	To upper-bound this with quantities involving $\CA$ rather than $\hat \CA$, we use the subspace perturbation result Lemma~\ref{lem:DeltaCA_techlemma}, and \eqref{eq:sym_xzlemma}:
	\begin{align*}
		\N{P_{\hat \CA}(x^{n-1}z)}_F^2 &\leq \N{P_{\CA}(x^{n-1}z)}_F^2 + \Delta_{\CA}\left\|\operatorname{Sym}(x^{n-1}z)\right\|_F^2 = \N{P_{\CA}(x^{n-1}z)}_F^2 + \frac1{n}\Delta_{\CA}, \\[0.9em]
		\langle P_{\hat \CA}(x^n), x^{n-2}z^2\rangle &\leq \langle P_{\CA}(x^n), x^{n-2}z^2\rangle + \left\|\operatorname{Sym}(x^{n-2}z^2)\right\|_F \Delta_{\CA}\\
		&=\langle P_{\CA}(x^n), x^{n-2}z^2\rangle +  \frac{\sqrt{2}}{\sqrt{n(n-1)}}\Delta_{\CA} \leq \langle P_{\CA}(x^n), x^{n-2}z^2\rangle + \frac{2}{n-1}\Delta_{\CA},\\[0.9em]
		F_{\hat \CA}(x) \hspace{20pt}&\hspace{-20pt}= \N{P_{\hat \CA}(x^n)}_F^2 \geq \N{P_{\CA}(x^n)}_F^2 - \Delta_{\CA} \geq \N{P_{\operatorname{Span}(a_i^n)}(x^n)}_F^2 - \Delta_{\CA} =  \zeta_i^{2n} - \Delta_{\CA}.
	\end{align*}
	Inserting these bounds into \eqref{eq:starting_point_concativity_} gives \eqref{eq:lb_hessian_value} as announced.
\end{proof}

We now recall some basics about geodesic convexity on the sphere.

\smallskip

\begin{definition}
Spherical caps and geodesics are defined on the unit sphere as follows.
\begin{enumerate} \setlength\itemsep{1em}
\item For $ y \in \mathbb{S}^{D-1}$ and $r \in  (0,1)$, define the \textit{spherical cap} with center $y$ and height $r$ by
\begin{equation*}
     B_r(y) := \{ x \in \mathbb{S}^{D-1} : \langle x, y \rangle \geq 1 - r \} \subseteq \mathbb{S}^{D-1}.
\end{equation*}
\item Given distinct points $x_1, x_2 \in B_{r}(y)$, the  \textit{geodesic segment} connecting $x_1$ to $x_2$ is the curve $c: [0, 1] \rightarrow B_r(y)$ defined by
\begin{equation} \label{eq:geo-def}
    c(t) :=  \cos(t\theta) x_1 + \sin(t\theta) x_2^{\perp} \quad \text{for } 0 \leq t \leq 1.
\end{equation}
Here $\theta := \cos^{-1}(\langle x_1, x_2 \rangle) \in (0, \pi)$ is the angle between $x_1$ and $x_2$,
and $x_2^{\perp}$ is the component of $x_2$ that is orthogonal to $x_1$ (normalized to lie on the unit sphere) given by
\begin{equation*}
x_2^{\perp} := \frac{x_2 - \langle x_1, x_2 \rangle x_1}{\|x_2 - \langle x_1, x_2 \rangle x_1\|_2} = \frac{-\langle x_1, x_2 \rangle}{\sqrt{1 - \langle x_1, x_2 \rangle^2}} x_1 + \frac{1}{\sqrt{1 - \langle x_1, x_2 \rangle^2}} x_2 \in \mathbb{S}^{D-1}.
\end{equation*}
(If $x_1 = x_2$, the geodesic segment connecting $x_1$ to $x_2$ is the constant curve at $x_1$.)
\end{enumerate}
\end{definition}

\smallskip

The standard notion of concavity for twice-differentiable functions on the sphere is as follows. 

\smallskip

\begin{definition}\label{def:geo-conc}
   Let $F$ be a real-valued function defined on an open subset of $\mathbb{S}^{D-1}$ containing the spherical cap $B_r(y)$.
    Assume that $F$ is twice-differentiable.
    \begin{enumerate}
        \item We say $F$ is \textit{geodesically strictly concave} on $B_r(y)$ if the Riemannian Hessian of $F$ is negative definite throughout $B_r(y)$,
        \begin{equation} \label{eq:strict-concave-def}
            z^{\top} \nabla^{2}_{\mathbb{S}^{D-1}}F(x)z < 0, \quad \quad \quad \forall x \in B_r(y) \,\, \forall z \in \mathbb{S}^{D-1} \text{ with } x \perp z.
        \end{equation}
        \item Let $\mu >0$.  We say $F$ is \textit{geodesically $\mu$-strongly concave} on $B_r(y)$ if the Riemannian Hessian of $F$ satisfies the following eigenvalue bound throughout $B_r(y)$,
        \begin{equation} \label{eq:strong-concave-def}
         z^{\top} \nabla^{2}_{\mathbb{S}^{D-1}}F(x)z \leq -\mu, \quad \quad \quad \forall x \in B_r(y) \,\, \forall z \in \mathbb{S}^{D-1} \text{ with } x \perp z.
        \end{equation}
    \end{enumerate}
\end{definition}

\smallskip

The terminology in Definition~\ref{def:geo-conc}is justified by the following standard lemma.

\smallskip

\begin{lemma}[Restricting to geodesics] \label{lem:restrict-geo}
     Let $F$ be a real-valued function defined on an open subset of $\mathbb{S}^{D-1}$ containing the spherical cap $B_r(y)$.
    Assume that $F$ is twice-differentiable.
    \begin{enumerate}
        \item If $F$ is geodesically strictly concave on $B_r(y)$, then for all geodesic segments $c:[0,1] \rightarrow B_r(y)$ with image contained in $B_r(y)$ and $c(0) \neq c(1)$, the pulled-back function $F \circ c :[0,1] \rightarrow \mathbb{R}$ is strictly concave on $[0,1]$. 
        \item Let $\mu >0$. If $F$ is geodesically $\mu$-strongly concave on $B_r(y)$, then for all geodesic segments $c:[0,1] \rightarrow B_r(y)$ with image contained in $B_r(y)$ and $c(0) \neq c(1)$, the pulled-back function $F \circ c :[0,1] \rightarrow \mathbb{R}$ is $\mu L(c)^2$-strongly concave on $[0,1]$.  Here $L(c) := \int_{0}^{1} \| \dot{c}(t) \|_2 dt =  \cos^{-1}(\langle c(0), c(1) \rangle)$ denotes the length of $c$. 
    \end{enumerate}
\end{lemma}
\begin{proof}
    See \cite[Thm.~11.19(3), Def.~11.3]{boumal2020intromanifolds} and \cite[Thm.~11.19(2), Def.~11.5]{boumal2020intromanifolds} respectively.
\end{proof}

\smallskip

The usual implications of geodesic concavity are as follows.

\smallskip

\begin{lemma} \label{lem:geoconcave-implications}
    Assume the setup of Lemma~\ref{lem:restrict-geo}.
    \begin{enumerate}
        \item If $F$ is geodesically strictly concave on $B_r(y)$,  there exists at most one point $x \in B_r(y)$ where $\nabla_{\mathbb{S}^{D-1}}F(x) = 0$.  Such a point $x$ is automatically a strict local maximizer of $F$.
        \item Let $\mu > 0$.  If $F$ is geodesically $\mu$-strongly concave on $B_r(y)$, then for all geodesics segments $c$ as in the lemma above, $F \circ c$ is upper-bounded by a concave quadratic function via:
             \begin{equation*}
           (F \circ c)(t) \, \leq \, (F \circ c)(0)   +   (F \circ c)'(0) \, t   -   \frac{\mu L(c)^2}{2} t^2, \quad \quad \quad \forall t \in [0,1].
      \end{equation*}
    \end{enumerate}
\end{lemma}

\begin{proof}
    1. We first show that for all geodesic segments $c : [0,1] \rightarrow B_r(y)$ with $c(0) \neq c(1)$, the derivative $(F \circ c)'$ is strictly decreasing on $[0,1]$.  To see this, note that  since $F$ is twice-differentiable, $F \circ c$ is twice-differentiable; and since $c$ is a geodesic it holds
    \begin{equation} \label{eq:diff-curve}
            (F \circ c)''(t) \, = \, \dot{c}(t)^{\top} \nabla^2_{\mathbb{S}^{D-1}} F(c(t)) \dot{c}(t),
        \end{equation}
      by \cite[Eq.~(5.30), page~106]{boumal2020intromanifolds}.
      The right-hand side of \eqref{eq:diff-curve} is strictly negative by \eqref{eq:strict-concave-def} and the fact $\dot{c}(t) \neq 0$ (note $\| \dot{c}(t)\|_2 = \cos^{-1}(\langle c(0), c(1) \rangle)$ and we are assuming $c(0) \neq c(1)$).
      Thus $(F \circ c)''$ is strictly negative on $[0,1]$.  Hence by the mean value theorem, $(F \circ c)'$ is strictly decreasing on $[0,1]$.

      Now, to deduce item 1 in the lemma, assume for a contradiction that there exist two distinct points $x_1, x_2 \in B_r(y)$ where the Riemannian gradient vanishes.
      Let $c$ be the geodesic segment connecting $x_1$ to $x_2$.
      Again by  \cite[Eq.~(5.30), page~106]{boumal2020intromanifolds},
      \begin{equation*}
          (F \circ c)'(t) \, = \, \dot{c}(t)^{\top} \nabla_{\mathbb{S}^{D-1}} F(c(t)).
      \end{equation*}
      By the vanishing Riemannian gradient assumption, this implies $(F \circ c)'(0) = 0$ and $(F \circ c)'(1) =0$.  But that contradicts the fact that $(F \circ c)'$ is strictly decreasing.  So item 1 follows. \\

      2. By Taylor's theorem applied to $F \circ c$, for each $t \in [0,1]$ there exists $\xi \in [0,t]$ such that
      \begin{equation} \label{eq:taylor-almost-done}
          (F \circ c)(t) \, = \, (F \circ c)(0) + (F \circ c)'(0)\,t \, + \frac{(F \circ c)''(\xi)}{2} t^2.
      \end{equation}
      From \eqref{eq:diff-curve}, \eqref{eq:strong-concave-def} and $\|\dot{c}(\xi)\|_2 = L(c)$, we get $(F \circ c)''(\xi) \leq -\mu L(c)^2$.  Insert this into \nolinebreak \eqref{eq:taylor-almost-done}.
\end{proof}

\smallskip

The next tool is how we complete the proofs of our main results, by showing a second-order critical point with sufficiently large functional value must land in a spherical cap around one of $sa_i$.

\smallskip

\begin{proposition}[Tool \#3]
\label{prop:concavity_and_local_maxes}
Let $\{a_i : i \in [K]\} \subseteq \bbS^{D-1}$ be any system of vectors, let $\CA = \operatorname{Span}\{a_i^n : i \in [K]\} \subseteq \symten{n}$ be the corresponding subspace of tensors and let $\hat \CA \subseteq \symten{n}$ be another subspace of tensors which is a perturbation of $\mathcal{A}$ with approximation error $\Delta_{\CA}$.
Let $0 < r \leq R < 1$, so there is an inclusion of spherical caps $B_r(sa_i) \subseteq B_R(sa_i)$.
Assume $F_{\hat \CA}$ is geodesically $n$-strongly concave on the inner cap $B_r(sa_i)$, the height of the inner cap satisfies $r > \Delta_{\CA}/n$, and $F_{\hat \CA}$ is geodesically strictly concave on the outer cap $B_R(sa_i)$.
Then, in $B_R(sa_i)$ there exists a unique first-order critical point $x^*$ of \textup{\ref{prob:nspm}}.
Further, this point is a strict local maximizer for  \textup{\ref{prob:nspm}} on $\mathbb{S}^{D-1}$.
Finally, the distance between $x^*$ and $sa_i$ is upper-bounded independently of $R$ and $r$ via:
\begin{equation} \label{eq:det_existence_max}
    \left\| x^* - sa_i \right\|_2^2 \leq \frac{2 \Delta_{\CA}}{n}.
    \end{equation}
\end{proposition}

\begin{proof}
By Lemma~\ref{lem:geoconcave-implications}(a) and the assumption that  $F_{\hat \CA}$ is strictly concave on the outer cap, there exists at most one critical point of $F_{\hat \CA}$ in $B_R(sa_i)$.
    To prove the rest of the proposition, we let the maximum of $F_{\hat \CA}$ in the inner cap $B_r(sa_i)$ be attained at $x^* \in B_r(sa_i)$; such a point exists by compactness of $B_r(sa_i)$.
    We will first show that $x^*$ satisfies \eqref{eq:det_existence_max}, which is the main assertion.
To this end, by the choice of $x^*$,  note $F_{\hat \CA}(x^*) \geq F_{\hat \CA}(sa_i)$.
Of course, $1 \geq F_{\hat \CA}(x^*)$.
By the perturbation bound \eqref{eq:deltaCA_norm_bound},
$F_{\hat \CA}(sa_i) \geq F_{\CA}(sa_i) - \Delta_{\CA} = 1 - \Delta_{\CA}$.
Combining the last three sentences gives
\begin{equation} \label{eq:concave-pf1}
    -\Delta_{\CA} \leq F_{\hat \CA}(s a_i) - F_{\hat \CA}(x^*).
\end{equation}

Now let $c$ be the geodesic segment \eqref{eq:geo-def} connecting $x^*$ to $sa_i$.  By Lemma~\ref{lem:geoconcave-implications}(b) and the assumption that $F_{\hat \CA}$ is $n$-strictly concave on the inner cap, we have $F_{\hat \CA}(s a_i) - F_{\hat \CA}(x^*) \leq (F \circ c)'(0) - \frac{n L(c)^2}{2}$.
By the choice of $x^*$,  $(F \circ c)(0) \geq (F \circ c)(t)$ for all $t \in [0,1]$, whence $(F \circ c)'(0) \leq 0$.  It follows
\begin{equation} \label{eq:concave-pf2}
    F_{\hat \CA}(s a_i) - F_{\hat \CA}(x^*) \leq  - \frac{n L(c)^2}{2}.
\end{equation}
    Putting Eq.~\eqref{eq:concave-pf1} and \eqref{eq:concave-pf2} together, we get
    \begin{equation} \label{eq:concave-pf3}
        - \Delta_{\CA} \leq - \frac{n L(c)^2}{2} \,\,\, \Longleftrightarrow \,\,\, L(c)^2 \leq \frac{2 \Delta_{\CA}}{n}.
    \end{equation}
    Here the geodesic length $L(c)$ is $\theta_0$, where $\theta_0 = \cos^{-1}(\langle x^*, sa_i \rangle)$ is the angle between $x^*$ and $sa_i$ in radians.
    This is related to the squared Euclidean distance between $x^*$ and $sa_i$ as follows:
    \begin{equation*}
        \| x^* - sa_i \|_2^2 = 2 - 2\langle x^*, sa_i \rangle = 2 - 2\cos(\theta_0).
    \end{equation*}
    Using the elementary inequality $\cos(\theta) \geq 1 - \tfrac{1}{2}\theta^2$, which is true for all $\theta \in \mathbb{R}$, we see
    \begin{equation*}
        \| x^* - sa_i \|_2^2 \, \leq \, \theta_0^2 = L(c)^2.
    \end{equation*}
    Thus \eqref{eq:det_existence_max} follows from \eqref{eq:concave-pf3} as desired.
    We can quickly obtain the rest of the statement, and see that $x^*$ is a strict local maximizer of $F_{\hat \CA}$ on $\mathbb{S}^{D-1}
    $ (hence also a first-order critical point), by noting
    \begin{align*}
       2 - 2 \langle x^*, sa_i \rangle  =  \| x^* - sa_i \|_2^2  \leq  \frac{2 \Delta_{\CA}}{n}   \,\,\,\, \implies \,\,\,\,  1 - \langle x^*, sa_i \rangle   \leq  \frac{\Delta_{\CA}}{n}.
    \end{align*}
    For then from the assumption that $r > \Delta_{\CA}/n$, we get $1 - \langle x^*, sa_i \rangle < r$.  Hence $x^*$ lies strictly in the interior of the inner  cap $B_r(sa_i)$.  Therefore $x^*$ is a local maximizer of $F_{\hat \CA}$ on $\mathbb{S}^{D-1}$.  By Lemma~\ref{lem:geoconcave-implications}(a) and strict concavity, $x^*$ is in fact a strict local maximizer.  This ends the proof.
    \end{proof}

\medskip

\section{Proof of Theorem \ref{thm:main_result_deterministic}}
\label{subsec:proof_theorem_main_deterministic}

In this section we prove our deterministic theorem. 
Here we define $\zeta$, $\sigma$ and $R_{i,s}$ as in Definition~\ref{def:basic-setup}.

\subsection[Bounds involving error terms]{Bounds on $R_{i,s}$ and $\| G_n^{-1}\|_{2}$}
We start by relating the auxiliary scalars $R_{i,s}$ to the frame constants $\rho_s$.

\begin{lemma} \label{lem:det-Ris-bound}
    For each $i \in [K]$ and $s \in [n]$, we have
    \begin{align} \label{eq:Ris-all-bound}
 |R_{i,s}| \, \leq \, \max_{\ell: \ell \neq i} |\sigma_{\ell} \zeta_{\ell}^{n-s}| \, \rho_s \, \leq \, \| \sigma \odot \zeta^{\odot n-s} \|_{\infty} \, \rho_s.
    \end{align}
    If $s$ is even, then we have a slight refinement:
    \begin{align} \label{eq:Ris-even-bound}
     \min_{\ell: \ell \neq i} \sigma_{\ell} \zeta_{\ell}^{n-s} \rho_s  \, \leq \, R_{i,s} \, \leq \, \max_{\ell: \ell \neq i} \sigma_{\ell} \zeta_{\ell}^{n-s} \rho_s.
     \end{align}
\end{lemma}
\begin{proof}
    For each $i$ and $s$, use the definition of $R_{i,s}$, the triangle inequality and the definition of $\rho_s$:
    \begin{align*}
        |R_{i,s}| & = |\sum_{\ell: \ell \neq i} \sigma_{\ell} \zeta_{\ell}^{n-s} \langle a_{i}, a_{\ell} \rangle^s|  \leq \sum_{\ell: \ell \neq i} |\sigma_{\ell} \zeta_{\ell}^{n-s}| |\langle a_{i}, a_{\ell} \rangle|^s \leq \max_{\ell : \ell \neq i} |\sigma_{\ell} \zeta_{\ell}^{n-s}| \sum_{\ell : \ell \neq i} |\langle a_i, a_{\ell} \rangle|^s \\
        & \leq \max_{\ell : \ell \neq i} | \sigma_{\ell} \zeta_{\ell}^{n-s}| ( \sum_{\ell =1}^K |\langle a_i, a_{\ell} \rangle|^s  - 1) \leq \max_{\ell : \ell \neq i} | \sigma_{\ell} \zeta_{\ell}^{n-s}| \rho_s \leq \| \sigma \odot \zeta^{\odot n-s} \|_{\infty} \rho_s.
    \end{align*}
    If $s$ is even, then each term $\langle a_i, a_{\ell} \rangle^{s}$ is nonnegative.  So we have
    \begin{align*}
        R_{i,s} = \sum_{\ell: \ell \neq i} \sigma_{\ell} \zeta_{\ell}^{n-s} \langle a_{i}, a_{\ell} \rangle^s \leq \max_{\ell : \ell \neq i} \sigma_{\ell} \zeta_{\ell}^{n-s} \sum_{\ell: \ell \neq i} \langle a_i, a_{\ell} \rangle^s \leq \max_{\ell : \ell \neq i} \sigma_{\ell} \zeta_{\ell}^{n-s} \rho_s,
    \end{align*}
    and likewise $R_{i,s} \geq \min_{\ell : \ell \neq i} \sigma_{\ell} \zeta_{\ell}^{n-s} \rho_s$ as announced.
\end{proof}
\begin{remark}
The frame constants are small under the  assumptions of Theorem~\ref{thm:main_result_deterministic} with $\tau > 0$, for then:
\begin{align}
  n^2 \rho_2 \leq n^2 \rho_2 + (n^2 + n)\rho_n < \frac{1}{6}\quad &\implies \quad \rho_2 < \frac{1}{6n^2} \leq \frac{1}{24}  \label{eq:deterministic-bound-rho-2} \\
    \textup{and}  \,\,\,\, (n^2 + n)\rho_n  \leq n^2\rho_2 + (n^2 + n)\rho_n < \frac{1}{6} \quad &\implies \quad \rho_n < \frac{1}{6(n^2 + n)} = \frac{1}{36}. \label{eq:deterministic-bound-rho-n}
\end{align}
\end{remark}
Consequently in the setting of Theorem~\ref{thm:main_result_deterministic}, the system $\{ a_i^n : i \in [K]\}$ is linearly independent.
\begin{corollary} \label{cor:assumptions-invertible}
    Under the assumptions of Theorem~\ref{thm:main_result_deterministic}, the Grammian matrix given by $(G_n)_{i,j} = \langle a_i, a_j \rangle^n$ is invertible.  In fact,
    \begin{equation*} 
    \| G_n^{-1}\|_{2} \leq \frac{1}{1-\rho_n} \leq \frac{36}{35}.
    \end{equation*}
\end{corollary}
\begin{proof}
    This follows immediately from \eqref{eq:deterministic-bound-rho-n}, and the bound  $1-\rho_n \leq \mu_n(G_n)$ in Lemma~\ref{lem:incoherence_scalars}.
\end{proof}

\subsection{Lower bound on maximum correlation coefficient}

We begin the proof of Theorem \ref{thm:main_result_deterministic} by showing that any constrained
second-order critical point $x$ has at least one large correlation coefficient.

\begin{proposition}
\label{prop:one_big_deterministic_case}
Consider a system $\{a_i : i \in [K]\}$ that satisfies the assumptions of Theorem \ref{thm:main_result_deterministic}.
For a second-order critical point $x$ of \textup{\ref{prob:nspm}}, then either $F_{\mathcal{A}}(x) = 0$ or we have
\begin{align*}
\|\zeta\|_{\infty}^2 &\geq  1 -\left(1+\frac{n}{2(n-1)}\right)  \frac{\rho_2}{1+\rho_2} - \frac{n}{2(n-1)}\frac{\rho_n}{(1-2\rho_n)(1+\rho_2)}  - \frac{4n-2}{2(n-1)}\frac{\Delta_{\CA}}{F_{\CA}(x)}\\
&\geq 1 - 2\rho_2 - 2 \rho_n - 3\frac{\Delta_{\CA}}{F_{\CA}(x)}. 
\end{align*}
\end{proposition}
\begin{proof}
Fix a second-order critical point $x\in \bbS^{D-1}$ of \textup{\ref{prob:nspm}}, and let $\zeta = \zeta(x) = A^\top x \in \mathbb{R}^K$.
Since we can change $a_{\ell}$ to $-a_{\ell}$ without altering the
function $F_{\CA}$ or the constants $\rho_s$, we may assume without loss of generality that $\zeta_{\ell} \geq 0$ for each $\ell \in [K]$.
We may also assume that $F_{\CA}(x) \neq 0$.

Our starting point is Eq.~\eqref{eq:starting_point_max_correlation_coeff} in Proposition~\ref{prop:auxliary_one_big},
which says that for each $i \in [K]$ we have
\begin{align}
\label{eq:starting_point_max_correlation_coeff_used}
\left(1+2(n-1)\zeta_i^2\right) F_{\CA}(x) &\geq (2n-1)\sigma_i\zeta_i^{n-2} + n\zeta_i^{n-2}R_{i,n} + (n-1)R_{i,2} - (4n-2)\Delta_{\CA}.
\end{align}
Since we want to deduce that $\| \zeta \|_{\infty}$ is large,
we fix $i \in [K]$ to be an index such that the first (main) term on the right-hand side of \eqref{eq:starting_point_max_correlation_coeff_used} is maximal; that is,
 let $i \in \argmax_{\ell \in [K]} \sigma_{\ell} \zeta_{\ell}^{n-2}$.
 Now notice
\begin{equation}\label{eq:first-easy-one}
0 < F_{\CA}(x) = \langle P_{\CA}(x^n), x^n\rangle = \langle \sum_{\ell=1}^K \sigma_{\ell} a_{\ell}^n, x^n \rangle =  \sum_{\ell=1}^{K}\sigma_{\ell}\zeta_{\ell}^n \leq \sigma_i\zeta_i^{n-2}\sum_{\ell=1}^{K}\zeta_\ell^2,
\end{equation}
from which it follows that $\sigma_i \zeta_i^{n-2} > 0$. 
In \eqref{eq:starting_point_max_correlation_coeff_used}, our goal is to bound $F_{\CA}(x)$ on the left-hand side from above, and $R_{i,n}$, $R_{i,2}$ on the right-hand side from below, all
by quantities involving $\zeta_i$.

Firstly using \eqref{eq:first-easy-one}, the fact $\sigma_i \zeta_i^{n-2} > 0$ and the definition of $\rho_2$, we have
\begin{align}
\label{eq:det_aux_lhs}
F_{\CA}(x) \leq \sum_{\ell=1}^{K}\sigma_{\ell}\zeta_{\ell}^n \leq \sigma_i \zeta_{i}^{n-2} \sum_{\ell=1}^{K} \zeta_\ell^2 = \sigma_i \zeta_{i}^{n-2}(1+\rho_2).
\end{align}

Then by Lemma~\ref{lem:det-Ris-bound}, Eq.~\eqref{eq:Ris-even-bound} with $s=n$ we have
\begin{align}
\label{eq:det_aux_rhs_1}
    R_{i,2} \geq \min_{\ell: \ell \neq i}\sigma_{\ell}\zeta_{\ell}^{n-2}\rho_2,
\end{align}
while by \eqref{eq:Ris-all-bound} with $s=n$ we have
\begin{align}
\label{eq:det_aux_rhs_2}
\zeta_i^{n-2}R_{i,n} \geq -\zeta_i^{n-2}\N{\sigma}_{\infty}\rho_n.
\end{align}
To further bound the right-hand sides in \eqref{eq:det_aux_rhs_1} and \eqref{eq:det_aux_rhs_2} in terms of $\zeta_i$, we now show that $\sigma_{i}\approx \N{\sigma_{\infty}}$ and that
$\min_{\ell \neq i}\sigma_{\ell}\zeta_{\ell}^{n-2} \geq -\sigma_i \zeta_{i}^{n-2}$.

\medskip

\noindent
\underline{\textit{Showing $\sigma_{i}\approx \N{\sigma}_{\infty}$}}:
Let $j \in [K]$ be an index such that $\SN{\sigma_j} = \N{\sigma}_{\infty}$.
We note that $\sigma_j \geq 0$,
because if $\sigma_j < 0$,  then using $\zeta_j \geq 0$ and \eqref{eq:sigma_and_zeta} we would have
\begin{align*}
\N{\sigma}_{\infty}&= \SN{\sigma_j} \leq \SN{\sigma_j-\zeta_j^n} = 
\SN{R_{j,n}}
|\sum_{\ell \neq j}\sigma_\ell \langle a_j, a_\ell\rangle^n| \leq \N{\sigma}_{\infty}\rho_n.
\end{align*}
This contradicts the assumption $\rho_n < 1$, so indeed $\sigma_j \geq 0$. 
Now consider the three facts:
\begin{enumerate}
    \item $\sigma_i \zeta_i^{n-2} \geq \sigma_j \zeta_i^{n-2}$ by definition of $i$;
    \item $\sigma_j \geq \sigma_i$ by $\sigma_j \geq 0$ and definition of $j$;
    \item $\zeta_i \geq 0$ and $\zeta_j \geq 0$.
\end{enumerate}
It follows that $\zeta_i \geq \zeta_j$.
Using this and the bound for the auxiliary constants \eqref{eq:Ris-all-bound}, we can estimate:
\begin{align*}
\sigma_j - \sigma_i \leq \sigma_j - \zeta_j^{n} + \zeta_i^{n} - \sigma_i = -R_{j,n} \!+\! R_{i,n} \leq
|R_{j,n}| + |R_{i,n}|
\leq 2\sigma_j\rho_n.
\end{align*}

Rearranging gives
\begin{align}\label{eq:this-is-good-one}
\sigma_j \geq \sigma_i \geq (1-2\rho_n)\sigma_j = (1-2\rho_n)\N{\sigma}_{\infty}.
\end{align}
By the assumptions in Theorem~\ref{thm:main_result_deterministic} and \eqref{eq:deterministic-bound-rho-n}, in particular $\rho_n/(1 - 2\rho_n) > 0$.  Hence multiplying throughout by $\rho_n/(1 - 2\rho_n)$ in \eqref{eq:this-is-good-one} yields
\begin{align*}
    \frac{\rho_n}{1 - 2 \rho_n} \sigma_i \leq \rho_n \| \sigma \|_{\infty},
\end{align*}
and then substituting into \eqref{eq:det_aux_rhs_2} we obtain
\begin{align}
\label{eq:det_aux_rhs_3}
\zeta_i^{n-2}R_{i,n} \geq -\zeta_i^{n-2}\N{\sigma}_{\infty}\rho_n \geq -\zeta_i^{n-2}\frac{\rho_n}{1-2\rho_n}\sigma_i.
\end{align}

\medskip

\noindent
\underline{\textit{Showing $\min_{\ell \neq i}\sigma_{\ell}\zeta_{\ell}^{n-2} \geq -\sigma_i \zeta_{i}^{n-2}$}}:
Assume to the contrary that $\min_{\ell \neq i}\sigma_{\ell}\zeta_{\ell}^{n-2} < -\sigma_i \zeta_{i}^{n-2}$
and let $k$ be an index attaining the minimum on the left-hand side.
Since $\sigma_i \zeta_{i}^{n-2} > 0$ and $\zeta_{k}^{n-2} \geq 0$, it must be that $\sigma_{k} < 0$.   Using $\sigma_{k} < 0$ and $\zeta_k \geq 0$, we estimate
\begin{align*}
\SN{\sigma_k} \leq \SN{\sigma_k - \zeta_k^n}_{\infty} = |R_{k,n}| \leq \N{\sigma}_{\infty}\rho_n.
\end{align*}
This implies
\begin{align*}
\N{\sigma \odot \zeta^{n-2}}_{\infty} = \max\{\sigma_i \zeta_i^{n-2}, -\sigma_k\zeta_k^{n-2}\} = - \sigma_k\zeta_k^{n-2} \leq \N{\sigma}_{\infty}\zeta_k^{n-2} \rho_n.
\end{align*}
Similarly to the previous part, using $-\sigma_k \zeta_k^{n-2} > \sigma_k \zeta_j^{n-2}$ we can see $\zeta_k \geq \zeta_j$, where $j$ is again the index
such that $\sigma_j = \N{\sigma}_{\infty}$.
Hence, as before, we have
\begin{equation*}
\sigma_j - \sigma_k = \sigma_j - \zeta_j^n + \zeta_k^n - \sigma_k = -R_{j,n} + R_{k,n} \leq 2 \sigma_j \rho_n,
\end{equation*}
which rearranges to $\sigma_k \geq (1-2\rho_n)\sigma_j$.
Since we have previously shown $\sigma_j \geq 0$ and we  know $\rho_n < 1/2$ by \eqref{eq:deterministic-bound-rho-n}, 
this is in contradiction with $\sigma_k < 0$.
Thus it follows that $\min_{\ell: \ell \neq i}\sigma_{\ell}\zeta_{\ell}^{n-2} \geq -\sigma_i \zeta_{i}^{n-2}$ as desired.
Substituting into \eqref{eq:det_aux_rhs_1} yields
\begin{align}
\label{eq:det_aux_rhs_4}
R_{i,2} \geq \min_{\ell: \ell \neq i}\sigma_{\ell}\zeta_{\ell}^{n-2}\rho_2 \geq -  \sigma_i \zeta_i^{n-2}\rho_2.
\end{align}

\medskip

\noindent
We can now complete the proof of the proposition.
Substituting the bounds \eqref{eq:det_aux_lhs}, \eqref{eq:det_aux_rhs_3}, and \eqref{eq:det_aux_rhs_4} into
the second-order criticality inequality \eqref{eq:starting_point_max_correlation_coeff_used}, we get
\begin{align*}
\left(1+2(n-1)\zeta_i^2\right)(1+\rho_2)\sigma_i\zeta_i^{n-2} &\geq (2n-1)\sigma_i\zeta_i^{n-2} - n \frac{\rho_n}{1-2\rho_n}\sigma_i\zeta_i^{n-2} \\
&\quad - (n-1)\rho_2 \sigma_i \zeta_i^{n-2} - (4n-2)\Delta_{\CA}.
\end{align*}
Dividing by $(1+\rho_2)\sigma_i \zeta_i^{n-2} >0$, subtracting $1$ and then dividing by $2(n-1)$, this \nolinebreak simplifies~to
\begin{align*}
\zeta_i^2 \geq \frac{1}{1+\rho_2}-  \frac{n}{2(n-1)}\left(\frac{\rho_n}{(1-2\rho_n)(1+\rho_2)} + \frac{\rho_2}{1+\rho_2}\right)  - \frac{4n-2}{2(n-1)}\frac{\Delta_{\CA}}{(1+\rho_2) \sigma_i \zeta_i^{n-2}}.
\end{align*}
We replace the denominator in the last term by $F_{\CA}(x)$ by re-using the bound $F_{\CA}(x) \leq \sigma_i \zeta_i^{n-2}(1+\rho_2)$  \eqref{eq:det_aux_lhs}.
Then we rearrange the terms and substitute $\|\zeta\|_{\infty}^2 \geq \zeta_i^2$.  This gives the first inequality in the proposition.
The second stated inequality is an immediate consequence of the first, where we simplify each term up to a constant factor using
 $\frac{n}{2(n-1)} \leq 1$, $1 + \rho_2 \geq 1$, $1-2\rho_n > \frac{1}{2}$ and $\frac{4n-2}{2(n-1)} \leq 3$ .
\end{proof}

\subsection{Concavity}

In the next proof step we analyze the Riemannian Hessian and show that it is strictly negative definite.
\begin{proposition}
\label{prop:det_concavity}
Consider a system $\{a_i : i \in [K]\}$ that satisfies the assumptions in Theorem \ref{thm:main_result_deterministic}.
For any $x, z \in \bbS^{D-1}$ with $z \perp x$, the Riemannian Hessian of $F_{\hat\CA}$ satisfies
\begin{align}
\label{eq:det_lb_hessian_value}
z^\top \nabla_{\bbS^{D-1}}^2 F_{\hat \CA}(x) z \leq -2n+ 2n^2\frac{2-\rho_n}{1-\rho_n}\left(1-\|\zeta\|_{\infty}^{2n}\right) + 2n(3n-2)\frac{\rho_n}{1-\rho_n} + 8n\Delta_{\CA}.
\end{align}
\end{proposition}
\begin{proof}
Let $i \in [K]$ be such that $|\zeta_i| = \N{\zeta}_{\infty}$ where $\zeta = \zeta(x) = A^\top x$.   Our starting point is \eqref{eq:lb_hessian_value} in Proposition \ref{prop:concavity_base}, which reads
\begin{align}
\label{eq:det_concav_tbc}
	\frac{1}{2n}z^\top \nabla_{\bbS^{D-1}}^2 F_{\hat \CA}(x) z \leq n\N{P_{\CA}(x^{n-1}z)}_F^2 + (n-1)\langle P_{\CA}(x^n), x^{n-2}z^2\rangle - \zeta_i^{2n} + 4\Delta_{\CA}.
\end{align}
Expanding $P_{\CA}(x^n)=\sum_{\ell=1}^{K}\sigma_{\ell}a_\ell^n$ as in Definition~\ref{def:basic-setup}, we can upper-bound
the second term on the right-hand side of \eqref{eq:det_concav_tbc} as follows by applying H\"older's inequality and the definition of $\rho_n$:
\begin{align}
\nonumber
\langle P_{\CA}(x^n), x^{n-2}z^2\rangle &= \sum_{\ell=1}^{K}\sigma_{\ell}\zeta_\ell^{n-2}\langle a_\ell, z\rangle^2=
 \sum_{\ell=1}^{K}\zeta_\ell^{2n-2}\langle a_\ell, z\rangle^2 + \sum_{\ell=1}^{K}(\sigma_\ell - \zeta_\ell^{n})\zeta_\ell^{n-2}\langle a_\ell, z\rangle^2\\
 \nonumber
 & \leq \sum_{\ell=1}^{K}\zeta_\ell^{2n-2}\langle a_\ell, z\rangle^2 + \| \sigma - \zeta^{\odot n}\|_{\infty} \sum_{\ell=1}^{K}\zeta_\ell^{n-2}\langle a_\ell, z\rangle^2\\ \nonumber
 &\leq \sum_{\ell=1}^{K}\zeta_\ell^{2n-2}\langle a_\ell, z\rangle^2 + \N{\sigma - \zeta^{\odot n}}_{\infty} \left(\sum_{\ell=1}^{K}\zeta_\ell^{n}\right)^{\frac{n-2}{n}} \left(\sum_{\ell=1}^{K}\langle a_\ell, z\rangle^n\right)^{\frac{2}{n}}\\
 \label{eq:det_concav_tbc_2}
 &\leq \sum_{\ell=1}^{K}\zeta_\ell^{2n-2}\langle a_\ell, z\rangle^2 + \N{\sigma - \zeta^{\odot n}}_{\infty}(1+\rho_n).
\end{align}
Further we can quickly bound the second term in the right-hand side of \eqref{eq:det_concav_tbc_2}.
Firstly by \eqref{eq:sigma_and_zeta}  and \eqref{eq:Ris-all-bound}, we have
$\| \sigma - \zeta^{\odot n} \|_{\infty} = \max_{i \in [K]} |R_{i,n}| \leq \| \sigma \|_{\infty} \rho_n$.
Then we can bound $\| \sigma \|_{\infty}$ via
\begin{align*}
  \N{\sigma}_{\infty} &= \N{\sigma - \zeta^{\odot n}+\zeta^{\odot n}}_{\infty}  \leq 1 + \N{\sigma - \zeta^{\odot n}}_{\infty} \leq 1 + \N{\sigma}_{\infty}\rho_n \,\, \implies \,\, \N{\sigma}_{\infty}\leq (1-\rho_n)^{-1}.
\end{align*}
Putting the last two sentences together gives us 
\begin{align} \label{eq:second-term-cool}
    \| \sigma - \zeta^{\odot n}\|_{\infty} (1 + \rho_n) \leq \| \sigma \|_{\infty} \rho_n (1+\rho_n) \leq \frac{1 + \rho_n}{1 - \rho_n}\rho_n.
\end{align}
As for the first term on the right-hand side of \eqref{eq:det_concav_tbc_2}, we first split the sum:
\begin{equation} \label{eq:split-sum}
    \sum_{\ell=1}^K \zeta_{\ell}^{2n-2} \langle a_{\ell}, z \rangle^2 = \zeta_i^{2n-2} \langle a_i, z \rangle^2 + \sum_{\ell: \ell \neq i} \zeta_{\ell}^{2n-2} \langle a_{\ell}, z\rangle^2.
\end{equation}
We use Bessel's inequality and the assumptions $x, z \in \mathbb{S}^{D-1}$ and $x \perp z$ to get
\begin{align} \label{eq:split-sum-1}
    \zeta_i^{2n-2} \langle a_i, z \rangle^2 = \zeta_i^{2n-2}(\langle a_i, z \rangle^2 + \langle a_i, x \rangle^2 - \zeta_i^2)
    \leq \zeta_i^{2n-2}(\| a_i \|_2^2 - \zeta_i^2) \leq \zeta_i^{2n-2}(1 - \zeta_i^2).
\end{align}
Then we estimate
\begin{align} \label{eq:split-sum-2}
    \sum_{\ell: \ell\neq i} \zeta_\ell^{2n-2}\langle a_\ell, z\rangle^2 &\leq \sum_{\ell: \ell\neq i} \zeta_\ell^{2n-2} =  \sum_{\ell = 1}^{K}\zeta_\ell^{2n-2} -\zeta_i^{2n-2} \le 1 + \rho_{2n-2}-\zeta_i^{2n-2} \leq 1 + \rho_{n}-\zeta_i^{2n-2},
\end{align}
where the last inequality used $2n-2 \geq n$. Finally we substitute \eqref{eq:split-sum-1} and \eqref{eq:split-sum-2} into \eqref{eq:split-sum} to get
\begin{equation} \label{eq:wow-first-term}
    \sum_{\ell=1}^K \zeta_{\ell}^{2n-2} \langle a_{\ell}, z \rangle^2 \leq 1 - \zeta_{i}^{2n} + \rho_n.
\end{equation}
Combining \eqref{eq:det_concav_tbc_2}, \eqref{eq:wow-first-term}, \eqref{eq:second-term-cool} we arrive at:
\begin{align} \label{eq:super-nice}
\langle P_{\CA}(x^n), x^{n-2}z^2 \rangle \leq 1 - \zeta_{i}^{2n} + \rho_n + \frac{1+\rho_n}{1-\rho_n}\rho_n.
\end{align}

At this point, we turn to the first term on the right-hand side of \eqref{eq:det_concav_tbc}.
Recalling Lemma \ref{lem:frob_norms_extended}, Definition~\ref{def:basic-setup} and Corollary~\ref{cor:assumptions-invertible},
we may write
\begin{align} \label{eq:pretty-OK}
\N{P_{\CA}(x^{n-1}z)}_F^2  =  \eta(x^{n-1}z)^{\!\top}  G_n^{-1}\eta(x^{n-1}z),
\end{align}
where $\eta(x^{n-1}z)_{\ell} := \langle x^{n-1}z, a_{\ell}^n  \rangle = \zeta_{\ell}^{n-1} \langle a_{\ell}, z \rangle$.
Clearly \eqref{eq:pretty-OK} is less than or equal to
\begin{align*}
    \N{G_n^{-1}}_2 \| \eta(x^{n-1}z)\|_2^2 = \N{G_n^{-1}}_2\sum_{\ell=1}^{K}\zeta_\ell^{2n-2}\langle a_\ell, z\rangle^2.
\end{align*}
Re-using the just-proven \eqref{eq:wow-first-term} as well as Corollary~\ref{cor:assumptions-invertible}, it follows that
\begin{equation}\label{eq:nice-very-nice}
   \N{P_{\CA}(x^{n-1}z)}_F^2 \leq  \N{G_n^{-1}}_2\sum_{\ell=1}^{K}\zeta_\ell^{2n-2}\langle a_\ell, z\rangle^2 \leq \frac{1 - \zeta_{i}^{2n} + \rho_n}{1 - \rho_n}.
\end{equation}
To complete the proof, we plug \eqref{eq:nice-very-nice} and \eqref{eq:super-nice} into \eqref{eq:det_concav_tbc} and rearrange terms to get \eqref{eq:det_lb_hessian_value}.
\end{proof}

 As a corollary of \ref{prop:det_concavity}, we can now prove the existence of exactly one local maximizer
within spherical caps of the global maximizers $\pm a_i's$ of the clean objective $F_{\CA}$.

\begin{corollary}
\label{cor:det_unique_maximizer}
Consider a system $\{a_i : i \in [K]\}$ that satisfies the assumptions in Theorem \ref{thm:main_result_deterministic}.
Define the height $r_+ \in (0,1)$ by
\begin{align*}
1 - r_{+} &:=  \left(1 - \frac{1}{n} \frac{1-\rho_n}{2-\rho_n} + \frac{3n-2}{n}  \frac{\rho_n}{2-\rho_n} + \frac{4}{n} \frac{1-\rho_n}{2-\rho_n}\Delta_{\CA}\right)^{\!\frac{1}{2n}} \\
&\,\,\leq \left(1 - \frac{1}{3n} + 2\rho_n + \frac{2}{n} \Delta_{\CA}\right)^{\!\frac{1}{2n}}.
\end{align*}
Then the $2K$ spherical caps $B_{r_+}(a_1), B_{r_+}(-a_1), B_{r_+}(a_2), \ldots, B_{r_+}(-a_K)$ are disjoint.  Further for each $s \in \{-1, +1\}$ and $i \in [K]$, there exists exactly one first-order critical point $x_*$ of $F_{\hat \CA}$ in $B_{r_+}(sa_i)$.  Morerover $x_*$ is a strict local maximum of $F_{\hat \CA}$ (hence second-order critical), and
 \begin{equation*}
     \| x_* - sa_i \|_2^2 \leq \frac{2\Delta_{\CA}}{n}.
 \end{equation*}
\end{corollary}

\begin{proof}
First, we check that $r_+ \in (0,1)$.  Indeed $r_+ > 0$ is clear from \eqref{eq:deterministic-bound-rho-n}, and $r_+ < 1$ is equivalent to $(3n-1)\rho_n + 4(1-\rho_n)\Delta_{\CA}<1$ which holds by the assumptions in Theorem~\ref{thm:main_result_deterministic}.
Also $1 - r_+ \leq (1 - \tfrac{1}{3n} + 2\rho_n + \tfrac{2}{n}\Delta_{\CA})^{\frac{1}{2n}}$ as stated, by the bound on $\rho_n$ \eqref{eq:deterministic-bound-rho-n}.

Next, we check that the $2K$ spherical caps are disjoint.  For a contradiction, suppose there exist distinct $(s, i), (s', i') \in \{-1,1\} \times [K]$ and $x \in \mathbb{S}^{D-1}$ such that $x \in B_{r_+}(sa_i) \cap B_{r_+}(s'a_{i'})$.    Thus $\langle x, sa_i \rangle \geq 1 - r_+$ and $\langle x, s' a_{i'} \rangle \geq 1 - r_{+}$.  By the triangle inequality with respect to geodesic distances on the sphere, we know $\arccos{\langle sa_i, s' a_{i'} \rangle} \leq \arccos{\langle sa_i, x \rangle} + \arccos{\langle x, s' a_{i'} \rangle} \leq 2 \arccos{(1 - r_+)}$, hence by the double angle formula for cosine $\langle sa_i, s' a_{i'} \rangle \geq 2(1-r_+)^2 - 1$.  But note that this gives a lower bound on $\rho_2$:  letting $y \in \mathbb{S}^{D-1}$ be the midpoint of the arc joining $sa_i$ and $s'a_{i'}$, we have
\begin{align*}
    \rho_2 \geq \langle y, a_i \rangle^2 + \langle y, a_{i'} \rangle^2 - 1 = \langle y, sa_i \rangle^2 + \langle y, s'a_{i'} \rangle^2 - 1 = \langle sa_i, s' a_{i'} \rangle  \geq 2(1 - r_{+})^2 - 1,
\end{align*}
where the second equality is again by the double angle formula.  Therefore
\begin{align*}
\rho_2 \geq 2(1-r_+)^2 - 1 \geq 2\left(1 - \frac{1}{2n}\right)^{\frac{1}{n}} - 1 > 2\left(1 - \frac{1}{2n}\right) - 1 = 1 - \frac{1}{n},
\end{align*}
which contradicts the implied upper bound on $\rho_2$ \eqref{eq:deterministic-bound-rho-2}.  So the $2K$ spherical caps are disjoint.

For the remaining statements we apply Proposition~\ref{prop:concavity_and_local_maxes}.  Thus it suffices to check that $F_{\hat \CA}$ is  strictly $n$-concave on $B(sa_i, 1 - \frac{\Delta_{\CA}}{n})$ and strictly concave in the interior of $B_{r_{+}}(sa_i)$.  To verify these, we will use Proposition \ref{prop:det_concavity} to upper-bound the eigenvalues of the Riemmanian Hessian of $F_{\hat \CA}$.

So for the strict $n$-concavity condition, we want to show:
\begin{equation} \label{eq:want-n-strict}
    -2n + 2n^2 \frac{2 - \rho_n}{1 - \rho_n}(1 - (1 - \frac{\Delta_{\CA}}{n})^{2n}) + 2n(3n-2) \frac{\rho_n}{1-\rho_n} + 8n \Delta_{\CA} < -n.
\end{equation}
Since $(1 - x)^{2n} \geq 1 - 2nx$ for $x \in \mathbb{R}$, we have
\begin{equation*}
    \left(1 - \frac{\Delta_{\CA}}{n}\right)^{2n} \geq 1 - 2 \Delta_{\CA},
\end{equation*}
whence \eqref{eq:want-n-strict} is implied by
\begin{equation}\label{eq:want-n-strict-2}
    -2n + 4n^2 \frac{2 - \rho_n}{1 - \rho_n}\Delta_{\CA} + 2n(3n-2) \frac{\rho_n}{1-\rho_n} + 8n \Delta_{\CA} < -n.
\end{equation}
Rearranging terms, Eq.~\eqref{eq:want-n-strict-2} is equivalent to
\begin{equation}
    \left( 4n(2 - \rho_n) + 8(1 - \rho_n) \right) \Delta_{\CA} + (6n-3)\rho_n < 1. \label{eq:this-works}
\end{equation}
However indeed, \eqref{eq:this-works} holds as a consequence of the assumptions in Theorem~\ref{thm:main_result_deterministic}. 

As for the concavity condition on $B_{r_{+}}(sa_i)$, we want to show:
\begin{equation} \label{eq:want-n-strict2}
    -2n + 2n^2 \frac{2 - \rho_n}{1 - \rho_n}(1 - (1 - r_{+})^{2n}) + 2n(3n-2) \frac{\rho_n}{1-\rho_n} + 8n \Delta_{\CA} \leq 0.
\end{equation}
However upon substituting in the definition of $r_{+}$, we see that the left-hand side of \eqref{eq:want-n-strict2} is precisely $0$, i.e., $r_+$ was chosen to be maximal so that  \eqref{eq:want-n-strict2} holds.  This completes the proof of the corollary.
\end{proof}

\medskip

\subsection{Completing the proof of Theorem \ref{thm:main_result_deterministic}}

Finally, we can conclude the proof of Theorem \ref{thm:main_result_deterministic} by combining
the previous three steps.
\begin{proof}[Proof of Theorem \ref{thm:main_result_deterministic}]
Let $\ell = \frac{3n^2 + 2}{2\tau} \Delta_{\CA} + \Delta_{\CA}$.
It suffices to verify the following two claims:
\vspace{-0.3em}
\begin{enumerate}
    \item Each second-order critical point $x^*$ of \textup{\ref{prob:nspm}}  satisfying $F_{\hat \CA}(x^*) > \ell$ must lie in one of the spherical caps $B_{r_+}(a_1), B_{r_+}(-a_1), \ldots, B_{r_+}(-a_K)$ from Corollary~\ref{cor:det_unique_maximizer}.
    \item The level $\ell$ satisfies $\ell < 1 - \Delta_{\CA}$.
\end{enumerate}
\vspace{-0.3em}
It is easy to see that these statements imply Theorem~\ref{thm:main_result_deterministic}.  By the first claim and Corollary~\ref{cor:det_unique_maximizer}, there are \textit{at most} $2K$ second-order critical points $x^*$ with $F_{\hat \CA}(x^*) > \ell$.
They are all strict local maximizers of $F_{\hat \CA}$.  For each such point, there exists unique $i \in [K], s \in \{-1,+1\}$ with $\| x^* - sa_i \|_2^2 \leq \frac{2\Delta_{\CA}}{n}$.  Meanwhile by the second claim, for each $s \in \{-1, +1\}, i \in [K]$ the  maximizer $x^*$ of $F_{\hat \CA}$ in $B_{r_+}(sa_i)$ from Corollary~\ref{cor:det_unique_maximizer} satisfies $F_{\hat \CA}(x) \geq F_{\hat \CA}(sa_i) \geq F_{\CA}(sa_i) - \Delta_{\CA} = 1-\Delta_{\CA} > \ell$. Thus there are \textit{at least} $2K$ second-order critical points in the superlevel set, and the theorem follows.

Now we verify the first claim.  Letting $\zeta = A^{\top}x^*$, we want $\|\zeta\|_{\infty} \geq 1 - r_+$. Indeed using the bound
\begin{align*}
\frac{\Delta_{\CA}}{F_{\CA}(x^*)} &\leq \frac{\Delta_{\CA}}{F_{\hat \CA}(x^*) - \Delta_{\CA}} < \frac{\Delta_{\CA}}{\ell - \Delta_{\CA}} = \frac{2\tau}{2+3n^2}\\
&=	\frac{2\tau}{3n^2} \frac{3n^2}{2+3n^2} \le \frac{2\tau}{3n^2} \frac{3n^2 + 4 \tau}{2+3n^2 + 4\tau} = \frac{2\tau}{3n^2} \left(1 - \frac2{2+3n^2 + 4\tau}\right)=\frac{2}{3n^2} (\tau - \Delta_0) \\
&\le \frac{2}{3n^2} (\tau - \Delta_{\CA}),
\end{align*}
where we used $\tfrac{a+c}{b+c}\ge \tfrac{a}{b}$ if $b\ge a\ge 0$ and $c\ge 0$, Proposition \ref{prop:one_big_deterministic_case} gives
\begin{align*}
   \| \zeta \|_{\infty}^2 \geq 1 - 2 \rho_2 - 2 \rho_n - 3 \frac{\Delta_{\CA}}{F_{\CA}(x)} &\geq 1 - 2 \rho_2 - 2 \rho_n - \frac{2}{n^2} (\tau - \Delta_{\CA}) \\ &= 1 - 2\rho_2 - 2\rho_n - \frac{2}{n^2}\left( \frac{1}{6} - \Delta_{\CA} - n^2 \rho_2 - (n^2 + n)\rho_n \right) \\
    &= 1 - \frac{1}{3n^2} + \frac{2}{n} \rho_n + \frac{2}{n^2}\Delta_{\CA} \\
    & \geq \left(1 - \frac{1}{3n} + 2\rho_n + \frac{2}{n}\Delta_{\CA} \right)^{\!\frac{1}{n}} \\
    &\geq (1 - r_+)^2.
\end{align*}
Note that the penultimate line used  $(1 - x)^{\frac{1}{n}} \leq 1 - \frac{x}{n}$ for $x \leq 1$, and the last line was a part of Corollary~\ref{cor:det_unique_maximizer}.

Next we check the second claim.  Substituting the definition of $\ell$, we require
\begin{align*}
    \frac{3n^2 + 2}{2\tau} \Delta_{\CA} + \Delta_{\CA} < 1 - \Delta_{\CA} \,\,\, \Longleftrightarrow \,\,\, \Delta_{\CA} < \frac1{\frac{3n^2 + 2}{2\tau} + 2} = \Delta_0.
\end{align*}
But this is guaranteed by the assumption $\Delta_{\CA}<\Delta_0$.  The proof of Theorem~\ref{thm:main_result_deterministic} is finished.
\end{proof}

\smallskip

\section{Proof of Theorem~\ref{thm:main_result_for_now}}
\label{subsec:proof_theorem_main_overcomplete}
In this section we prove our random overcomplete theorem.  Here we define $\zeta$, $\sigma$ and $R_{i,s}$ as in Definition~\ref{def:basic-setup}.
We recall the shorthand
$\varepsilon_K = K\log^n(K)/D^n$. Furthermore, as required by the statement of Theorem~\ref{thm:main_result_for_now}, we assume that conditions hold \ref{enum:RIP}-\ref{enum:GInverse} hold.

\subsection[Bounds involving error terms]{Bounds on $R_{i,s}$}
A central part of the proof is to bound the remainder $|R_{i,s}|$. We start with a general result that bounds inner products with the
basis coefficients $\sigma$. 
Then we deduce a first bound on $R_{i,s}$.

\begin{lemma}\label{lem:sigma_eta_bound}
Assume that conditions \textup{\ref{enum:correlation_supp}} and \textup{\ref{enum:GInverse}} hold.
	For any vector $\xi \in \bbR^K$, we have
	\begin{equation} \label{eq:inner-product-sigma}
		\SN{\sigma^\top \xi} \le \sqrt{c_2 F_{\CA}(x)} \|\xi\|_2.
	\end{equation}
	In particular, letting $R_s = \max_i |R_{i,s}|$ as in Definition~\ref{def:basic-setup}, we have
	\begin{equation}\label{eq:sqrtfxrhobound}
		R_s \le \sqrt{c_2 c_1^s F_{\CA}(x)} \, \varepsilon_K ^{\frac{s}{2n}} \N{\zeta}_{2n}^{n-s}\!.
	\end{equation}
\end{lemma}
\begin{proof}
	Since $G_n^{-1}$ is symmetric and positive definite, we can reason as follows for the first claim:
	\begin{align*}
	    \SN{\sigma^{\top} \xi} &= \SN{(\zeta^{\odot n})^{\top} G_n^{-1} \xi} & (\textup{Eq.~\eqref{eq:coefficient_identity}}) \\[0.15em]
	    & \leq \sqrt{(\zeta^{\odot n})^{\top} G_n^{-1} \zeta^{\odot n}} \sqrt{\xi^{\top} G_n^{-1} \xi} & \textup{(Cauchy-Schwartz w.r.t. $G_n^{-1}$)} \\[0.15em]
	    & = \sqrt{F_{\CA}(x)} \sqrt{\xi^{\top}G_n^{-1}\xi} & \textup{(Eq.~\eqref{eq:norm_identity})}\\[0.15em]
	    & \leq \sqrt{F_{\CA}(x)} \sqrt{\|G_n^{-1}\|_2} \|\xi\|_2 & \\[0.15em]
	    & \leq \sqrt{c_2 F_{\CA}(x)} \| \xi \|_2 & \textup{(condition \ref{enum:GInverse})}.
	\end{align*}

	For \eqref{eq:sqrtfxrhobound}, we note that $R_{i,s} = \sigma^\top \xi$, where $\xi \in \mathbb{R}^K$ is given by $\xi_j = \mathbf{1}(j\;\!\!\!\neq\;\!\!\! i) \zeta_{j}^{n-s}\langle a_i, a_j\rangle^{s}$ (each $j \in [K]$).
	Substituting this into \eqref{eq:inner-product-sigma} gives \eqref{eq:sqrtfxrhobound}, because
	\begin{align*}
		\|\xi \|_2 &= \sqrt{\sum_{j: j\neq i}\zeta_{j}^{2n-2s}\langle a_i, a_j\rangle^{2s}} & \\[0.15em]
		& \leq \sqrt{\left( \sum_{j : j \neq i} \zeta_j^{2n}\right)^{\!\!\!\frac{2n-2s}{2n}} \left( \sum_{j : j \neq i} \langle a_i, a_j \rangle^{2n} \right)^{\!\!\!\frac{2s}{2n}}} & \textup{(H\"older's inequality)}\\[0.15em]
		& \leq \N{\zeta}_{2n}^{n-s}\left( K \left(\frac{c_1\log(K)}{D}\right)^{\!\!n}\right)^{\!\!\frac{s}{2n}} & \textup{(condition \ref{enum:correlation_supp})}\\[0.15em]
		&= c_1^{\frac{s}2}\varepsilon_K^\frac{s}{2n} \N{\zeta}_{2n}^{n-s},
	\end{align*}
	where the last equation follows from the definition of $\varepsilon_K$.  This completes the proof of the lemma.
\end{proof}

The restricted isometry property plays an important part in our analysis of the overcomplete regime.
In the next statement, we present the inequalities that we will use which rely on the RIP assumption.
Then as a corollary, we obtain a bound on the remainders $R_{i,s}$ that will be more useful than \eqref{eq:sqrtfxrhobound}.

\begin{lemma}[Consequences of RIP]\label{lem:RIP_inequalities}
Assume that conditions \textup{\ref{enum:RIP_supp}-\ref{enum:GInverse}} hold.
Let $x\in \bbS^{D-1}$ and let $\CI=\CI(x) \subseteq [K]$ satisfy the conclusions of Lemma~\ref{prop:RIP_consequence}.
Then it holds
\begin{equation}\label{eq:z2nnf_lemaux}
	\SN{F_{\CA}(x) - \sum_{i\in \CI} \sigma_i \zeta_i^n}  \le  \sqrt{ c_2 \tilde c_\delta^n F_{\CA}(x) \varepsilon_K}.
\end{equation}
Further, there exists a constant $C_1$, depending only on $n$, $c_1$, $c_2$ and $\tilde c_\delta$, and another constant $C_2$, depending only on $n$ and $\tilde c_{\delta}^n$, such that we have
\begin{equation} \label{eq:bound-zeta-2n}
\N{\zeta}_{2n}^{2n} \le F_{\CA}(x) + C_1 \sqrt{\varepsilon_K F_{\CA}(x)} + C_2 \varepsilon_K.
\end{equation}
\end{lemma}

\begin{proof}
    To show \eqref{eq:z2nnf_lemaux}, start with the identity $F_{\CA}(x) = \sum_{i=1}^K \sigma_i \zeta_i^n$, then apply the inner product bound \eqref{eq:inner-product-sigma} with $\xi \in \mathbb{R}^K$ given by $\xi_j = \mathbf{1}(j \notin \mathcal{I}) \zeta_j^n$, and then use Lemma~\ref{prop:RIP_consequence} based on RIP:
    \begin{align*}
        \SN{F_{\CA}(x) - \sum_{i \in \mathcal{I}} \sigma_i \zeta_i^n} &= \SN{\sum_{i \notin \mathcal{I}} \sigma_i \zeta_i^n} & \\[0.1em]
        & \leq \sqrt{c_2 F_{\CA}(x)} \sqrt{\sum_{i \notin \mathcal{I}} \zeta_i^{2n}} & \textup{(Eq.~\eqref{eq:inner-product-sigma})} \\[0.1em]
        & \leq \sqrt{c_2 F_{\CA}(x)} \sqrt{K \! \left(\tilde c_{\delta}^n  \frac{\log(K)}{D}\right)^{\!\!n}} & \textup{(Lemma~\ref{prop:RIP_consequence})} \\[0.1em]
        & = \sqrt{c_2 \tilde c_{\delta}^n F_{\CA}(x) \varepsilon_K}.
    \end{align*}
    Next to show \eqref{eq:bound-zeta-2n}, split $\N{\zeta}_{2n}^{2n}$ into a sum over $\mathcal{I}$ and $[K] \setminus \mathcal{I}$:
    \begin{align} \label{eq:zeta2n-1}
        \N{\zeta}_{2n}^{2n} = \sum_{i \in \mathcal{I}} \zeta_i^{2n} + \sum_{i \notin \mathcal{I}} \zeta_i^{2n}.
    \end{align}
    As before, the sum over $[K] \setminus \mathcal{I}$ is upper-bounded using the second conclusion in Lemma~\ref{prop:RIP_consequence}:
    \begin{align} \label{eq:zeta2n-2}
        \sum_{i \notin \mathcal{I}} \zeta_i^{2n} \leq \tilde c_{\delta}^{n} \varepsilon_K.
    \end{align}
    As for the sum over $\mathcal{I}$ in \eqref{eq:zeta2n-1}, we can relate this to $F_{\CA}(x)$ by using
    \begin{align} \label{eq:zeta2n-3}
        \sum_{i \in \mathcal{I}} \zeta_i^{2n} = \sum_{i \in \mathcal{I}} \sigma_i \zeta_i^n + \sum_{i \in \mathcal{I}} R_{i,n} \zeta_i^n.
    \end{align}
    Indeed we have already proven that the first sum on the right-hand side of \eqref{eq:zeta2n-3} is approximately $F_{\CA}(x)$; more precisely \eqref{eq:z2nnf_lemaux} holds.  We can see that the other sum is small by using our initial bound on $R_s$ \eqref{eq:sqrtfxrhobound} together with the first conclusion of Lemma~\ref{prop:RIP_consequence}:
    \begin{align*}
        \SN{\sum_{i \in \mathcal{I}} R_{i,n} \zeta_i^n} &\leq R_n \sum_{i \in \mathcal{I}} \SN{\zeta_i}^n \leq  R_n \sum_{i \in \mathcal{I}} \zeta_i^2 \leq R_n (1 + \delta) \\ &\leq (1+ \delta)\sqrt{c_2 c_1^n F_{\CA}(x) \varepsilon_K} \leq 2 \sqrt{c_2 c_1^n F_{\CA}(x) \varepsilon_K}.
    \end{align*}
    Putting the last three sentences together via the triangle inequality gives
    \begin{align*} 
        \SN{\sum_{i \in \mathcal{I}} \zeta_i^{2n} - F_{\CA}(x)} \leq \sqrt{c_2 \tilde  c_{\delta}^n F_{\CA}(x) \varepsilon_K} + 2\sqrt{c_2 c_1^n F_{\CA}(x) \varepsilon_K} .
    \end{align*}
    Combining \eqref{eq:zeta2n-1}, \eqref{eq:zeta2n-2}, \eqref{eq:zeta2n-3}, we conclude that we may set
    \begin{align*}
      C_1 = \tilde c_{\delta}^n  \quad \quad \textup{and} \quad \quad  C_2 = \sqrt{c_2 \tilde c_{\delta}^n}  + 2 \sqrt{c_2 c_1^n},
    \end{align*}
    and \eqref{eq:bound-zeta-2n} follows as desired.
\end{proof}

\begin{corollary}\label{cor:Rs_bound}
Assume that conditions \textup{\ref{enum:RIP_supp}-\ref{enum:GInverse}} hold.	Let $x \in \mathbb{S}^{D-1}$ be arbitrary and $s \in \{0, \ldots, n\}$.
Then there exists a constant $\tilde C_s$, depending only on $s$, $n$, $c_1$, $c_2$ and $\tilde c_\delta$, such that
	\begin{equation}\label{eq:sqrtfxrhobound2}
		R_s \le \tilde C_s \varepsilon_K ^{\frac{s}{2n}}  \max\{\varepsilon_K, F_{\CA}(x)\}^{1-\frac{s}{2n}}.
	\end{equation}

\end{corollary}

\begin{proof}
Plugging \eqref{eq:bound-zeta-2n} into \eqref{eq:sqrtfxrhobound} yields
\begin{align*}
    R_s \leq \sqrt{c_2 c_1^s F_{\CA}(x)} \, \varepsilon_K ^{\frac{s}{2n}} \left( F_{\CA}(x) + C_1 \sqrt{\varepsilon_K F_{\CA}(x)} + C_2 \varepsilon_K \right)^{\!\!\frac{n-s}{2n}} \!\! \leq \tilde C_s \varepsilon_K^{\frac{s}{2n}} \max\{\varepsilon_K, F_{\CA}(x)\}^{1-\frac{s}{2n}},
\end{align*}
where $\tilde C_s = \sqrt{c_2 c_1^s} \left( 1 + C_1 + C_2 \right)^{\! \frac{n-s}{2n}}$\!.
\end{proof}

\medskip

\subsection{Lower bound on maximum correlation coefficient}

We now provide a lower bound on $\|\zeta\|_\infty$ depending on $\frac{\varepsilon_K}{F_{\CA}(x)}$.

\begin{proposition}
	\label{prop:zetainf_lower_bound}
	Suppose that $x$ is a second-order critical point of \textup{\ref{prob:nspm}} and $F_{\CA}(x) \ge \varepsilon_K$.
	Then there exists a constant $\overline{C}_2$ depending only on $n$, $c_1$, $c_2$, $\tilde c_\delta, \delta$, and another constant $\overline{C}_n$ depending only on $n, c_1, c_2, \tilde c_{\delta}$, such that we have
	\begin{align}\label{eq:zetainf_lower_bound}
		\|\zeta\|_{\infty}^2 &\ge 1 - \frac{2n-1}{2n-2}\frac{\delta}{1+\delta} - \overline{C}_2 \left(\frac{\varepsilon_K}{F_{\CA}(x)}\right)^{\!\!\frac1{2}} - \overline{C}_n \left(\frac{\varepsilon_K}{F_{\CA}(x)}\right)^{\!\!\frac1{n}} - 3\frac{\Delta_{\CA}}{F_{\CA}(x)},
		\\
		\nonumber
		&\ge 1 - \frac{2n-1}{2n-2}\frac{\delta}{1+\delta} - (\overline{C}_2 + \overline{C}_n) \left(\frac{\varepsilon_K}{F_{\CA}(x)}\right)^{\!\!\frac1{n}} - 3\frac{\Delta_{\CA}}{F_{\CA}(x)}.
	\end{align}
\end{proposition}

\begin{remark}
Note that while the statement is for a constrained second-order critical point of $F_{\hat \CA}$,  \eqref{eq:zetainf_lower_bound} contains $F_{\CA}(x)$. This is intentional: later we use this result together with a bound on $F_{\CA}(x)$.
\end{remark}
\begin{proof}
    Our starting point is Eq.~\eqref{eq:starting_point_max_correlation_coeff} in Proposition~\ref{prop:auxliary_one_big},
which says that for each $i \in [K]$ we have
\begin{align*}
\left(1+2(n-1)\zeta_i^2\right) F_{\CA}(x) &\geq (2n-1)\sigma_i\zeta_i^{n-2} + n\zeta_i^{n-2}R_{i,n} + (n-1)R_{i,2} - (4n-2)\Delta_{\CA}.
\end{align*}
Multiplying both sides by $\zeta_i^2$ and using that $1 + 2(n-1)\| \zeta \|_{\infty}^2 \geq 1+ 2(n-1)\zeta_i^2$, we have
\begin{align} \label{eq:nice-startingpoint}
\zeta_i^2 \left(1+2(n-1)\|\zeta\|_{\infty}^2\right) \!F_{\CA}(x) &\geq (2n-1)\sigma_i\zeta_i^{n} + \zeta_i^2 \left(n\zeta_i^{n-2}R_{i,n} + (n-1)R_{i,2} - (4n-2)\Delta_{\CA}\right)\!.
\end{align}
Now let $\mathcal{I} = \mathcal{I}(x) \subseteq [K]$ satisfy the conclusions of Lemma~\ref{prop:RIP_consequence}, and sum the inequalities \eqref{eq:nice-startingpoint} as $i$ ranges over $\mathcal{I}$.
Denoting $\|\zeta\|_{2, \mathcal{I}}^2 = \sum_{i \in \mathcal{I}} \zeta_i^2$, this gives
\begin{align} \label{eq:lower-max-0}
    \|\zeta\|_{2, \mathcal{I}}^2 \left(1 + 2(n-1) \|\zeta\|_{\infty}^2 \right) F_{\CA}(x) \geq (2n-1) \sum_{i \in \mathcal{I}} \sigma_i \zeta_i^n + \sum_{i \in \mathcal{I}} \zeta_i^2(n & \zeta_i^{n-2} R_{i,n} +(n-1)R_{i,2} \nonumber \\[-0.8em] &-(4n-2)\Delta_{\CA}).
\end{align}
We lower-bound each of the sums on the right-hand side in turn.  Firstly by \eqref{eq:z2nnf_lemaux},
\begin{align} \label{eq:lower-max-1}
    (2n-1) \sum_{i \in \mathcal{I}} \sigma_i \zeta_i^n \geq (2n-1) \left( F_{\CA}(x) - \sqrt{c_2 \tilde c_{\delta}^n F_{\CA}(x) \varepsilon_K} \right).
\end{align}
Secondly by combining the triangle inequality, the trivial bound $|\zeta_i^{n-2}| \leq 1$,  Corollary~\ref{cor:Rs_bound} with $s=2, n$ and the assumption $F_{\CA}(x) \geq \varepsilon_K$, we have
\begin{align} \label{eq:lower-max-2}
     & \sum_{i \in \mathcal{I}} \zeta_i^2(n \zeta_i^{n-2} R_{i,n} +(n-1)R_{i,2}  -(4n-2)\Delta_{\CA}) \nonumber \\
     & \geq - \| \zeta \|_{2, \mathcal{I}}^2 \left( n R_n + (n-1) R_{2} + (4n-2) \Delta_{\CA} \right) \nonumber \\
     & \geq - \| \zeta \|_{2, \mathcal{I}}^2 \left(n \tilde C_n \sqrt{\varepsilon_K F_{\CA}(x)} + (n-1) \tilde C_2 \varepsilon_K^{\frac{1}{n}} F_{\CA}(x)^{\frac{n-1}{n}} + (4n-2) \Delta_{\CA} \right).
\end{align}
Inserting \eqref{eq:lower-max-1} and \eqref{eq:lower-max-2} into \eqref{eq:lower-max-0} gives
\begin{align*}
    \| \zeta \|_{2, \mathcal{I}}^2 ( 1 + 2(n-1) & \| \zeta \|_{\infty}^2) F_{\CA}(x) \geq  (2n-1) \left( F_{\CA}(x) - \sqrt{c_2 \tilde c_{\delta}^n F_{\CA}(x) \varepsilon_K} \right) \\ & - \| \zeta \|_{2, \mathcal{I}}^2 \left(n \tilde C_n \sqrt{\varepsilon_K F_{\CA}(x)} + (n-1) \tilde C_2 \varepsilon_K^{\frac{1}{n}} F_{\CA}(x)^{\frac{n-1}{n}} + (4n-2) \Delta_{\CA} \right).
\end{align*}
Isolating $\| \zeta \|_{\infty}^2$, and using $1 - \delta \leq \| \zeta \|_{2, \mathcal{I}}^2 \leq 1 + \delta$ and $\frac{4n-2}{2n-2} \leq 3$, yields \eqref{eq:zetainf_lower_bound} as desired where
\begin{align*}
    \overline{C}_2 = \frac{1}{1-\delta} \frac{\sqrt{c_2 \tilde c_{\delta}^n}}{2n-2} + \frac{n}{2n-2} \tilde C_{n} \quad \quad \textup{and} \quad \quad \overline{C}_n = \frac{1}{2} \tilde C_2.
\end{align*}
This completes the proof of the proposition.
\end{proof}

	\medskip

\subsection{Concavity}

We now show that $F_{\CA}$ is strictly geodesically concave in a spherical cap around each $\pm a_i$.

\begin{proposition}\label{prop:concavity_plus}
	There exist constants $C$, $\Delta_0$ and $D_0$, where $C$ and $\Delta_0$ depend only on $n$, $c_2$, and $D_0$ depends additionally on $c_1$, $\tilde c_\delta$, such that $C<1$ and if $\Delta_{\CA} <  \Delta_0$, and $D\ge D_0$, then for all $i\in [K]$ and $s\in \{-1, 1\}$, it holds
	\begin{enumerate}
	\item $F_{\hat \CA}$ is strictly concave on the spherical cap $B_{r_+}(s a_i)$, where	
	\begin{align*}
		r_{+} &:=  1- \left(1-C\left(0.99 - 4\Delta_{\CA}
		\right)^2\right)^{\frac1{2n}}.
	\end{align*}
	\item There exists exactly one local maximum in each spherical cap $B_{r_+}(s a_i)$, and denoting it by $x_*$, we have
	\begin{equation}\label{eq:existence_max}
		\| x_* - s a_i \|^2 \le \frac{2\Delta_{\CA}}{n}.
	\end{equation}
	\end{enumerate} 
\end{proposition}

\begin{proof}
We first note that Lemma~\ref{lem:RIP_inequalities} implies that, for all $y\in \bbS^{D-1}$,
\begin{align*}
\sum_{i=1}^{K}\langle y, a_i\rangle^{2n} =\|\zeta(y)\|_{2n}^{2n} &\le F_{\CA}(y) + C \sqrt{\varepsilon_K \max\{\varepsilon_K, F_{\CA}(y)\}}
\le 1 + O(\sqrt{\varepsilon_K}),
\end{align*}
thus since $\lim_{D\to \infty} \varepsilon_K =0$, for all $\mu \in (0,1)$ there exist $D_0$ such that if $D\ge D_0$,
\begin{align}
	\label{eq:concativity_minimal_assumption}
	\sum_{i=1}^{K}\langle y, a_i\rangle^{2n} \leq 1 + \mu,\quad \textrm{ uniformly for all } y \in \bbS^{D-1}.
\end{align}

Let $i\in [K]$, $s\in \{-1, 1\}$ and $y \perp x$ arbitrary and recall \eqref{eq:lb_hessian_value},
\begin{align}
	\label{eq:concav_tbc}
	\frac{1}{2n}y^\top \nabla_{\bbS^{D-1}}^2 F_{\hat \CA}(x) y \leq n\N{P_{\CA}(x^{n-1}y)}^2 + (n-1)\langle P_{\CA}(x^n), x^{n-2}y^2\rangle - \zeta_i^{2n} + 4\Delta_{\CA}.
\end{align}
The second term can be bounded using Lemma~\ref{lem:DeltaCA_techlemma}.
\begin{align*}
	\SN{\langle P_{\CA}(x^n), x^{n-2}y^2\rangle} &\leq \SN{\langle x^n, x^{n-2}y^2\rangle} + \SN{\langle P_{\CA}(x^n) -x^n, x^{n-2}y^2\rangle}\\
	&\leq 0 + \N{P_{\CA^\perp}(x^n)}_F \|\sym(x^{n-2}y^2)\|\\ &= \frac{\sqrt 2}{\sqrt{n(n-1)}} \sqrt{1-\N{P_{\CA}(x^n)}_F^2} \leq \frac{2}{n-1}\sqrt{1-\langle a_i, x\rangle^{2n}}.
\end{align*}
The first term in \eqref{eq:concav_tbc} is slightly more involved. We first
use Lemma~\ref{lem:frob_norms_extended} and \ref{enum:GInverse_supp} to write
\begin{align*}
	\N{P_{\CA}(x^{n-1}y)}_F^2 = \beta(x^{n-1} y)^\top G_n^{-1} \beta(x^{n-1}y) \leq c_2 \N{\beta(x^{n-1} y)}_2^2.
\end{align*}
Then, by leveraging the fact that $y\perp x$, we further get
\begin{align*}
	\N{\beta(x^{n-1},y)}_2^2 &= \sum_{j=1}^{K}\langle a_j, x\rangle^{(2n-2)}\langle a_j, y\rangle^2 \leq \langle a_i, x\rangle^{2n-2}\langle a_i, y\rangle^2 + \sum_{j\neq i}^{K}\langle a_j, x\rangle^{(2n-2)}\langle a_j, y\rangle^2\\
	&\leq \langle a_i, x\rangle^{2n-2}(\N{a_i}^2 - \langle a_i, x\rangle^2) + \sum_{j\neq i}^{K}\langle a_j, x\rangle^{(2n-2)}\langle a_j, y\rangle^2\\
	&= \langle a_i, x\rangle^{2n-2} - \langle a_i, x\rangle^{2n} + \sum_{j\neq i}^{K}\langle a_j, x\rangle^{(2n-2)}\langle a_j, y\rangle^2,
\end{align*}
where $\langle a_i, y\rangle^2 + \langle a_i, x\rangle^2 \le \|a_i\|^2$ is a consequence of Bessel's inequality.
On one hand, since the function $f(x) = (1-x)^{\frac{n-1}{n}} - (1-x)$ is concave in $[0,1]$, and $f'(0)=\frac1{n}$, we have $f(x)\le \frac{x}{n}$, which implies $\langle a_i, x\rangle^{2n-2}(1-\langle a_i, x\rangle^2)\le \frac{1-\langle a_i, x\rangle^{2n}}{n}$.  On the other hand, using H\"older's inequality with $p=n$ and $q = \frac{n}{n-1}$
satisfying $1/n + (n-1)/n = 1$, and Equation \eqref{eq:concativity_minimal_assumption}, we get
\begin{align*}
	\sum_{j\neq i}^{K}\langle a_j, x\rangle^{(2n-2)}\langle a_j, y\rangle^2 &\leq \left(\sum_{j\neq i}^{K}\langle a_j, x\rangle^{2n}\right)^{\frac{n-1}{n}}
	\left(\sum_{j\neq i}\langle a_j, y\rangle^{2n}\right)^{\frac{1}{n}}\\
	&\leq (1+\mu)^{\frac{1}{n}}\left(1+\mu -\langle a_i, x\rangle^{2n}\right)^{\frac{n-1}{n}}\\
	&\le \left(1+\frac{\mu}{n}\right)\left(1+\mu -\langle a_i, x\rangle^{2n}\right)^{\frac{n-1}{n}}.
\end{align*}
Combining the estimates and writing $\eta := 1-\langle a_i, x\rangle^{2n}$ for short,  we obtain
\begin{align}
	\nonumber \frac{1}{2n}y^\top \nabla_{\bbS^{D-1}}^2 F_{\hat\CA}(x) y &\leq
	\N{G_n^{-1}}_2\left(\eta + (n+\mu)\left(\eta+\mu \right)^{\frac{n-1}{n}}\right) + (n-1)\sqrt{\eta}-1+\eta + 4\Delta_{\CA}\\
	\nonumber &\leq \N{G_n^{-1}}_2\left(\eta + (n+\mu)(\eta^{\frac{n-1}{n}}+\mu^{\frac{n-1}{n}})\right) + (n-1)\sqrt{\eta}-1+\eta + 4\Delta_{\CA}\\
	\nonumber &\leq \N{G_n^{-1}}_2 \left((1+n+\mu)\sqrt{\eta}  +(n+\mu)\mu^{\frac{n-1}{n}}\right) + n\sqrt{\eta} +4\Delta_{\CA} - 1\\
	\nonumber
	 &\le \left(n+(n+\mu+1)c_2\right)\sqrt{\eta} + (n+\mu)c_2 \mu^{\frac{n-1}{n}} + 4\Delta_{\CA} - 1.
\end{align}
Now choosing $\mu$ (and consequently $D_0$) such that $(n+\mu)c_2 \mu^{\frac{n-1}{n}} \le 0.01$, noticing that $\mu c_2 \le (n+\mu)c_2 \mu^{\frac{n-1}{n}}$, and setting $C=\left(\frac{1}{n+0.01+(n+1)c_2}\right)^2$, we get
\begin{equation}
\label{eq:concavity_oc_no_mu}
\frac{1}{2n}y^\top \nabla_{\bbS^{D-1}}^2 F_{\hat\CA}(x) y \le \sqrt{\frac{\eta}{C}} +  4\Delta_{\CA} - 0.99,
\end{equation}
and thus, by rearranging terms, we get that $F_{\hat\CA}$ is strictly geodesically concave in $B_{r_+}{s a_i}$.

We now define 
\begin{align*}
	\tilde r_+:= 1 - \left(1-C\left(0.49-4\Delta_{\CA}\right)^2\right)^{\frac{1}{2n}},
\end{align*}
and notice that for all $x\in B_{\tilde r_+}(s a_i)$ and $y\in \bbS^{D-1}$, \eqref{eq:concavity_oc_no_mu} implies that $y^\top \nabla_{\bbS^{D-1}}^2 F_{\hat\CA}(x) y \le -n$. We also have
\begin{align*}
	1-\sqrt[2n]{1-C\left(0.49-4\Delta_{\CA}\right)^2}
	&>1 - \left(1-\frac{C}{2n}\left(0.49-4\Delta_{\CA}\right)^2\right) =0.49^2\frac{C}{2n}\left(1-\frac{4}{0.49}\Delta_{\CA}\right)^2 \\
	&\ge 0.49^2\frac{C}{2n}\left(1-\frac{8}{0.49}\Delta_{\CA}\right)>\frac{\Delta_{\CA}}{n},
\end{align*}
where the last equality follows if we choose $\Delta_0 = 0.49^2\frac{C}{2n} / \left(\frac1n + 3.92\frac{C}{2n}\right)$, since $\Delta_{\CA}<\Delta_0$. Therefore, applying Proposition~\ref{prop:concavity_and_local_maxes}, we obtain that it exists exactly one local maximum $x_*\in B_{\tilde r_+}(s a_i)$, and $x_*$ satisfies \eqref{eq:existence_max}.

Finally, since $F_{\hat\CA}$ is strictly geodesically concave on $B_{r_+}(s a_i)$, Proposition~\ref{prop:concavity_and_local_maxes} implies that there exists at most one local maximum in $B_{r_+}(s a_i)$, thus since $B_{\tilde r_+}(s a_i)\subset B_{r_+}(s a_i)$, $x_*$ is also the only local maximum of~$B_{r_+}(s a_i)$.
\end{proof}

\medskip

\subsection{Completing the proof}

We now have all the pieces to conclude the argument for Theorem~\ref{thm:main_result_for_now}.
\begin{proof}[Proof of Theorem~\ref{thm:main_result_for_now}]
Suppose that $x$ is a local maximum and $F_{\hat\CA}(x)\ge  C \varepsilon_K + 5 \Delta_{\CA}$, where $C$ will be defined below.
Our proof strategy is to show that this implies that $x$ is in one of the $2K$ spherical caps defined in Proposition~\ref{prop:concavity_plus}. The rest of the proof then follows from Proposition~\ref{prop:concavity_plus}, since each spherical cap contains exactly one local maximum.

Denote $\tilde C = \overline{C}_2 + \overline{C}_n$ from the statement of Proposition~\ref{prop:zetainf_lower_bound}, and define $C$ so that $C\ge 1$. Therefore $F_{\hat\CA}(x)\ge \varepsilon_K$ and Proposition~\ref{prop:zetainf_lower_bound} implies
\begin{equation*}
	\|\zeta\|_{\infty}^2 \ge 1 - \frac{2n-1}{2n-2}\frac{\delta}{1+\delta} -\tilde C \left(\frac{\varepsilon_K}{F_{\CA}(x)}\right)^{\frac1{n}}  - 3\frac{\Delta_{\CA}}{F_{\CA}(x)}.
\end{equation*}
By the AM-GM inequality, we have
\begin{equation*}
\tilde C \left(\frac{\varepsilon_K}{F_{\CA}(x)}\right)^{\frac1{n}} \le \frac1{n}\left(4^{n-1} \tilde C^n \frac{\varepsilon_K}{F_{\CA}(x)} + \frac14(n-1)\right),
\end{equation*}
thus
\begin{align*}
\tilde C \left(\frac{\varepsilon_K}{F_{\CA}(x)}\right)^{\frac1{n}}  + 3\frac{\Delta_{\CA}}{F_{\CA}(x)} &\le \frac1{F_{\CA}(x)}\left(\frac{4^{n-1}}{n} \tilde C^n \varepsilon_K + 3\Delta_{\CA} \right) + \frac14 - \frac1{4n}.
\end{align*}
We now define $C:= \frac13 4^{n} \tilde C^n$, thus having
\begin{equation*}
F_{\CA}(x) \ge F_{\hat\CA}(x) -  \Delta_{\CA} \ge C \varepsilon_K + 4\Delta_{\CA} \ge \frac{4}{3}\left(\frac{4^{n-1}}{n}\tilde C^n \varepsilon_K + 3 \Delta_{\CA}\right),
\end{equation*}
and
\begin{align*}
	\tilde C \left(\frac{\varepsilon_K}{F_{\CA}(x)}\right)^{\frac1{n}}  + 3\frac{\Delta_{\CA}}{F_{\CA}(x)}
	&\le \frac34 + \frac14 - \frac1{4n}\le 1 - \frac1{4n}.
\end{align*}
If we now choose $\delta_0$ such that $\frac{2n-1}{2n-2}\frac{\delta_0}{1+\delta_0}\le \frac1{8n}$, we have that if $\delta< \delta_0$, then
\begin{align}\label{eq:1/8nbound}
	\|\zeta\|_{\infty}^2 \ge 1 - \frac{2n-1}{2n-2}\frac{\delta}{1+\delta} - \left(1-\frac1{4n}\right) \ge \frac1{8n}.
\end{align}
Note that we always have $F_{\CA}(x)\ge \|\zeta\|_{\infty}^{2n}$. To show this, define $i$ such that $\|\zeta\|_{\infty} = |\zeta_i|$. Then $F_{\CA}(x) = \|P_{\CA}(x^n)\|^2 \ge \|P_{a_i^n}(x^n)\|^2 = \zeta_i^{2n} = \|\zeta\|_{\infty}^{2n}$. Therefore, \eqref{eq:1/8nbound} implies that  $F_{\CA}(x)\ge\left(\frac1{8n}\right)^{n}$, and since $\lim_{D\to \infty} \varepsilon_K = 0$, for any $\nu\in (0,1)$, there exists $D_0$ such that if $D\ge D_0$, then
\begin{equation*}
\tilde C \left(\frac{\varepsilon_K}{F_{\CA}(x)}\right)^{\frac1{n}} \le 8n \tilde C \varepsilon_K^{\frac1{n}} \le \frac12 \nu.
\end{equation*}

Furthermore, choosing $\delta_0$ such that $\frac{2n-1}{2n-2}\frac{\delta_0}{1+\delta_0}\le \min\left\{\frac1{8n}, \frac12  \nu\right\}$, we obtain, by applying Proposition~\ref{prop:zetainf_lower_bound} again, that for all $\nu \in (0,1)$ there exists $D_0$ and $\delta_0$ such that, if $D\le D_0$ and $\delta\le \delta_0$,

\begin{equation}
\label{eq:zeta_infty_last_bound}
	\|\zeta\|_{\infty}^2 \ge 1 - \nu - 3\frac{\Delta_{\CA}}{F_{\CA}(x)}.
\end{equation}
Simply to improve the implicit constant, we can choose $\nu = 0.01$, and notice that $F_{\CA}(x)\ge 4 \Delta_{\CA}$, to get $\|\zeta\|_{\infty}^2 \ge 0.96/4$, and consequently, $F_{\CA}(x) \ge (0.96/4)^{n}$. Plugging this in \eqref{eq:zeta_infty_last_bound}, we obtain
$$\|\zeta\|_{\infty}^2 \ge 1 - \nu - 3(4/0.96)^n\Delta_{\CA}.$$
We now further choose $\Delta_0$ and a smaller $\nu$ such that, if $\Delta_{\CA}<\Delta_0$, then
\begin{align}
\nonumber
\sqrt[n]{1-C\left(0.99 - 4\Delta_{\CA}\right)^2} &\le 1-\frac{C}{n}\left(0.99 - 4\Delta_{\CA}\right)^2\\
\nonumber
&\le 1-\frac{C}{n}\left(0.98 - 7.92\Delta_{\CA}\right)\\
\label{eq:main_resultc5delta0}
&\le 1 - \nu - 3(4/0.96)^n \Delta_{\CA}\le \|\zeta\|_{\infty}^2.
\end{align}
This then implies that $x$ must lie in one of the caps defined in Proposition~\ref{prop:concavity_plus}, and the rest of the result follows from Proposition~\ref{prop:concavity_plus}. For \eqref{eq:main_resultc5delta0} to hold, it suffices to choose $\nu = \min(0.49\frac{C}{n},0.01)$ and $\Delta_0 = 0.49\frac{C}{n} / \left(3(4/0.96)^n + 7.92\frac{C}{n}\right)$.
\end{proof}

\section{Deflation bounds and proof of Theorem~\ref{thm:SPM_deflation_analysis}}
\label{sec:deflation}

Suppose we want to decompose the tensor $\hat T$, which is an estimate of the low rank tensor $T=\sum_{i=1}^K \lambda_i a_i^{\otimes m}$. We will consider the following slight modification of the SPM algorithm \cite{kileel2019subspace}.

\begin{algorithm}
	\caption{SPM with modified deflation step}
	\label{alg:modified_SPM}
	\begin{algorithmic}
		\Require $\hat T \in \CT_{D}^m$  
		\algrenewcommand\algorithmicrequire	{\textbf{Hyper-parameters:}}
		\Require $\alpha, \tau \in \bbR^+$ \quad \quad \phantom{a}
		\Ensure $(\hat a_k, \hat \lambda_k)_{k\in [K]}$
		\State $\hat M \gets \opmyreshape(\hat T, [D^{n}, D^{m-n}])$.
		\State Set $K$ as the number of singular values of $\hat M$ exceeding $\alpha$.
		\State Set $\hat M_K$ as the rank-$K$ truncation of the SVD of $\hat M$. 
		\State Let $\hat{\mathcal{A}}=\operatorname{colspan}(\hat U)$ and $\hat{\mathcal{A}}_{1} = \hat{\mathcal{A}}$
		\For{$k=1,\dots,K$}
			\State Obtain $\tilde a_k$ by applying \textsc{Power Method} \cite{kileel2019subspace} on functional $F_{\hat{ \mathcal{A}}_{k}}$ until convergence.
			\State Repeat last step with new initializations until $F_{\hat{ \mathcal{A}}_{k}}(\tilde a_k)\ge \tau$.
			\If{$k=1$}{ $\hat a_1 \gets \tilde a_1$} 
				\Else  \ Obtain $\hat a_k$ by applying \textsc{Power Method} on  $F_{\hat{ \mathcal{A}}}$ with $\tilde a_k$ as starting point until convergence.
				\EndIf
			\State \vspace{-15pt}
			\begin{align}\label{eq:lambda_formula}
				\hat \lambda_k \gets \frac{1}{\opvec(\hat a_k^{\otimes n})^\top (\hat M_K^\top)^\dagger \opvec(\hat a_k^{\otimes m-n})} \hspace{170pt}&
			\end{align}
			\If{$k<K$}
			\State $\hat{ \mathcal{A}}_{k+1} \gets \hat{ \mathcal{A}}_{k} \cap ( (\hat M_K^\top)^\dagger \opvec(\hat a_k^{\otimes m-n}))^\perp$
			\EndIf
		\EndFor
	\end{algorithmic}
\end{algorithm}

\newpage

We analyze Algorithm~\ref{alg:modified_SPM} assuming the conditions of Theorem \ref{thm:main_result_deterministic} or \ref{thm:main_result_for_now}.  The conclusion will be Theorem~\ref{thm:SPM_deflation_analysis}. We note the sign and permutation ambiguity there are inherent to the CP decomposition.

A result on the stability of matrix pseudo-inversion is needed; we include a proof for completeness.

\begin{lemma}\label{lem:pseudo-inverse-bound}
	Let $W, \hat W\in \bbR^{p\times q}$ and suppose that $W$ has exactly $r$ nonzero singular values\linebreak $\sigma_1(W) \geq \ldots \geq \sigma_r(W) > 0$ and $\Delta_W := \| W - \hat W \|_2 < \sigma_r(W)$. Define $\hat W_r$ as the SVD of $\hat W$ truncated at the $r$-th singular value. Then
	$$\|W^\dagger - \hat W_r^\dagger\|_2 \le \frac{\Delta_W}{\sigma_r(W) - \Delta_W} \left(\frac{2}{\sigma_r(W)} + \frac{1}{\sigma_r(W) - \Delta_W}\right).$$
\end{lemma}

\begin{proof}
	We denote the singular value decomposition of matrices $W,\hat W \in \bbR^{p\times q}$ as
	\begin{align*}
		W= \begin{pmatrix}
			U_1 & U_2
		\end{pmatrix}
		\begin{pmatrix}
			\Sigma_1 & 0 \\
			0 & \Sigma_2
		\end{pmatrix}
		\begin{pmatrix}
			V_1^\top \\
			V_2^\top
		\end{pmatrix}
		\quad \textrm{ and }\quad
		\hat W = \begin{pmatrix}
			\hat U_1 & \hat U_2
		\end{pmatrix}
		\begin{pmatrix}
			\hat \Sigma_1 & 0 \\
			0 & \hat \Sigma_2
		\end{pmatrix}
		\begin{pmatrix}
			\hat V_1^\top \\
			\hat V_2^\top
		\end{pmatrix},
	\end{align*}
	where $\Sigma_1$ and $\hat\Sigma_1$ are diagonal matrices with the largest $r$ singular values of $W$ and $\hat W$ on the diagonal, respectively. We have $\|W^\dagger - \hat W_r^\dagger\|_2 =\N{U_1 \Sigma_1^{-1} V_1^\top - \hat U_1 \hat \Sigma_1^{-1} \hat V_1^\top}_2$. Then
	\begin{align*}
		U_1 \Sigma_1^{-1} V_1^\top - \hat U_1 \hat \Sigma_1^{-1} \hat V_1^\top
		&= (U_1 U_1^{\top} - \hat U_1 \hat U_1^{\top}) U_1 \Sigma_1^{-1} V_1^\top + \hat U_1 \hat U_1^{\top} (U_1 \Sigma_1^{-1} V_1^\top - \hat U_1 \hat \Sigma_1^{-1} \hat V_1^\top)\\[1pt]
		&= (U_1 U_1^{\top} - \hat U_1 \hat U_1^{\top}) U_1 \Sigma_1^{-1} V_1^\top + \hat U_1 \hat \Sigma_1^{-1} \hat V_1^\top ( \hat W^\top U_1 \Sigma_1^{-1} V_1^\top - \hat V_1 \hat V_1^\top),\\[0.15pt]
	\end{align*}
	where we used $ U_1 \hat U_1^{\top} \hat U_1 \hat \Sigma_1^{-1} \hat V_1^\top = \hat U_1 \hat \Sigma_1^{-1} \hat V_1^\top = \hat U_1 \hat \Sigma_1^{-1} \hat V_1^\top \hat V_1 \hat V_1^\top$ and \\[0.15pt]
	\begin{equation*}
		\hat U_1 \hat \Sigma_1^{-1} \hat V_1^\top W^\top = \hat U_1 \hat \Sigma_1^{-1} \hat V_1^\top (\hat V_1 \hat \Sigma_1 \hat U_1^\top + \hat V_2 \hat \Sigma_2 \hat U_2^\top ) = \hat U_1 \hat U_1^\top.
	\end{equation*}
	Furthermore,
	\begin{align*}
		\hat W^\top U_1 \Sigma_1^{-1} V_1^\top - \hat V_1 \hat V_1^\top
		&= \hat W^\top U_1 \Sigma_1^{-1} V_1^\top - V_1 V_1^\top + (V_1 V_1^\top - \hat V_1 \hat V_1^\top)\\
		&= (\hat W^\top - V_1 \Sigma_1 U_1^\top) U_1 \Sigma_1^{-1} V_1^\top + (V_1 V_1^\top - \hat V_1 \hat V_1^\top)\\
		&= (\hat W^\top - W^\top) U_1 \Sigma_1^{-1} V_1^\top + (V_1 V_1^\top - \hat V_1 \hat V_1^\top).
	\end{align*}
	We then have
	\begin{align*}
		\|W^\dagger - \hat W_r^\dagger\|_2
		&=\N{U_1 \Sigma_1^{-1} V_1^\top - \hat U_1 \hat \Sigma_1^{-1} \hat V_1^\top}_2\\
		&=  \big\|(U_1 U_1^\top - \hat U_1 \hat U_1^\top) U_1 \Sigma_1^{-1} V_1^\top + \hat U_1 \hat \Sigma_1^{-1} \hat V_1^\top (\hat W^\top - W^\top) U_1 \Sigma_1^{-1} V_1^\top \\
		&\hspace{200pt} + \hat U_1 \hat \Sigma_1^{-1} \hat V_1^\top(V_1 V_1^\top - \hat V_1 \hat V_1^\top)\big\|_2 \\
		&\le \mathscalebox{.965}{\N{U_1 U_1^\top - \hat U_1 \hat U_1^\top}_2 \N{U_1 \Sigma_1^{-1} V_1^\top}_2 + \N{\hat U_1 \hat \Sigma_1^{-1} \hat V_1^\top}_2 \N{\hat W^\top - W^\top}_2 \N{U_1 \Sigma_1^{-1} V_1^\top}_2} \\
		&\hspace{195pt} \mathscalebox{.965}{+ \N{\hat U_1 \hat \Sigma_1^{-1} \hat V_1^\top}_2 \N{V_1 V_1^\top - \hat V_1 \hat V_1^\top}_2} \\
		&\le \frac{\Delta_W}{\sigma_r(W) - \Delta_W} \left(\frac{2}{\sigma_r(W)} + \frac{1}{\sigma_r(W) - \Delta_W}\right).
	\end{align*}
	The last line uses Wedin's bound and $\textN{\hat U_1 \hat \Sigma_1^{-1} \hat V_1^\top}_2 = \frac1{\sigma_r(\hat W)}\le \frac1{\sigma_r(W) - \Delta_W}$, which comes from Weyl's inequality.
\end{proof}

\smallskip
\smallskip

Next we control the propagation of error on the weights and intermediate subspaces in Algorithm~\ref{alg:modified_SPM}.
\begin{proposition}\label{lem:deflation_error_bounds}
	Define $T, \hat T$ as in Algorithm \ref{alg:modified_SPM},
	with $M:=\opmyreshape(T, [D^{n}, D^{m-n}])$ and its singular values $\sigma_1(M) \geq \ldots \geq \sigma_K(M) > 0$.
	Let $\hat M := \opmyreshape(\hat T, [D^{n}, D^{m-n}])$,
	and suppose that $\Delta_M := \| M - \hat M \|_2 < \sigma_K(M)$.
	Then the error for $\hat \lambda_k$ obtained by Algorithm \ref{alg:modified_SPM} is bounded as follows:
	\begin{equation}\label{eq:lambda_err_bound}
		\SN{\frac1{\lambda_k} - \frac1{\hat \lambda_k}} \le \frac1{\sigma_K(M)}\sqrt{2m} \|a_k - \hat a_k\| +  \frac{\Delta_M}{\sigma_K(M) - \Delta_M} \left(\frac{2}{\sigma_K(M)} + \frac{1}{\sigma_K(M) - \Delta_M}\right).
	\end{equation}
	Additionally, let $\mathcal{A}_{k}=\Span{a_i^{\otimes n}, i\in \{k,\dots, K\}}$, $A_{[k-1]}$ the submatrix of $A$ with columns in $[k-1]$, $G_{d,[k-1]} = (A_{[k-1]}^{\bullet s})^{\top} (A_{[k-1]}^{\bullet k}) = (A_{[k-1]}^\top A_{[k-1]})^{\odot k}$ and $\varphi_{d,k-1} = \sqrt{\|G_{d, [k-1]}\|_2}$. Then the error of the deflated subspace is bounded by
	\begin{align}
		\nonumber \N{P_{\mathcal{A}_{k}} -  P_{\hat{\mathcal{A}}_{k}}}_2
		&\le \frac{\Delta_M}{\sigma_K(M) - \Delta_M}\left(1+  \max_{i\in [K]} |\lambda_i| \sqrt{\|G_{n}\|_2 \|G_{m-n}\|_2}  \left(\mathscalebox{.8}{\frac{2}{\sigma_K(M)}} + \mathscalebox{.8}{\frac{1}{\sigma_K(M) - \Delta_M}}\right)\right)  \\
		\label{eq:subspace_err_bound}
		&\hspace{35pt} + \max_{i\in [K]}  |\lambda_i| \sqrt{\|G_{n}\|_2} \frac{1}{\sigma_K(M) - \Delta_M} \sqrt{(m-n) \sum_{i=1}^{k-1} \|a_s - \hat a_s\|^2}.
	\end{align}
\end{proposition}

\begin{proof}
	We start by showing that \eqref{eq:lambda_formula} holds in the clean case, that is,
	\begin{equation*}
		\lambda_k = \frac{1}{\opvec(a_k^{\otimes n})^\top (M^\top)^\dagger \opvec(a_k^{\otimes m-n})}.
	\end{equation*}
	We have
	$$M = \sum_{i=1}^K \lambda_i \opvec(a_i^{\otimes n}) \opvec(a_i^{\otimes m-n})^\top = A^{\bullet n} \Lambda (A^{\bullet m-n})^\top,$$
	where $\Lambda = \diag(\lambda_1, \dots, \lambda_K)$. Therefore, $(M^\top)^\dagger = ((A^{\bullet n})^\top)^\dagger \Lambda^{-1}(A^{\bullet m-n})^\dagger$ and
	\begin{equation}\label{eq:pseudo_inv_lambda}
		(A^{\bullet n})^\top (M^\top)^\dagger A^{\bullet m-n} = \Lambda^{-1}.
	\end{equation}
	Denoting by $e_k$ the $k$-th canonical basis vector, we have $\opvec(a_k^{\otimes n}) = (A^{\bullet n}) e_k$, thus $${\opvec(a_k^{\otimes n})^\top (M^\top)^\dagger \opvec(a_k^{\otimes m-n})} = e_k^{\top} \left((A^{\bullet n})^\top (M^\top)^\dagger A^{\bullet m-n} \right) e_k = e_k^{\top} \Lambda^{-1} e_k = \frac1{\lambda_k}.$$
	Then
	\begin{align*}
		\SN{\frac1{\hat \lambda_k} - \frac1{\lambda_k}}&=\SN{\opvec(\hat a_k^{\otimes n})^\top (\hat M_K^\top)^\dagger \opvec(\hat a_k^{\otimes m-n}) - \opvec(a_k^{\otimes n})^\top (M^\top)^\dagger \opvec(a_k^{\otimes m-n})}\\
		&\le \SN{\opvec(\hat a_k^{\otimes n})^\top ((\hat M_K^\top)^\dagger - (M^\top)^\dagger) \opvec(\hat a_k^{\otimes m-n})} \\
		&\hspace{70pt}+ \SN{\opvec(\hat a_k^{\otimes n})^\top(M^\top)^\dagger (\opvec(\hat a_k^{\otimes m-n}) - \opvec(a_k^{\otimes m-n}))}\\
		&\hspace{140pt}+ \SN{(\opvec(\hat a_k^{\otimes n})-\opvec(a_k^{\otimes n}))^\top (M^\top)^\dagger \opvec(a_k^{\otimes m-n})}
		\\
		&\le \N{\hat M_K^\dagger - M^\dagger}_2 + \|M^\dagger\|_2\left(\|a_i^{\otimes n} - \hat a_i^{\otimes n}\|_F + \|a_i^{\otimes m-n} - \hat a_i^{\otimes m-n}\|_F\right).
	\end{align*}
	The first term is bounded using Lemma \ref{lem:pseudo-inverse-bound}, while the bound for the second term follows from $\|M^\dagger\|_2 = \frac1{\sigma_k(M)}$ and
	\begin{align}
		\nonumber
		\|x^{\otimes d} - y^{\otimes d}\|^2 &= 2 - 2 (x^\top y)^d =
		2 - 2 \left(1-\frac12\|x-y\|^2\right)^d\\
		&\le 2 - 2 + d\|x-y\|^2 = d\|x-y\|^2,
		\label{eq:d_norm_bound}
	\end{align}
	where we used $(1-x)^d\ge 1 - d x$ for all $x\le 1$ and $d\ge 1$. Using AM -- GM, we obtain
	\begin{align*}
		\|a_i^{\otimes n} - \hat a_i^{\otimes n}\|_F + \|a_i^{\otimes m-n} - \hat a_i^{\otimes m-n}\|_F\le \left(\sqrt{m-n} + \sqrt{n}\right)\|a_i - \hat a_i\|\le \sqrt{2m}\|a_i - \hat a_i\|,
	\end{align*}
	and the proof of \eqref{eq:lambda_formula} is complete.
	Regarding, \eqref{eq:subspace_err_bound}, we have that
	\begin{equation*}
		\hat \CA_k = \hat{ \mathcal{A}} \cap \bigcap_{i=1}^{k-1} ( (\hat M_K^\top)^\dagger \opvec(\hat a_i^{\otimes m-n}))^\perp = \hat{ \mathcal{A}} \cap \operatorname{colspan}( (\hat M_K^\top)^\dagger \hat A_{[k-1]}^{\bullet m-n})^\perp
	\end{equation*}
	and
	\begin{equation}\label{eq:hatAk_contained}
		\hat{\mathcal{A}} = \colspan(\hat M_K) = \colspan((\hat M_K^\top)^\dagger) \supseteq \colspan((\hat M_K^\top)^\dagger \hat A_{[k-1]}^{\bullet m-n}).
	\end{equation}
	We first show that if $\hat T = T$, $\hat \CA_k = \CA_k$, that is,
	$$\CA \cap \operatorname{colspan}( (M^\top)^\dagger A_{[k-1]}^{\bullet m-n})^\perp = \Span{a_i^{\otimes n}, i\in \{k,\dots, K\}}.$$
	It is enough to show that $\dim(\hat \CA_k) = K-k+1 = \dim(\CA_k)$ and that $\CA_k \subset \hat \CA_k$.
	We have that $\dim(\hat \CA) = K$, $\dim(\colspan((\hat M_K^\top)^\dagger A_{[k-1]}^{\bullet m-n})) = k-1$, thus \eqref{eq:hatAk_contained} implies that $\dim(\hat \CA_k) = K - k +1$. To show that $\hat \CA_k \subset \CA \cap \operatorname{colspan}( (M^\top)^\dagger A_{[k-1]}^{\bullet n})^\perp$, note that,
	$$\CA_k = \Span{a_i^{\otimes n}, i\in \{k,\dots, K\}} = \colspan(A_{[k:K]}^{\bullet n}) \subset \colspan(A^{\bullet m-n}) = \CA,$$
	where $[k:K] = \{k,\dots, K\}$.
	On other hand, $\colspan(A_{[k:K]}^{\bullet n}) \subset ( (M^\top)^\dagger A_{[k-1]}^{\bullet m-n})^\perp$ is equivalent to, for all $i\in [k:K], j\in [k-1]$,
	$\opvec(a_i^{\otimes n})^\top (M^\top)^\dagger \opvec(a_k^{\otimes m-n})=0$. In fact, \eqref{eq:pseudo_inv_lambda} implies that, $\opvec(a_i^{\otimes n})^\top (M^\top)^\dagger \opvec(a_k^{\otimes m-n})=(\Lambda^{-1})_{ij} = 0$ for all  $i\in [k:K], j\in [k-1]$, as we wanted to show.

	Let $U=(U_1, U_2), \hat U=(\hat U_1, \hat U_2)$ be orthogonal matrices such that the columns of $U_1$ and $\hat U_1$ are orthonormal basis for $\CA$ and $\hat \CA$, respectively. Moreover, since $\colspan((\hat M_K^\top)^\dagger \hat A_{[k-1]}^{\bullet m-n})\subset \hat \CA$, there exist orthogonal matrices $O = (O_1, O_2)$ and $\hat O = (\hat O_1, \hat O_2)$ such that $\colspan(U_1 O_1) = \colspan((M^\top)^\dagger A_{[k-1]}^{\bullet m-n})$ and $\colspan(\hat U_1 \hat O_1) = \colspan((\hat M_K^\top)^\dagger \hat A_{[k-1]}^{\bullet m-n})$, respectively.
	We then have $\CA_k = \colspan(U_1 O_2)$ $\hat \CA_k = \colspan(\hat U_1 \hat O_2)$, thus:
	\begin{align*}
		\N{P_{\CA_k} - P_{\hat \CA_k}}_2 &= \N{U_1 O_2 O_2^\top U_1^\top - \hat U_1 \hat O_2 \hat O_2^\top \hat U_1^\top}_2\\
		&= \N{U_1 U_1^\top - \hat U_1 \hat U_1^\top - U_1 O_1 O_1^\top U_1^\top - \hat U_1 \hat O_1 \hat O_1^\top \hat U_1^\top}_2\\
		&\le \N{U_1 U_1^\top - \hat U_1 \hat U_1^\top}_2 + \N{U_1 O_1 O_1^\top U_1^\top - \hat U_1 \hat O_1 \hat O_1^\top \hat U_1^\top}_2.
	\end{align*}
	The first term is bounded by Lemma~\ref{lem:error_in_subspace}, while the second term is bounded using \cite[Theorem 2.5]{chen2016perturbation}:
	\begin{align*}
		\N{U_1 O_1 O_1^\top U_1^\top - \hat U_1 \hat O_1 \hat O_1^\top \hat U_1^\top}_2
		&\le \N{\left((M^\top)^\dagger A_{[k-1]}^{\bullet m-n} - (\hat M_K^\top)^\dagger \hat A_{[k-1]}^{\bullet m-n}\right)\left((M^\top)^\dagger A_{[k-1]}^{\bullet m-n}\right)^{\dagger}}_2\\
		&\le \N{(M^\top)^\dagger A_{[k-1]}^{\bullet m-n} - (\hat M_K^\top)^\dagger \hat A_{[k-1]}^{\bullet m-n}}_2\N{\left((M^\top)^\dagger A_{[k-1]}^{\bullet m-n}\right)^{\dagger}}_2.
	\end{align*}
	Let $\eta_1 \ge \dots \ge \eta_K$ be the singular values of $(M^\top)^\dagger A^{\bullet m-n}$, and $\nu_1\ge\dots\ge \nu_{k-1}$ the singular values of $(M^\top)^\dagger A_{[k-1]}^{\bullet m-n}$. We have $\N{\left((M^\top)^\dagger A^{\bullet m-n}\right)^{\dagger}}_2 = \frac1{\eta_K}$ and $\N{\left((M^\top)^\dagger A_{[k-1]}^{\bullet m-n}\right)^{\dagger}}_2 = \frac1{\nu_{k-1}}$. Moreover, by \cite[Theorem 1]{thompson1972principal}, we have $\nu_{k-1}\ge \eta_K$, thus $\N{\left((M^\top)^\dagger A_{[k-1]}^{\bullet m-n}\right)^{\dagger}}_2 \le \N{\left((M^\top)^\dagger A^{\bullet m-n}\right)^{\dagger}}_2$.
	Furthermore \eqref{eq:pseudo_inv_lambda} implies that $\left((M^\top)^\dagger A^{\bullet m-n}\right)^{\dagger} = \Lambda (A^{\bullet n})^{\top}$, therefore
	$$\N{\left((M^\top)^\dagger A_{[k-1]}^{\bullet m-n}\right)^{\dagger}}_2 \le \N{\Lambda (A^{\bullet n})^{\top}}_2 \le  \max_{i\in [K]}  |\lambda_i| \sqrt{\|G_{n}\|_2}.$$
	On the other hand, we have
	\begin{align*}
		\hspace{90pt}&\hspace{-90pt}\N{(M^\top)^\dagger A_{[k-1]}^{\bullet m-n} - (\hat M_K^\top)^\dagger \hat A_{[k-1]}^{\bullet m-n}}_2 \\
		&\le \N{\left((M^\top)^\dagger - (\hat M_K^\top)^\dagger\right) A_{[k-1]}^{\bullet m-n}}_2 + \N{(\hat M_K^\top)^\dagger \left(A_{[k-1]}^{\bullet m-n} - \hat A_{[k-1]}^{\bullet m-n}\right)}_2\\
		&\le \N{(M^\top)^\dagger - (\hat M_K^\top)^\dagger}_2 \N{A_{[k-1]}^{\bullet m-n}}_2 + \N{(\hat M_K^\top)^\dagger}_2 \N{A_{[k-1]}^{\bullet m-n} - \hat A_{[k-1]}^{\bullet m-n}}_2.
	\end{align*}

	The term $\N{(M^\top)^\dagger - (\hat M_K^\top)^\dagger}_2$ is bounded by Lemma~\ref{lem:pseudo-inverse-bound}, $$\N{A_{[k-1]}^{\bullet m-n}}_2 \le \N{A^{\bullet m-n}}_2 = \sqrt{\|G_{m-n}\|_2},$$ and $\N{(\hat M_K^\top)^\dagger}_2 =  \frac1{\sigma_K(\hat M_K )} \le \frac1{\sigma_K(M) - \Delta_{M}}$. Finally,
	\begin{align*}
		\N{A_{[k-1]}^{\bullet m-n} - \hat A_{[k-1]}^{\bullet m-n}}_2
		&\le \N{A_{[k-1]}^{\bullet m-n} - \hat A_{[k-1]}^{\bullet m-n}}_F
		= \sqrt{\sum_{i=1}^{k-1} \N{a_i^{\otimes m-n} - \hat a_i^{\otimes m-n}}_F^2}\\
		&\le \sqrt{(m-n)\sum_{i=1}^{k-1} \N{a_i - \hat a_i}_2^2},
	\end{align*}
	where we used \eqref{eq:d_norm_bound} in the last line.
\end{proof}

We can now prove our guarantee for end-to-end symmetric tensor decomposition using SPM.

\begin{proof}[Proof of Theorem~\ref{thm:SPM_deflation_analysis}]
	First, to show that Algorithm \ref{alg:modified_SPM} picks the correct tensor rank $K$, we use the assumptions of Theorem~\ref{thm:SPM_deflation_analysis} and Weyl's inequality to obtain
	$$\sigma_K(\hat M) \ge \sigma_K(M) - \Delta_M > \alpha > \Delta_{M} \ge \sigma_{K+1}(\hat M).$$
	Then, we show the bound on $\|\hat a_k - a_k\|$ using induction on $k$. When $k=1$, we have $\hat\CA_{1} = \hat\CA$. Then, Lemma~\ref{lem:error_in_subspace} implies that $\|P_{\hat\CA} - P_{\CA}\| \le \hat \Delta_{\CA}$. Since $\hat a_1$ was returned by the \textsc{Power Method}, it is a second order critical point of $F_{\hat\CA}$, and Algorithm \ref{alg:modified_SPM} asserts that $F_{\hat\CA}(\hat a_i) > \tau \ge \ell(\tilde \Delta_{\CA}) \ge \ell(\hat \Delta_{\CA})$. Therefore $\hat a_1$ is a second order critical point in the level set stated in Theorem \ref{thm:main_result_deterministic} / \ref{thm:main_result_for_now}, which implies that there exists $\pi(1):=i\in [K]$ and $s_1\in \{-1,1\}$ such that $\|\hat a_1- s_1 a_{\pi(1)}\|\le \sqrt{\tfrac{2\hat \Delta_{\CA}}{n}}$. If $k>1$, let $\pi([k-1])=\{\pi(i), i\in [k-1]\}$ and $\mathcal{A}_{k} = \Span{a_i^{\otimes n}, i\in [K]\backslash \pi([k-1])}$. Since the subspace deflation step in Algorithm \ref{alg:modified_SPM} does not depend on the sign of $\hat a_i$, $i\in [k-1]$, we can flip the sign in the bound provided in Proposition~\ref{lem:deflation_error_bounds}
	\begin{align*}
		\N{P_{\mathcal{A}_{k}} - P_{\hat{\mathcal{A}}_{k}}}_2
		&\le \frac{\Delta_M}{\sigma_K(M) - \Delta_M}\left(1+  \max_{i\in [K]} |\lambda_i| \sqrt{\|G_{n}\|_2 \|G_{m-n}\|_2}  \left(\mathscalebox{.81}{\frac{2}{\sigma_K(M)}} + \mathscalebox{.81}{\frac{1}{\sigma_K(M) - \Delta_M}}\right)\right)  \\
		&\hspace{50pt} + \max_{i\in [K]}  |\lambda_i| \sqrt{\|G_{n}\|_2} \frac{1}{\sigma_K(M) - \Delta_M} \sqrt{(m-n) \sum_{i=1}^{k-1} \|a_{\pi(i)} - s_i \hat a_i\|^2} \\
		&\le \hat \Delta_{\CA}\left(1+  \max_{i\in [K]} |\lambda_i| \sqrt{\|G_{n}\|_2 \|G_{m-n}\|_2}  \frac{4}{\sigma_K(M)}\right)  \\
		&\hspace{50pt} + \max_{i\in [K]}  |\lambda_i| \sqrt{\|G_{n}\|_2} \frac{2}{\sigma_K(M)} \sqrt{\frac{2(m-n)(k-1)}{n}} \sqrt{\hat \Delta_{\CA} }\\
		&\le \hat \Delta_{\CA}\left(1+  \max_{i\in [K]} |\lambda_i| \sqrt{\|G_{n}\|_2 \|G_{m-n}\|_2}  \frac{4}{\sigma_K(M)}\right)  \\
		&\hspace{50pt} + \max_{i\in [K]}  |\lambda_i| \sqrt{\|G_{n}\|_2} \frac{2}{\sigma_K(M)} \sqrt{\frac{2(m-n)K}{n}} \sqrt{\hat \Delta_{\CA} }\\
		&= \tilde \Delta_{\CA},
	\end{align*}
	where we set
	\begin{equation}
    \label{eq:defl_C1}
      C_1 = 1+  \max_{i\in [K]} |\lambda_i| \sqrt{\|G_{n}\|_2 \|G_{m-n}\|_2}  \frac{4}{\sigma_K(M)}
  \end{equation}
	and
	\begin{equation}
    \label{eq:defl_C2}
    C_2 =  \max_{i\in [K]}  |\lambda_i| \sqrt{\|G_{n}\|_2} \frac{2}{\sigma_K(M)} \sqrt{\frac{2(m-n)K}{n}}.
  \end{equation}

	Then, since $F_{\hat\CA_k}(\tilde a_k) > \tau \ge \ell(\tilde \Delta_{\CA})$, and $\tilde a_k$ is a second order critical point of $\hat\CA_k$, Theorem \ref{thm:main_result_deterministic} / \ref{thm:main_result_for_now} implies that there exists $\pi(k)~\in~[K]\backslash \pi([k~-~1]$ and $s_k \in \{-1,1\}$, such that $\|\hat a_k- s_k a_{\pi(k)}\|\le \sqrt{\tfrac{2\tilde  \Delta_{\CA}}{n}}$. In particular, it is shown in the proof of both theorems that $\hat a_k$ lies in a spherical cap centered around $s_k a_{\pi(k)}$ where $F_{\hat\CA_k}$ is concave. However, in both theorem statements, the radius of the spherical cap where concavity holds is a decreasing function of $\Delta_{\CA}$, therefore $\hat a_k$ is also in a spherical cap centered around $s_k a_{\pi(k)}$ where $F_{\hat\CA}$ is concave. Applying the \textsc{Power method} with the functional $F_{\hat\CA}$, using $\tilde a_i$ as a starting point will then converge for the local maxima in this concave region, and the bound for $\|\hat a_k - s_k a_{\pi(k)}\|$ follows.

	Finally, the error bound for $\hat \lambda_k$ follows from the bound on $\|\hat a_k-s_k a_{\pi(k)}\|$ and Proposition~\ref{lem:deflation_error_bounds}.
\end{proof}

We conclude with a bound on the error in terms of the scalar coefficients $\lambda_i$, which follows from a bound on $\sigma_K(M)$.

\begin{lemma}\label{lem:sigmak(M)_bound}
	With $\sigma_K(M)$ defined and under the same conditions of Theorem~\ref{thm:SPM_deflation_analysis}, it holds
	\begin{equation}
		\sigma_K(M) \ge \frac{\min_i |\lambda_i|} {\sqrt{\|G_n^{-1}\|_2 \|G_{m-n}^{-1}\|_2}}.
	\end{equation}
\end{lemma}
\begin{proof}
	Let $\mu_j(A)$ denote the $j$-th algorithm (in descending order) of $A$, that is $\mu_1(A) \ge \mu_2(A) \ge \cdots$. We first show that for all matrices $A_{n\times m}$ and $B_{n\times n}$ such that $m\ge n$ and $B$ is symmetric and positive definite, we have
	$$\mu_n(A^\top B A) \ge \mu_n(B) \mu_n(A^\top A)$$
	In fact, we have $A^\top B A = \mu_n(B) A^\top A + A^\top (B- \mu_n(B) I)A$, therefore by Weyl's inequality \cite{weyl1912asymptotische}: 
	$$\mu_n(A^\top B A) \ge \mu_n(B) \mu_n(A^\top A) + \mu_n(A^\top (B- \mu_n(B) I)A) \ge \mu_n(B) \mu_n(A^\top A).$$
	Applying this twice, we have
	\begin{align*}
		\sigma_K(M) &= \sigma_K\left(A^{\bullet n} \Lambda (A^{\bullet m-n})^\top\right) = \sqrt{\mu_K\left(A^{\bullet n} \Lambda (A^{\bullet m-n})^\top A^{\bullet m-n}\Lambda (A^{\bullet n})^\top \right)},\\
		&\ge \sqrt{\mu_K(G_{m-n})} \sqrt{\mu_K\left(A^{\bullet n} \Lambda^2  (A^{\bullet n})^\top \right)},\\
		&\ge \min_i |\lambda_i| \sqrt{\mu_K(G_{m-n})} \sqrt{\mu_K\left(A^{\bullet n} (A^{\bullet n})^\top \right)} = \min_i |\lambda_i| \sqrt{\sigma_K(G_n) \sigma_K(G_{m-n})}.
	\end{align*}
	Now the lemma follows from $\sigma_K(G_n) = 1/\|G_n^{-1}\|_2$.
\end{proof}

\smallskip

\begin{remark}[Small $\lambda_i$] \label{rem:small-lambda}
 Lemma \ref{lem:sigmak(M)_bound} indicates that the smaller $\min_{i} |\lambda_i|$ is, the larger is the error of the decomposition obtained by SPM. For instance, the lemma can be used to obtain the bound 
$$\hat \Delta_{\CA} \le \frac{\Delta_M \sqrt{\|G_n^{-1}\|_2 \|G_{m-n}^{-1}\|_2}}{\min_{i} |\lambda_i| - \Delta_M \sqrt{\|G_n^{-1}\|_2 \|G_{m-n}^{-1}\|_2}},$$
which suggests  a smaller $\min_i |\lambda_i|$ increases the subspace error $\Delta_{\CA}$.
This is also evident in the error arising from deflation. Substituting Lemma \ref{lem:sigmak(M)_bound} into the definition of $C_1$ and $C_2$ in Theorem~\ref{thm:SPM_deflation_analysis}, 
\begin{equation*}
	C_1 \le 1+ 4 \frac{\max_{i} |\lambda_i|}{\min_{i} |\lambda_i|}\sqrt{\kappa(G_{n}) \kappa(G_{m-n})},
\end{equation*}
and
\begin{equation*}
	C_2 \le  2  \frac{\max_{i} |\lambda_i|}{\min_{i} |\lambda_i|} \sqrt{\kappa(G_{n}) \N{G_{m-n}^{-1}}_2} \sqrt{\frac{2(m-n)K}{n}}.
\end{equation*}
Here we denoted the condition number of a matrix $A$ by $\kappa(A):= \N{A}_2 \N{A^{-1}}_2$.

Finally, the issue also shows up in the bound of the tensor coefficient error of Theorem~\ref{thm:SPM_deflation_analysis}:

$$\SN{\frac{s_k^m}{\lambda_{\pi(k)}} - \frac1{\hat \lambda_k}} \le \frac {\sqrt{\|G_n^{-1}\|_2 \|G_{m-n}^{-1}\|_2}}{\min_i |\lambda_i|}
\left(2\sqrt{m/n\hat \Delta_{\CA}}  +  4\hat \Delta_{\CA}\right).$$
\end{remark}

\end{document}